\documentclass{amsart}
\usepackage{amssymb}

\newcommand{\map}[1]{\xrightarrow{#1}}
\newcommand{\iso}{\cong}
\newcommand{\define}{\stackrel{\mathrm{def}}{=}}

\newcommand{\Gal}{\mathrm{Gal}}
\newcommand{\Hom}{\mathrm{Hom}}
\newcommand{\Aut}{\mathrm{Aut}}
\newcommand{\End}{\mathrm{End}}
\newcommand{\Spec}{\mathrm{Spec}}
\newcommand{\Spf}{\mathrm{Spf}}
\newcommand{\Q}{\mathbb Q}
\newcommand{\Z}{\mathbb Z}
\newcommand{\R}{\mathbb R}
\newcommand{\C}{\mathbb C}
\newcommand{\F}{\mathbb F}
\newcommand{\A}{\mathbb A}
\newcommand{\co}{\mathcal O}
\newcommand{\alg}{\mathrm{alg}}
\newcommand{\ord}{\mathrm{ord}}

\newcommand{\Tor}{\mathrm{Tor}}
\newcommand{\length}{\mathrm{length}}
\newcommand{\horizontal}{\mathrm{hor}}
\newcommand{\vertical}{\mathrm{ver} }
\newcommand{\good}{\bullet}
\newcommand{\bad}{{\bullet\bullet}}
\newcommand{\proper}{\mathrm{prop}}
\newcommand{\sing}{\mathrm{sing}}
\newcommand{\nonsing}{\mathrm{nsing}}

\input xy
\xyoption{all}

\begin{document}

\author{Benjamin Howard}
\address{Department of Mathematics, Boston College, Chestnut Hill, MA, 02467}
\email{howardbe@bc.edu}
\title{Intersection theory on Shimura surfaces II}


\begin{abstract}
This is the third  of a series of papers relating intersections of special cycles on  the integral model of a Shimura surface 
to Fourier coefficients of Hilbert modular forms.  More precisely, we embed the Shimura curve  over $\Q$ associated to a 
rational quaternion algebra into the Shimura surface associated to the base change of the quaternion algebra to a real 
quadratic field.  After extending the associated moduli problems  over $\mathbb{Z}$ we obtain an arithmetic threefold 
with a embedded arithmetic surface, which we view as a cycle of codimension one.  We then construct a family,  
indexed by totally positive algebraic integers in the real quadratic field,  of codimension two cycles (complex 
multiplication points) on the arithmetic threefold.   The intersection multiplicities of the codimension two cycles with the 
fixed codimension one cycle are shown to agree with the Fourier coefficients of a (very particular) Hilbert modular form 
of weight $3/2$.  The results are higher dimensional variants of results of Kudla-Rapoport-Yang, which relate 
intersection multiplicities of special cycles on the integral model of a Shimura curve to Fourier coefficients of a modular 
form in two variables.  
\end{abstract}

\maketitle

\theoremstyle{plain}
\newtheorem{Thm}{Theorem}[section]
\newtheorem{Prop}[Thm]{Proposition}
\newtheorem{Lem}[Thm]{Lemma}
\newtheorem{Cor}[Thm]{Corollary}
\newtheorem{Conj}[Thm]{Conjecture}
\newtheorem{BigThm}{Theorem}

\theoremstyle{definition}
\newtheorem{Def}[Thm]{Definition}
\newtheorem{Hyp}[Thm]{Hypothesis}

\theoremstyle{remark}
\newtheorem{Rem}[Thm]{Remark}
\newtheorem{Ques}[Thm]{Question}

 \numberwithin{equation}{section}
\renewcommand{\labelenumi}{(\alph{enumi})}
\renewcommand{\theBigThm}{\Alph{BigThm}}

\section{Introduction}

Let $F\subset \R$ be a real quadratic field of discriminant $d_F$, and fix a $\Z$-basis 
$\{\varpi_1,\varpi_2\}$ of $\co_F$.  Let $B_0$ be a quaternion division algebra over $\Q$ satisfying
\begin{enumerate}
\item
$B_0\otimes_\Q\R\iso M_2(\R)$ (fix one such isomorphism once and for all),
\item
every prime divisor of $\mathrm{disc}(B_0)$ splits in $F$,
\end{enumerate}
and set $B=B_0\otimes_{\Q}F$.  Choose a maximal order $\co_{B_0}$ of $B_0$ that is stable under the main 
involution $b\mapsto b^\iota$,   and set $\co_B=\co_{B_0}\otimes_\Z\co_F$.  The hypothesis that all primes 
ramified in $B_0$ split in $F$ implies that $\co_B$ is a maximal order of $B$.  We consider two functors from 
the category of $\Z$-schemes to the category of groupoids (recall that a groupoid is a category in which 
all arrows are isomorphisms).   The first, $\mathcal{M}_0$, associates to a 
$\Z$-scheme $S$ the category of abelian schemes over $S$ of relative dimension two equipped with an 
action of $\co_{B_0}$. The second, $\mathcal{M}$, associates to a $\Z$-scheme $S$ the category of abelian 
schemes over $S$ of relative dimension four equipped with an action of $\co_{B}$ (more precise 
definitions of these moduli problems are in \S \ref{s:moduli}; for the purposes of this introduction we omit some 
of the moduli data, \emph{e.g.}~ polarizations).  The moduli problems  $\mathcal{M}_0$ and $\mathcal{M}$ are 
representable by projective, regular Deligne-Mumford stacks of relative dimensions one and two, respectively, 
over $\Spec(\Z)$,  whose complex fibers are well-known from the theory of Shimura varieties.    
To describe these complex fibers define algebraic groups $G_0\subset G$ over $\Q$ by
$$
G_0(A) = (B_0\otimes_\Q A)^\times
$$
and 
$$
G(A)= \{ x\in (B\otimes_\Q A)^\times : \mathrm{Nm}(x)\in A^\times \}
$$
for any $\Q$-algebra $A$.  Here $\mathrm{Nm}$ is the reduced norm on $B^\times$.  Define discrete subgroups
$\Gamma_0= \co_{B_0}^\times \subset G_0(\R)$ and 
$$
\Gamma = \{ x\in \co_B^\times: \mathrm{Nm}(x) \in \Z^\times \} \subset G(\R).
$$
Let $X_0=\mathbb{P}^1(\C)\smallsetminus \mathbb{P}^1(\R)$, and set 
 $$X=(X_0^+\times X_0^+)\cup (X_0^-\times X_0^-)$$
 where $X_0^+$ and $X_0^-$ are the connected components of $X_0$.   
If we identify  $G_0(\R)$ with  $\mathrm{GL}_2(\R)$, and  identify $G(\R)$ with a subgroup of 
$\mathrm{GL}_2(\R)\times\mathrm{GL}_2(\R)$, there are isomorphisms of complex orbifolds
$$
\mathcal{M}_0(\C)\iso [\Gamma_0\backslash X_0] \qquad \mathcal{M}(\C)\iso [\Gamma\backslash X].
$$

For an abelian scheme $A_0$ over an arbitrary base scheme $S$ there is an abelian scheme $A_0\otimes\co_F$ 
whose functor of points satisfies $(A_0\otimes\co_F)(T) \iso A_0(T)\otimes_\Z\co_F$ for any $S$-scheme 
$T$ \cite[\S 3.1]{howardA}.
There is a canonical closed immersion $\mathcal{M}_0\map{}\mathcal{M}$, which is given on moduli by 
$A_0\mapsto A_0\otimes\co_F$, and which on complex fibers is induced by the diagonal embedding $X_0\map{}X$.    
This closed immersion induces a linear functional (defined in \S \ref{s:moduli})
\begin{equation}\label{relative degree}
\widehat{\deg}_{\mathcal{M}_0}:\widehat{\mathrm{CH}}^2(\mathcal{M}) \map{} \R
\end{equation}
on the codimension two Gillet-Soul\'e arithmetic Chow group of $\mathcal{M}$ (with rational coefficients) 
 called the \emph{arithmetic degree along} $\mathcal{M}_0$.   In an earlier work \cite{howardA} the author 
 constructed an  arithmetic cycle class
\begin{equation}\label{my class}
\widehat{\mathcal{Y}}(\alpha,v)\in \widehat{\mathrm{CH}}^2(\mathcal{M})
\end{equation}
depending on a totally positive $\alpha\in\co_F$ and a totally positive $v\in F\otimes_\Q\R\iso \R\times\R$.  
The construction of this class is based upon another moduli problem, $\mathcal{Y}(\alpha)$, which associates 
to a $\Z$-scheme $S$ the category of abelian schemes over $S$ equipped with an action of
 $\co_{B}$ and a commuting action of $\co_F[\sqrt{-\alpha}]$.  Roughly speaking, $\mathcal{Y}(\alpha)$ is the
  moduli space of points on $\mathcal{M}$ with complex multiplication by the order $\co_F[\sqrt{-\alpha}]$.   
   The evident forgetful map $\mathcal{Y}(\alpha)\map{}\mathcal{M}$ is finite and unramified, allowing one to
    view $\mathcal{Y}(\alpha)$ as a cycle on $\mathcal{M}$.  The cycle 
$$
\mathcal{Y}(\alpha)\times_\Z \Z[1/\mathrm{disc}(B_0)] \map{} \mathcal{M}\times_\Z \Z[1/\mathrm{disc}(B_0)]
$$
has codimension two, and may (depending on $\alpha$) have nonreduced vertical components at 
primes that are nonsplit
 in $F$.  At primes dividing  $\mathrm{disc}(B_0)$, the cycle $\mathcal{Y}(\alpha)$ may have vertical components of 
 codimension one in $\mathcal{M}$.  One of the main results of \cite{howardA} is the construction of certain natural 
 replacements for  these components of excess dimension, and we will briefly recall the essentials of this construction 
 in \S \ref{s:bad reduction}.  The codimension two cycle  on $\mathcal{M}$ underlying the cycle class (\ref{my class})
  is then $\mathcal{Y}(\alpha)$ with the modified vertical components at primes dividing $\mathrm{disc}(B_0)$.    
  The totally positive parameter $v\in F\otimes_\Q\R$ is used in the construction of a Green current for this cycle.

 As the value
\begin{equation}\label{intro deg}
\widehat{\deg}_{\mathcal{M}_0} \widehat{\mathcal{Y}}(\alpha,v)
\end{equation}
is essentially the intersection multiplicity of $\mathcal{Y}(\alpha)$ with $\mathcal{M}_0$ 
one would expect, following the general philosophy of Kudla \cite{kudla02,kudla03,kudla04b} 
and the results of Kudla-Rapoport-Yang \cite{KRY}, that the arithmetic degree (\ref{intro deg}) 
should be related to Fourier coefficients of the derivative of an Eisenstein series.  The main result 
of \cite{howardA} confirms that this is so, at least under the hypothesis that the field extension 
$F(\sqrt{-\alpha})/\Q$ is not biquadratic.  This hypothesis ensures that the cycles $\mathcal{Y}(\alpha)$ 
and $\mathcal{M}_0$  are disjoint in the generic fiber of $\mathcal{M}$, so that the arithmetic 
degree along $\mathcal{M}_0$ is essentially the usual naive intersection multiplicity.   In the present 
work we turn to  the  more difficult case in which $F(\sqrt{-\alpha})/\Q$ is biquadratic; thus we must compute 
intersection multiplicities of  cycles that intersect improperly, in which case the definition of the arithmetic 
degree along $\mathcal{M}_0$ requires the use of Chow's moving lemma \cite{roberts72} on the generic fiber 
of $\mathcal{M}$.  This complicates the picture considerably.

Before stating the main result, we  describe the automorphic form to which (\ref{intro deg}) is to be related.    
To the quadratic space of trace zero elements of $B_0$, Kudla-Rapoport-Yang \cite[(5.1.44)]{KRY}  attach  
an  Eisenstein series $\mathcal{E}_2(\tau,s,B_0)$ of weight $3/2$ on the Siegel half-space of genus two 
$\mathfrak{h}_2$.  This Eisenstein series satisfies a functional equation forcing $\mathcal{E}_2(\tau,0,B_0)=0$, 
and we denote by 
$$
\widehat{\phi}_2(\tau) = \mathcal{E}_2'(\tau,0,B_0)
$$
its derivative at $s=0$.  Let $\mathfrak{h}_1$ be the usual complex upper half-plane.  
The choice of $\Z$-basis $\{\varpi_1,\varpi_2\}$ of $\co_F$ determines an embedding 
$i_F:\mathfrak{h}_1\times\mathfrak{h}_1\map{}\mathfrak{h}_2$ by the rule 
$$
i_F(\tau_1,\tau_2) =  R\left( \begin{matrix} \tau_1 & \\ & \tau_2  \end{matrix}\right) {}^tR
$$
where the matrix $R$ is defined in  (\ref{twisty}).  Pulling back $\widehat{\phi}_2(\tau)$ 
by this embedding results in a Hilbert modular form $i_F^*\widehat{\phi}_2(\tau_1,\tau_2)$ of weight 
$3/2$ for the real quadratic field $F$,  having a Fourier expansion of the form
\begin{equation}\label{modular coefficients}
i_F^*\widehat{\phi}_2(\tau_1,\tau_2) = \sum_{\alpha\in\co_F} c(\alpha,v) \cdot 
q^\alpha
\end{equation}
in which $q^\alpha = e^{2\pi i \alpha\tau_1}e^{2\pi i \alpha^\sigma\tau_2}$,  $\sigma$  is the nontrivial 
Galois automorphism of $F/\Q$, and $v=(v_1,v_2)$ is the imaginary part of $(\tau_1,\tau_2)$,

Our main theorem is the following.

\begin{BigThm}\label{Big result}
Suppose that $\alpha\in\co_F$ and  $v\in F\otimes_\Q\R$ are both totally positive.   
If $2$ splits in $F$ and if $\alpha\co_F$ is relatively prime to the different of $F/\Q$ then
$$
\widehat{\mathrm{deg}}_{\mathcal{M}_0} \widehat{\mathcal{Y}}(\alpha,v) = c(\alpha,v).
$$
\end{BigThm}

\begin{Rem}
If $F(\sqrt{-\alpha})/\Q$ is not biquadratic then one does not need to assume that 
$2$ splits in $F$ or that $\alpha\co_F$ is relatively prime to the different (see the main 
result of \cite{howardA} or Theorem \ref{Thm:main result} below) and presumably the
result is true if these hypotheses are omitted altogether.  Both hypotheses are inherited from \cite{howardB}.  
If $p$ is nonsplit in $F$ then the stack $\mathcal{Y}(\alpha)$ may have nonreduced vertical components in
characteristic $p$, and if $p=2$ the calculation of the multiplicities of these  components in \cite[Theorem C]{howardB}
breaks down in a serious way.  To remove the assumption that $\alpha\co_F$ is prime to the different of $F/\Q$,
one would have to extend the statement and proof of \cite[Proposition 5.1.1]{howardB} to include the case of $c_0>0$.
Again, this probably requires some new ideas.
\end{Rem}

As for the proof of Theorem \ref{Big result},  we begin by noting that $c(\alpha,v)$ 
has already been computed by Kudla-Rapoport-Yang \cite{KRY}.    Let 
$$
\mathrm{Sym}_2(\Z)^\vee = \left\{ \left(\begin{matrix} a & b/2 \\ b/2 & c \end{matrix}\right) :   a,b,c\in\Z \right\}.
$$  
For each $T\in\mathrm{Sym}_2(\Z)^\vee$ and positive definite symmetric $\mathbf{v}\in M_2(\R)$, 
Kudla-Rapoport-Yang construct an arithmetic cycle class
$$
\widehat{\mathcal{Z}}(T,\mathbf{v}) \in \widehat{\mathrm{CH}}_\R^2(\mathcal{M}_0)
$$
(the arithmetic Chow group with real coefficients), and relate the Fourier coefficients of the 
genus two Siegel modular form $\widehat{\phi}_2(\tau)$ to the image of this class under the isomorphism
 $$
 \widehat{\deg}:\widehat{\mathrm{CH}}_\R^2(\mathcal{M}_0) \map{}\R
 $$
of \cite[\S 2.4]{KRY}.     More precisely, they prove that the Fourier expansion of $\widehat{\phi}_2(\tau)$ is 
$$
\widehat{\phi}_2(\tau) = \sum_{ T\in \mathrm{Sym}_2(\Z)^\vee} \widehat{\deg} \ \widehat{\mathcal{Z}} (T,\mathbf{v}) \cdot q^T
$$
where $\mathbf{v}$ is the imaginary part of $\tau\in\mathfrak{h}_2$, and 
$q^T=e^{2\pi i\cdot \mathrm{Tr}(T\tau)}$.  From this it is an easy exercise 
(see \cite[Lemma 5.2.1]{howardA}) to determine the Fourier coefficients of the pullback 
$\widehat{\phi}_2(\tau_1,\tau_2)$: for any $\alpha\in\co_F$ and any totally positive $v\in F\otimes_\Q\R$
$$
c(\alpha,v) = \sum_{T\in\Sigma(\alpha)} \widehat{\deg}\, \widehat{\mathcal{Z}}(T,\mathbf{v}), 
$$
in which $\mathbf{v}$ and $v=(v_1,v_2)\in \R\times\R$ are related by (\ref{twisty}) and
$$
\Sigma(\alpha) = \left\{
\left(\begin{matrix} a & b/2 \\ b/2 & c \end{matrix}\right)  \in\mathrm{Sym}_2(\Z)^\vee :  \alpha
=a\varpi_1^2+ b\varpi_1\varpi_2 + c\varpi_2^2 
\right\}.
$$
This leaves us with the problem of computing (\ref{intro deg}) for totally positive 
$\alpha$ and comparing with the values of $\widehat{\deg}\, \widehat{\mathcal{Z}}(T,\mathbf{v})$ 
(which are known by \cite[Chapter 6]{KRY}) in order to prove
\begin{equation}\label{arith decomp}
\widehat{\deg}_{\mathcal{M}_0} \widehat{\mathcal{Y}}(\alpha,v) = 
\sum_{T\in \Sigma(\alpha)} \widehat{\deg} \ \widehat{\mathcal{Z}} (T,\mathbf{v}),
\end{equation}
from which Theorem \ref{Big result} follows immediately.   It is the calculation of 
the left hand side of (\ref{arith decomp})  which  occupies the entirety of this paper, 
culminating in Theorem \ref{Thm:main result}.

In the calculation of (\ref{intro deg}) one encounters several obstacles.   
As noted earlier, the cycle $\mathcal{Y}(\alpha)$ on $\mathcal{M}$ may have 
components in codimension one, and so  must be modified in order to obtain the 
cycle class (\ref{my class}).  This is carried out in \cite{howardA}, and the construction 
of the modified components will be quickly recalled  in \S \ref{s:bad reduction}.    
 In order to compute (\ref{intro deg}) one must decompose the cycle $\mathcal{Y}(\alpha)$ 
 component-by-component,  and treat irreducible components in different ways depending 
 on whether they meet $\mathcal{M}_0$ properly or improperly.      If $\mathcal{D}$ is a component 
 of $\mathcal{Y}(\alpha)$ that meets  $\mathcal{M}_0$ improperly, then $\mathcal{D}$ is  
 contained in $\mathcal{M}_0$.  We attach an arithmetic  cycle class
$$
\widehat{\mathcal{D}}(v) \in \widehat{\mathrm{CH}}^2(\mathcal{M})
$$
to $\mathcal{D}$ and prove an  \emph{arithmetic adjunction formula} (Theorem \ref{Thm:adjunction}), 
which computes the arithmetic degree of $\widehat{\mathcal{D}}(v)$  along $\mathcal{M}_0$ 
in terms of data intrinsic to the divisor $\mathcal{D}$ on $\mathcal{M}_0$ (as opposed to data 
involving the relative positions of $\mathcal{D}$ and $\mathcal{M}_0$ in the larger threefold $\mathcal{M}$).    
While the method of derivation of the arithmetic adjunction formula should apply to the general problem of 
computing the arithmetic intersection of a codimension one cycle and a codimension two cycle  meeting  
improperly in an arithmetic threefold, the final formula is rather specific to the case at hand, as it makes 
use of the moduli interpretation of $i:\mathcal{M}_0\map{}\mathcal{M}$.  More precisely, an essential 
ingredient in the proof of Theorem \ref{Thm:adjunction} is the isomorphism of line bundles 
$$
i^*\omega\iso \omega_0\otimes\omega_0
$$
on $\mathcal{M}_0$, in which $\omega_0$ and $\omega$ are the canonical bundles on $\mathcal{M}_0$ and 
$\mathcal{M}$, respectively.  This isomorphism is proved in \S \ref{s:hodge} by using the 
Kodaira-Spencer isomorphism to give moduli-theoretic interpretations of the canonical bundles.

  This leaves the problem of  computing the intersection of the remaining components of 
  $\mathcal{Y}(\alpha)$  (those that are  not contained in $\mathcal{M}_0$)   against $\mathcal{M}_0$.      
  For these components  the arithmetic degree along $\mathcal{M}_0$ can be computed as a sum of 
  local intersections at each prime, plus an archimedean contribution.  Primes not dividing $\mathrm{disc}(B_0)$ 
  are treated in \S \ref{s:good reduction} using formal deformation theory, while primes dividing $\mathrm{disc}(B_0)$ 
  are treated in \S \ref{s:bad reduction} using the \v  Cerednik-Drinfeld uniformization of the formal completion 
  of $\mathcal{M}_{/\Z_p}$ along its special fiber.    Primes that are nonsplit in $F$  cause the most difficulty, 
  largely because of the presence of nonreduced vertical components in $\mathcal{Y}(\alpha)_{/\Z_p}$, whose 
  multiplicities must be determined.  As this calculation is very long and technical,  it has been relegated to a 
  separate article \cite{howardB}, which contains the bulk of the  deformation theory calculations at primes 
  nonsplit in $F$.  It is Proposition \ref{Prop:good singular III} and its corollary Proposition \ref{Prop:main unramified} 
  that make use of the calculations of \cite{howardB}.    Finally, the calculations of \S \ref{s:adjunction}, 
  \S \ref{s:good reduction}, 
  and \S \ref{s:bad reduction} are combined in \S \ref{s:pullbacks} to yield the final result Theorem \ref{Thm:main result}, 
  from which Theorem \ref{Big result} follows.

As explained in the introduction to \cite{howardA}, Theorem \ref{Big result} is one of the major steps 
toward the larger goal of proving a Gross-Zagier type theorem for Shimura surfaces, extending 
\cite[Corollary 1.0.7]{KRY} from Shimura curves to Shimura surfaces.  The next step is to find a good definition 
of the class (\ref{my class}) for all $\alpha\in \co_F$ (not just for $\alpha$ totally positive).  
For $\alpha\not=0$ but not totally positive the definition is straightforward: the cycle $\mathcal{Y}(\alpha)$ is empty, 
but the construction of the Green current $\Xi(\alpha,v)$ of  \S \ref{s:generic fiber}, and the proof of 
(\ref{arith decomp}) should pose no new difficulties.
 For $\alpha=0$ the definition of (\ref{my class}) is more subtle, 
and would follow \cite[\S 3.5]{KRY} or \cite[\S 6.5]{KRY}.  Roughly speaking, when $\alpha=0$ the class
(\ref{my class}) should be defined by viewing the metrized Hodge bundle of \S \ref{s:hodge}
as an element of $\widehat{\mathrm{CH}}^1(\mathcal{M})$, and taking its self-intersection.   Once 
(\ref{my class}) has been defined for all $\alpha$, the next step is to  form the 
generating series
\begin{equation}\label{generating series}
\widehat{\theta}(\tau_1,\tau_2) = \sum_{\alpha\in \co_F} 
\widehat{\mathcal{Y}}(\alpha,v) \cdot q^\alpha
 \in \widehat{\mathrm{CH}}^2(\mathcal{M}) [[q]].
\end{equation}
The extension of  Theorem \ref{Big result} to all $\alpha\in \co_F$ would then prove the equality of 
power series
\begin{equation}\label{generating pullback}
\widehat{\deg}_{\mathcal{M}_0}\widehat{\theta} (\tau_1,\tau_2)= i_F^* \widehat{\phi}_2 (\tau_1,\tau_2).
\end{equation}
The next step is the most challenging.  One would like to know, by analogy with  \cite[Theorem A]{KRY},  
that the generating series (\ref{generating series}) is a vector-valued, nonholomorphic, Hilbert modular form of weight $3/2$.
If this is the case, then given a weight $3/2$ Hilbert modular cuspform $f$, 
the Petersson inner product  of $f$ with (\ref{generating series}) defines a class
$$
\widehat{\Theta}(f) \in \widehat{\mathrm{CH}}^2(\mathcal{M}).
$$
If we denote by $L(f,s,B_0)$ the Petersson inner product of $f$ with $i_F^*\mathcal{E}_2(\tau_1,\tau_2,s,B_0)$,
the derivative $L'(f,0,B_0)$ is equal to the Petersson inner product of $f$ with $\widehat{\phi}_2$, and 
(\ref{generating pullback}) implies the Gross-Zagier style formula
$$
\widehat{\deg}_{\mathcal{M}_0} \widehat{\Theta}(f) = L'(f,0,B_0).
$$
The most serious obstacle between Theorem \ref{Big result} and this goal  is  the  modularity of the 
generating series (\ref{generating series}).

\subsection{Acknowledgements}

This research was supported in part by NSF grant DMS-0556174, and by a Sloan Foundation Research Fellowship.   
The author thanks Steve Kudla and  Michael  Rapoport for  helpful conversations.

\subsection{Notation}

Denote by $\mathfrak{D}_F$ the different of $F/\Q$.  Fix an embedding $F\hookrightarrow \R$,  and let $\sigma$
 be the nontrivial Galois automorphism of $F/\Q$.   Extend the ring homomorphism $F\map{}\R\times\R$ 
  defined by $\alpha\otimes 1\mapsto (\alpha,\alpha^\sigma)$ to an isomorphism of $\R$-algebras 
  $F\otimes_\Q \R \iso \R\times\R$,  denoted $v\mapsto (v_1,v_2)$.  

  Fix a positive involution $b\mapsto b^*$ of $B_0$ that leaves $\co_{B_0}$ stable and has the form 
  $b^*=s^{-1} b^\iota s$ for some $s\in \co_{B_0}$ with  $s^2=-\mathrm{disc}(B_0)$.  Extend $b\mapsto b^*$ 
  to an involution of $B$ that is trivial on $F$.   If $L$ is an algebraically closed field, 
  $\mathcal{X}$ is any algebraic stack, and  $P\in \mathcal{X}(L)$,  denote by 
  $\Aut_{\mathcal{X}}(P)$ the automorphism group of $P$ in the category $\mathcal{X}(L)$.


\section{Quaternionic Shimura varieties}
\label{s:moduli}


In this section we recall some of the basic definitions and notation of \cite{howardA}, concerning
abelian schemes with quaternionic multiplication, their moduli spaces, and the arithmetic Chow groups of those spaces.

By a \emph{QM abelian surface} (QM is short for \emph{quaternionic multiplication}) over a scheme $S$ we mean a pair 
$\mathbf{A}_0=(A_0,i_0)$ consisting of an abelian scheme $A_0\map{}S$  of relative dimension two and an action 
$i_0:\co_{B_0}\map{}\End(A_0)$ satisfying the Kottwitz condition of \cite[\S 3.1]{howardA}.  
A \emph{principal polarization} of $\mathbf{A}_0$ is a principal polarization  $\lambda_0:A_0\map{}A_0^\vee$ of the 
underlying abelian scheme that satisfies  $\lambda_0\circ i_0(b^*) = i_0(b)^\vee\circ \lambda_0$ for all 
$b\in \co_{B_0}$.  By a \emph{QM abelian fourfold} over a scheme $S$ we mean a pair $\mathbf{A}=(A,i)$ consisting of 
an abelian scheme $A\map{}S$  of relative dimension four and an action $i:\co_{B}\map{}\End(A)$ satisfying the 
Kottwitz condition.  A \emph{ $\mathfrak{D}_F^{-1}$-polarization} of $\mathbf{A}$ is a polarization  
$\lambda:A\map{}A^\vee$ of the underlying abelian fourfold that satisfies  $\lambda\circ i(b^*) = i(b)^\vee\circ \lambda$ 
for all $b\in \co_{B}$ and whose kernel is $A[\mathfrak{D}_F]$.  Such a $\lambda$ determines an isomorphism 
$A\otimes_{\co_F}\mathfrak{D}_F^{-1}\map{}A^\vee$.   By an endomorphism of $\mathbf{A}_0$ or $\mathbf{A}$ we 
mean an endomorphism of the underlying abelian scheme that commutes with the quaternionic action.

Let $\mathcal{M}_0$ be the Deligne-Mumford (DM) stack of principally polarized QM abelian surfaces over schemes, 
let $\mathcal{M}$ be the DM stack of  $\mathfrak{D}_F^{-1}$-polarized QM abelian fourfolds over schemes, and let  
$i:\mathcal{M}_0\map{}\mathcal{M}$
be the closed immersion  defined by  the functor
$$
(\mathbf{A}_0,\lambda_0) \mapsto (\mathbf{A}_0, \lambda_0) \otimes\co_F 
= (\mathbf{A}_0 \otimes\co_F, \lambda_0\otimes\co_F)
$$ 
as in \cite[\S 3.1]{howardA}.    The DM stacks $\mathcal{M}_0$ and $\mathcal{M}$ are regular  of dimensions two and 
three, respectively,  and are flat  and projective over $\Spec(\Z)$.  If we abbreviate $D=\mathrm{disc}(B_0)$  then  
 $\mathcal{M}_0$ is smooth  over  $\Z[1/D]$, and $\mathcal{M}$ is smooth  over  $\Z[1/ (Dd_F) ]$. 
  For references to proofs of these properties, see \cite[\S 3.1]{howardA}.  Briefly, for representability of $\mathcal{M}$
  by a quasi-projective stack use \cite[Chapters 6 and 7]{hida04}; for properties  of $\mathcal{M}$ over 
  $\Z[(Dd_F)^{-1}]$ use  \cite[Expos\'e III]{breen-labesse},   for properties 
  at primes dividing $d_F$ repeat the arguments of \cite{deligne-pappas}, and for primes dividing $D$ 
  use the Cerednick-Drinfeld uniformization  described in \cite[\S 4]{howardA};   projectivity follows from 
  quasi-projectivity and the valuative criterion of properness using \cite[Expos\'e III.6]{breen-labesse}.  
   It follows from  \cite[\S 6.3.2]{liu} that  $i$ is a regular 
immersion and hence that $\mathcal{M}_0$ is an effective Cartier divisor on $\mathcal{M}$.

If $(\mathbf{A}_0 , \lambda_0)$ is a principally polarized QM abelian surface over a connected base scheme $S$,  then 
$\lambda_0$ determines a Rosati involution $\tau\mapsto \tau^\dagger$ on the $\Q$-algebra $\End^0(\mathbf{A}_0)$ 
of  $B_0$-linear quasi-endomorphisms of  $\mathbf{A}_0$.  The \emph{Rosati trace} on $ \End^0(\mathbf{A}_0)$ is defined by 
$\mathrm{Tr}(\tau)=\tau+\tau^\dagger$, and an  endomorphism of $\mathbf{A}_0$ of Rosati trace zero is  called a
 \emph{special endomorphism}.  By \cite[Lemma 3.1.2]{howardA} the $\Q$-algebra 
 $ \End^0(\mathbf{A}_0)$ is either $\Q$, a quadratic  imaginary field, or a definite quaternion algebra.  
 As the Rosati involution is positive, it must be (in the three cases respectively) the identity, complex conjugation, 
 or the main involution (in the case of a definite quaternion algebra, this follows from  
  Albert's classification of division algebras over $\Q$ with  a positive involution \cite[Chapter 21]{mumford70}).  
  In particular the Rosati trace agrees with the reduced trace of 
 \cite[Chapter 19]{mumford70}, and our definition of special endomorphism agrees with the definition used in \cite{KRY}.
 The $\Z$-module of special endomorphisms of $(\mathbf{A}_0,\lambda_0)$ is 
 equipped with the symmetric $\Z$-valued bilinear form   $[\tau_1,\tau_2]= -\mathrm{Tr}(\tau_1\tau_2)$ and its 
 associated quadratic form $Q_0(\tau)=  -\tau^2.$

 Similarly, if $(\mathbf{A} , \lambda) $ is a principally polarized QM 
 abelian fourfold over  $S$  then the $F$-algebra $\End^0(\mathbf{A})$ of $B$-linear quasi-endomorphisms of $A$ 
 comes equipped with the  $F$-linear Rosati involution $\tau\mapsto \tau^\dagger$ determined by $\lambda$, and the 
 endomorphisms of $\mathbf{A}$ of Rosati trace zero  are again called \emph{special endomorphisms}.  The 
 $\co_F$-module of special endomorphisms of $(\mathbf{A},\lambda)$ has a symmetric bilinear form 
 $[\tau_1,\tau_2]= -\mathrm{Tr}(\tau_1\tau_2)$ and an associated quadratic form $Q(\tau)=-\tau^2$, 
 each of which is $\co_F$-valued.

For each nonzero $t\in \Z$  define, following \cite[\S 3.4]{KRY}, $\mathcal{Z}(t)$ to be the DM stack of triples 
$(\mathbf{A}_0,\lambda_0,s_0)$ in which $(\mathbf{A}_0,\lambda_0)$ is a principally polarized QM abelian surface 
over a scheme  and $s_0\in\End(\mathbf{A}_0)$ is a special endomorphism that satisfies $Q_0(s_0)= t$.  We view  
$\mathcal{Z}(t)$  also as a codimension one cycle on $\mathcal{M}_0$.  This means that every irreducible component of 
$\mathcal{Z}(t)$   is viewed as an irreducible cycle of $\mathcal{M}_0$ via the forgetful morphism, and 
is weighted according to the length of the strictly Henselian local ring at its generic point.   
By \cite[Proposition 3.4.5]{KRY}  the cycle $\mathcal{Z}(t)$ has no vertical components except possibly at primes 
dividing $\mathrm{disc}(B_0)$.  As a cycle we  decompose  
$\mathcal{Z}(t)=\mathcal{Z}^\horizontal(t) + \mathcal{Z}^\vertical(t)$
 into its horizontal and vertical parts, and then further decompose
$$
\mathcal{Z}^\vertical(t) = \sum_{p\mid\mathrm{disc}(B_0)} \mathcal{Z}^\vertical(t)_p.
$$  
As in \cite[\S 3.6]{KRY}, for each nonzero  $T\in\mathrm{Sym}_2(\Z)^\vee$ let $\mathcal{Z}(T)$ be the DM stack of 
quadruples $(\mathbf{A}_0,\lambda_0,s_1,s_2)$,  in which $(\mathbf{A}_0,\lambda_0)$ is as above  and 
$s_1,s_2\in\End(\mathbf{A}_0)$ are special endomorphisms that satisfy
\begin{equation}\label{quadratic matrix}
\frac{1}{2} \left( \begin{matrix}
[s_1,s_1]  &  [s_1,s_2] \\  [s_1,s_2] & [s_2,s_2]
\end{matrix}  \right) = T.
\end{equation}
If $\det(T)\not=0$ then, by \cite[Theorem 3.6.1]{KRY},  $\mathcal{Z}(T)$ is either empty or all of its points 
have residue field of the same characteristic $p\not=0$.  If this characteristic $p$ does not divide $\mathrm{disc}(B_0)$ 
then $\mathcal{Z}(T)$ is of dimension zero, while if $p$ does divide $\mathrm{disc}(B_0)$ then $\mathcal{Z}(T)$ may 
have vertical components of dimension one.  If $T\not=0$ but $\det(T)=0$, then there is a $t\in\Z$ for which 
$\mathcal{Z}(T)\iso \mathcal{Z}(t)$.  See \cite[Lemma 6.4.1]{KRY} or (\ref{degenerate cycle}) below. 
 In any case one has $\mathrm{dim}\ \mathcal{Z}(T)\le 1$.

For each nonzero $\alpha\in\co_F$ let $\mathcal{Y}(\alpha)$ be the DM stack of triples 
$(\mathbf{A},\lambda,t_\alpha)$,  in which $(\mathbf{A},\lambda)$ is a $\mathfrak{D}_F^{-1}$-polarized 
QM abelian fourfold over a scheme and $t_\alpha$ is a special endomorphism of $\mathbf{A}$ satisfying 
$Q(t_\alpha)=\alpha$.  
Let $\phi:\mathcal{Y}(\alpha)\map{}\mathcal{M}$ be the functor that forgets the data $t_\alpha$, and   define 
$\mathcal{Y}_0(\alpha)=\mathcal{Y}(\alpha)\times_{\mathcal{M}}\mathcal{M}_0$, so that there is a cartesian diagram
$$
\xymatrix{
{\mathcal{Y}_0(\alpha)}  \ar[r]^{\phi_0} \ar[d]_{ j }  & { \mathcal{M}_0 }  \ar[d]^{i}  \\
{ \mathcal{Y}(\alpha)  } \ar[r]_{\phi}  & {\mathcal{M}, }
}
$$
in which both vertical arrows are closed immersions and both horizontal arrows are proper and quasi-finite, hence finite.  
As in  \cite[\S 3.1]{howardA}, there is a canonical decomposition
\begin{equation}\label{moduli decomp}
\mathcal{Y}_0(\alpha) \iso \bigsqcup_{T\in\Sigma(\alpha)} \mathcal{Z}(T)
\end{equation}
defined as follows.  An object of the category $\mathcal{Y}_0(\alpha)$ consists of an object 
$(\mathbf{A},\lambda,t_\alpha)$ of $\mathcal{Y}(\alpha)$, an object  $(\mathbf{A}_0,\lambda_0)$ of $\mathcal{M}_0$, 
and an isomorphism $(\mathbf{A},\lambda)\iso (\mathbf{A}_0,\lambda_0)\otimes\co_F$.  This isomorphisms determines 
an isomorphism 
$$
\End(\mathbf{A})\iso \End(\mathbf{A}_0)\otimes\co_F
$$
which allows us to write $t_\alpha= s_1\varpi_1+s_2\varpi_2$ for some special endomorphisms 
$s_1,s_2\in\End(\mathbf{A}_0)$.  The condition $Q(t_\alpha)=\alpha$ is equivalent to the condition that the matrix
 (\ref{quadratic matrix}) lies in $\Sigma(\alpha)$, and the isomorphism (\ref{moduli decomp}) takes the above data to the 
 quadruple $(\mathbf{A}_0,\lambda_0,s_1,s_2)$.

\begin{Rem}
If $P\in\mathcal{M}(L)$ with $L$ an algebraically closed field, define
 $$
 e_P = |\Aut_{\mathcal{M}(L)}(P) |.
 $$  
 If $P\in\mathcal{M}_0(L)$ then we will routinely confuse $P$ with its image in $\mathcal{M}(L)$.  As 
  \cite[Lemma 3.1.1]{howardA} implies that the automorphism group of $P$ in $\mathcal{M}_0(L)$ is isomorphic to the 
  automorphism group of $P$ in $\mathcal{M}(L)$, we also have
 $$
 e_P = |\Aut_{\mathcal{M}_0(L)}(P) |.
 $$  
  Similarly, if $P\in \mathcal{Z}(T)(L)$ (respectively $P\in\mathcal{Z}(t)(L)$) then $e_P$ denotes the size of the 
  automorphism group of $P$ in $\mathcal{M}_0(L)$, where  $P$ is regarded as an object of this latter category via the 
  obvious forgetful map $\mathcal{Z}(T)\map{}\mathcal{M}_0$ (respectively $\mathcal{Z}(t)\map{}\mathcal{M}_0$).
\end{Rem}

The remainder of this section is devoted to the careful construction of the linear functional (\ref{relative degree}).
If $\mathcal{X}$ is either $\mathcal{M}_0$ or $\mathcal{M}$,  we let $\widehat{Z}^k(\mathcal{X})$ be the $\Q$-vector 
space of pairs $(\mathcal{D},\Xi)$, in which $\mathcal{D}$ is a codimension $k$ cycle on $\mathcal{X}$ with rational 
coefficients, and $\Xi$ is an equivalence class of Green currents for $\mathcal{D}$.   The codimension $k$ arithmetic 
Chow group  $\widehat{\mathrm{CH}}^k(\mathcal{X})$ of $\mathcal{X}$, as defined by Gillet-Soul\'e 
\cite{BGS, gillet-soule90, soule92}, is the quotient of $\widehat{Z}^k(\mathcal{X})$ by the 
subspace spanned by pairs of the form 
$$
\widehat{\mathrm{div}}(f) = \big(\mathrm{div}(f), [-\log|f|^2]\big)
$$ 
for $f$ a rational function on an integral substack of $\mathcal{X}$ of codimension $k-1$.  Any class  
$\widehat{\mathcal{D}}\in \widehat{\mathrm{CH}}^2(\mathcal{M})$  may be represented by a pair 
$(\mathcal{D},\Xi_\mathcal{D})$ in which $\mathcal{D}$ has support disjoint from $\mathcal{M}_0$ 
in the generic fiber (by first expressing the generic fiber $\mathcal{M}_{/\Q}$ as a the quotient of a 
$\Q$-scheme $M$ by the action of a finite group $H$,  applying the  Moving Lemma over $\Q$ 
proved in  \cite{roberts72}, and then averaging over $H$).   Thus to define $\widehat{\deg}_{\mathcal{M}_0}\widehat{\mathcal{D}}$
we may assume that $\mathcal{D}$ is irreducible and is disjoint from $\mathcal{M}_0$ in the generic fiber of $\mathcal{M}$.
The  arithmetic degree along $\mathcal{M}_0$ is then defined as a sum of local contributions, which we now describe.

Fix a prime $p$ and  an isomorphism of stacks $\mathcal{M}_{/\Z_p} \iso [H\backslash M]$ with $M$ a $\Z_p$-scheme 
and $H$ a finite group of automorphisms of $M$  (for example by  imposing prime-to-$p$ level structure on the moduli 
problem defining the stack $\mathcal{M}$).   Set 
\begin{equation}\label{M_0}
M_0=\mathcal{M}_0\times_{\mathcal{M}} M .
\end{equation}  
If $D$ is an irreducible  cycle of codimension two on $M$ that is not contained in $M_0$,    define the 
\emph{Serre intersection multiplicity} at $p$ 
\begin{equation}\label{first finite intersection}
I_p(D , M_0 ) = \sum_{x\in M_0(\F_p^\alg)}   \sum_{\ell \ge 0}
  (-1)^\ell  \cdot  \length_{\co_{M_0,x}} \Tor_\ell^{  \co_{ M,x } } (\co_{D,x}, \co_{M_0,x} )
\end{equation}
where we view both $\co_D$ and $\co_{M_0}$ as coherent $\co_{M}$-modules.  In fact only the $\ell=0$ term 
contributes to the right hand side: as $D$ is integral of dimension one, the  local ring of $\co_D$ at any point of $D$ is a 
Cohen-Macaulay local ring, and hence the stalk  $\co_{D,x}$ at any  $x\in M(\F_p^\alg)$ is a Cohen-Macaulay 
$\co_{M,x}$-module \cite[p.~63]{serre00}.  The regularity of $M_0$ implies that the stalk $\co_{M_0,x}$ is also  
Cohen-Macaulay as an $\co_{M,x}$-module, and hence
\begin{equation}\label{no tor}
 \Tor_\ell^{  \co_{ M,x } } (\co_{D,x}, \co_{M_0,x} ) =0
\end{equation}
for $\ell>0$ by \cite[p.~111]{serre00}.

The definition (\ref{first finite intersection})  can be extended to  cycles supported in characteristic $p$, including 
those that meet $M_0$ improperly.  If $\mathcal{F}_0$ is a coherent $\co_{M_{0/\F_p}}$-module,  define the 
Euler characteristic 
$$
\chi(\mathcal{F}_0) =  \sum_{k \ge 0} (-1)^k \mathrm{dim}_{\F_p} H^k( M_{0/\F_p}, \mathcal{F}_0).
$$
If the sheaf $\mathcal{F}_0$ is supported in dimension zero then 
\begin{equation}\label{simple euler}
\chi(\mathcal{F}_0) =  \sum_{x\in M_0(\F_p^\alg)}     \length_{\co_{M_0,x}} \mathcal{F}_{0,x}.
\end{equation}
For any irreducible vertical cycle $D$ of codimension two on $M$, the coherent $\co_M$-module
 $\Tor_\ell^{  \co_{ M,x } } (\co_{D,x}, \co_{M_0,x} )$ is  annihilated by $p$ and by the ideal sheaf of the closed 
 subscheme $M_0\map{}M$, and hence may be viewed as a coherent $\co_{M_{0/\F_p}}$-module.  
 Thus we may define
 \begin{equation}\label{second finite intersection}
 I_p(D , M_0 )  =    \sum_{\ell \ge 0}  (-1)^\ell  \cdot  \chi( \Tor_\ell^{  \co_M } (\co_D, \co_{M_0} ) ).
 \end{equation}
If  $D$ is both vertical and not contained in $M_0$ then one sees using (\ref{simple euler}) that the two definitions 
(\ref{first finite intersection}) and (\ref{second finite intersection}) of $I_p(D,M_0)$ agree.   By extending $I_p(D,M_0)$ 
linearly in the first variable, we then define $I_p(D,M_0)$ for any codimension two cycle $D$ on $M$ whose support is 
disjoint from $M_0$ in the generic fiber.   If  $f$ is a rational function on any irreducible component of $M_{/\F_p}$ and 
$D$ is the associated Weil divisor on $M_{/\F_p}$, viewed as a vertical cycle on $M$ of codimension two, then  one can 
show that $I_p(D,M_0)=0$.      By  \cite[Lemma 4.2]{gillet84}, any codimension two cycle  $\mathcal{D}$ on 
$\mathcal{M}$ with support disjoint from $\mathcal{M}_0$ in the generic fiber determines an $H$-invariant codimension 
two cycle $D$ on $M$  that is disjoint from $M_0$ in the generic fiber.   Thus we may define
$$
I_p(\mathcal{D},\mathcal{M}_0) = \frac{1}{|H|} I_p(D,M_0).
$$

Now consider the situation at the infinite place.  Choose an isomorphism of stacks 
$\mathcal{M}_{/\Q} \iso [H\backslash M]$ with $M$ a $\Q$-scheme and $H$ a finite group of automorphisms of $M$.  
Again define $M_0$ by (\ref{M_0}).   Suppose that $D$ is any codimension two cycle on $M$ that is disjoint from $M_0$, 
and that $\Xi_D$ is a Green current for $D$ in the sense of \cite[\S 1.2]{gillet-soule90}.  We give two definitions of 
$I_\infty(\Xi_D,M_0)$.  The first definition uses the methods of \cite[\S 1.3]{gillet-soule90}.   We say that two currents 
$\Xi$ and $\Xi'$ on $M$ (or on  $M_0$) are \emph{equivalent} if there are smooth currents $u$ and $v$ such that 
$$
\Xi'=\Xi +  \partial u + \overline{\partial} v.
$$
One may replace $\Xi_D$ by an equivalent current $\Xi_D'$ that  is a Green form of logarithmic type for $D$.  In 
particular, as  the support of $D$ is assumed to be disjoint from $M_0$,  $\Xi_D'$ is represented by a smooth 
$(1,1)$-form in a complex neighborhood of $M_0$, and  the pullback $i^*\Xi_D'$ is a differential form of top degree on 
the smooth manifold $M_0(\C)$.  Define 
$$
I_\infty(\Xi_D,M_0) = \frac{1}{2} \int_{M_0(\C)} i^*\Xi_D'.
$$
The second definition uses the methods of \cite[\S 2.1.5]{gillet-soule90}.   Briefly, one can construct a family 
$\{\omega^\epsilon\}_{\epsilon>0}$ of smooth $(1,1)$-forms on $M(\C)$ that converge, as $\epsilon\to 0$, to the 
delta current $\delta_{M_0}$ on $M$.  One then defines
$$
I_\infty(\Xi_D,M_0) = \lim_{\epsilon\to 0} \frac{1}{2} \int_{M(\C)} \omega^\epsilon \wedge \Xi_D
$$
where the integral on the right is understood to mean evaluation of the current $\omega^\epsilon \wedge \Xi_D$ at the 
constant function $1$.  One checks that this definition agrees with the first definition using \cite[\S 2.2.12]{gillet-soule90}.  
Now suppose $\mathcal{D}$ is a codimension two cycle on $\mathcal{M}$ and let $D$ be the associated $H$-invariant 
cycle on $M$ as in the previous paragraph.  A Green current $\Xi_{\mathcal{D}}$ for $\mathcal{D}$ is  defined to 
be an $H$-invariant Green current  $\Xi_D$ for $D$.   If $\mathcal{D}$ has support disjoint from $\mathcal{M}_0$ in the 
generic fiber and $\Xi_{\mathcal{D}}$ is a Green current for $\mathcal{D}$, we define
$$
I_\infty(\Xi_\mathcal{D},\mathcal{M}_0) = \frac{1}{|H|} I_\infty( \Xi_D, M_0).
$$

The  \emph{arithmetic degree along $\mathcal{M}_0$} of $\widehat{\mathcal{D}}$
$$
\widehat{\deg}_{\mathcal{M}_0}\widehat{\mathcal{D}} = I_\infty(\Xi,\mathcal{M}_0) 
+ \sum_{p\mathrm{\ prime}}  I_p(\mathcal{D},\mathcal{M}_0)\log(p)
$$
does not depend on the choice of representative $(\mathcal{D},\Xi_\mathcal{D})$, and defines the
desired linear functional
\begin{equation}\label{arithmetic degree II}
\widehat{\deg}_{\mathcal{M}_0} : \widehat{\mathrm{CH}}^2(\mathcal{M})  \map{}\R .
\end{equation}
One proves that the arithmetic degree along $\mathcal{M}_0$ does not depend of the 
choice of $(\mathcal{D},\Xi_\mathcal{D})$ representing $\widehat{\mathcal{D}}$
by showing that the definition given above agrees with the definition found in  \cite[\S 2.3]{howardA}.


\section{Complex uniformization}
\label{s:generic fiber}


Fix a totally positive $\alpha\in\co_F$ and abbreviate  $\mathcal{Y}=\mathcal{Y}(\alpha)$.    In this 
section we review the well-known complex uniformizations of $\mathcal{M}_0(\C)$ and $\mathcal{M}(\C)$, 
and construct a Green current for the $0$-cycle $\mathcal{Y}(\C)$ on $\mathcal{M}(\C)$.

Choose an isomorphism of stacks $\mathcal{M}_{/\Q}\iso [H\backslash M]$ with $M$ a $\Q$-scheme and 
$H$ a finite group of automorphisms of $M$, and abbreviate  
$$
Y =\mathcal{Y} \times_{\mathcal{M}} M.
$$ 
Recalling the forgetful map $\phi:Y\map{}M$, we attach to $Y$ the $0$-cycle  on $M$
\begin{equation}\label{s:general cycle}
C_\Q=\sum_{y \in Y}  \length_{ \co_{Y,y} } (  \co_{Y,y} ) \cdot \phi(y)
\end{equation}
which, using \cite[Lemma 4.2]{gillet84}, descends to a codimension two  cycle  $\mathcal{C}_\Q$ on 
$\mathcal{M}_{/\Q}$ independent of  the choice of  presentation $M\map{}\mathcal{M}_{/\Q}$.   
There is a unique decomposition 
$$
\mathcal{C}_\Q =  \mathcal{C}_\Q^\good + \mathcal{C}_\Q^\bad
$$ 
of cycles on $\mathcal{M}_{/\Q} $  such that  $\mathcal{C}^\bad_\Q$ is supported on  
$\mathcal{M}_{0/\Q}$ and $\mathcal{C}^\good_\Q$ has support disjoint from $\mathcal{M}_{0/\Q}$.   
We will construct Green currents for the cycles $\mathcal{C}_\Q^\good$ and $\mathcal{C}_\Q^\bad$.

\begin{Rem}
In (\ref{s:general cycle}) we in fact have $\length_{ \co_{Y,y} } (  \co_{Y,y} )=1$  for each $y\in Y$.  
Indeed,  $\mathcal{Y}_{/\Q}$  is \'etale over $\Spec(\Q)$ by \cite[Lemma 3.1.3]{howardA}, and in 
particular $Y$ is a disjoint union of spectra of  number fields.
\end{Rem}

Let $X_0$ and $X$ be as in the introduction.
The obvious inclusion $X\map{}X_0\times X_0$ is denoted $x\mapsto (x_1,x_2)$ and $\pi_i:X\map{}X_0$ 
denotes  the function $\pi_i(x)=x_i$.   Let 
$$
\mu_0= y^{-2} \cdot dx\wedge dy
$$ 
be the usual hyperbolic volume form on $X_0$.    For any positive  $u\in\R$,  Kudla  \cite[\S 7.3]{KRY}  
has constructed a symmetric Green function $g_u^0(z_1,z_2)$  for the diagonal  on $X_0\times X_0$, and a smooth 
symmetric function  $\phi_u^0(z_1,z_2)$   on  $X_0\times X_0$.   These functions have the property that for any fixed  
$x_0\in X_0$,  the  smooth  function in the variable $z_0\in X_0\smallsetminus\{x_0\}$
$$
\mathbf{g}_0(x_0,u) (z_0)= g_u^0(x_0,z_0)
$$
and the smooth $(1,1)$-form on $X_0$
$$
\Phi_0(x_0,u) (z_0)= \phi_u^0(x_0, z_0 )\mu_0(z_0)
$$
 satisfy the Green equation
$$
dd^c\mathbf{g}_0(x_0,u) + \delta_{x_0} = \Phi_0(x_0,u)
$$
of $(1,1)$-currents on $X_0$.  For a point $x\in X^\pm$ and a  pair $v=(v_1,v_2)$ of positive real numbers the functions
$$
\mathbf{g}_1(x,v)= \pi_1^*\mathbf{g}_0(x_1,v_1)
\qquad
\mathbf{g}_2(x,v)= \pi_2^*\mathbf{g}_0(x_2,v_2)
$$
are Green functions for the divisors $\{x_1\}\times X_0^\pm$ and $X_0^\pm\times\{x_2\}$ on $X^\pm$.  
Extend these functions by $0$ to $X^\mp$.   If we define $(1,1)$-forms
$$
\Phi_1(x,v) = \pi_1^*\Phi_0(x_1,v_1)
\qquad
\Phi_2(x,v) = \pi_2^* \Phi_0(x_2,v_2)
$$
on $X^\pm$ and extend by $0$ to $X^\mp$, then the star product defined by \cite[\S 2.1]{gillet-soule90}
\begin{eqnarray*}
\mathbf{g}(x,v)  &=& \mathbf{g}_1(x , v ) *   \mathbf{g}_2(x , v ) \\
&=&  \mathbf{g}_1(x,v) \wedge  \delta_{  X_0^{\pm}\times\{x_2\} }   + \mathbf{g}_2(x,v)   \wedge \Phi_1(x,v) \\
&=& \mathbf{g}_2(x,v) \wedge \delta_{\{x_1\}\times X_0^{\pm} } + \mathbf{g}_1(x,v) \wedge  \Phi_2(x , v )
\end{eqnarray*}
is a $(1,1)$-current on $X$ satisfying the Green equation
$$
dd^c \mathbf{g}(x,v) +\delta_x =  \Phi_1( x ,  v )  \wedge  \Phi_2( x , v ).
$$

Our fixed isomorphism $B_0\otimes_\Q\R\iso M_2(\R)$ determines an isomorphism $G_0(\R)\iso \mathrm{GL}_2(\R)$, 
and hence determines an action of $G_0(\R)$ on $X_0$.  The induced isomorphism 
$B\otimes_\Q\R\iso M_2(\R)\times M_2(\R)$ then determines an isomorphism
$$
G(\R)\iso \{ (g_1,g_2)\in \mathrm{GL}_2(\R)\times\mathrm{GL}_2(\R) : \det(g_1)=\det(g_2) \},
$$
and hence an action of $G(\R)$ on $X$.   The inclusion $G(\R)\map{}G_0(\R)\times G_0(\R)$ 
is denoted $\gamma\mapsto (\gamma_1,\gamma_2)$.  By Shimura's theory there are orbifold presentations
\begin{equation}\label{orbifold}
\mathcal{M}_0(\C) \iso [\Gamma_0\backslash X_0 ]
\qquad
\mathcal{M}(\C)\iso [\Gamma \backslash X],
\end{equation}
and the morphism $\mathcal{M}_0(\C)\map{}\mathcal{M}(\C)$ of \S \ref{s:moduli} is induced by the 
diagonal inclusion $X_0\map{}X$.

Following \cite[\S 3.2]{howardA} the fibers of the universal QM abelian surface on 
$[\Gamma_0\backslash X_0]$ can be described as follows.  For each $z_0\in X_0$ 
define an isomorphism of real vector spaces
$$
 \rho_{0,z_0} : B_0\otimes_\Q\R \map{} \C^2
\qquad
\rho_{0,z_0}(A) = A\cdot \left[\begin{matrix}  z_0 \\ 1\end{matrix}\right]
$$
and set $\Lambda_{0,z_0}=\rho_{0,z_0}(\co_{B_0})$.  Then $\mathbf{A}_{0,z_0} = \C^2/\Lambda_{0,z_0}$ is a QM abelian surface, 
and the perfect alternating pairing $\psi_0 :  \co_{B_0}\times \co_{B_0}\map{} \Z$ defined in \cite[\S 3.1]{howardA}  
determines a pairing $\psi_{0,z_0}$ on  $\Lambda_{0,z}$ with the property that one of $\pm \psi_{0,z_0}$ 
(depending on the connected component of $X_0$ containing $z_0$) is a Riemann form.  Thus $\mathbf{A}_{0,z_0}$ 
comes equipped with a principal polarization $\lambda_{0,z_0}$, and $(\mathbf{A}_{0,z_0} , \lambda_{0,z_0})$ is a 
QM abelian surface that depends only on the $\Gamma_0$-orbit of $z_0$.  Similarly, for each $z\in X$ we write 
$(z_1,z_2)$ for the corresponding point of $X_0\times X_0$ and  define an isomorphism
$$
\rho_z: B\otimes_\Q\R\iso (B_0\otimes_\Q\R) \times (B_0\otimes_\Q\R) \iso \C^2\times\C^2
$$
by  $\rho_z=\rho_{0,z_1} \times \rho_{0,z_2}$.  Set $\Lambda_z= \rho_z(\co_B)$.  Extend $\psi_0$ 
$\co_F$-linearly to an $\co_F$-valued pairing on $\co_B\iso \co_{B_0}\otimes_\Z\co_F$, and define 
$\psi= \mathrm{Tr}_{F/\Q}\circ \psi_0$.  As above,  $\psi$ determines a pairing $\psi_z$ on $\Lambda_z$, 
and one of $\pm \psi_z$ is a Riemann form for $\Lambda_z$.  The resulting polarization $\lambda_z$ of 
$\mathbf{A}_z = (\C^2\times\C^2)/\Lambda_z$ determines a $\mathfrak{D}_F^{-1}$-polarized QM abelian fourfold 
$(\mathbf{A}_z,\lambda_z)$ that depends only on the $\Gamma$-orbit of $z$.

Let $V_0$ and $V$ denote the trace zero elements of $B_0$ and $B$, respectively, with 
$G_0(\Q)$ and $G(\Q)$ acting on $V_0$ and $V$ by conjugation.  We  identify
 $$
 V\otimes_\Q\R\iso (V_0\otimes_\Q\R) \times  (V_0\otimes_\Q\R)
 $$ 
 and write $\tau\mapsto (\tau_1,\tau_2)$ for the isomorphism.  The $F$-vector space $V$ is 
 endowed with the $G(\Q)$-invariant $F$-valued quadratic form $Q(\tau)=-\tau^2$, and 
 $V_0$ is endowed  with the $G_0(\Q)$-invariant  $\Q$-valued quadratic form $Q_0$ defined by 
 the same formula.  Each $\tau_0 \in V_0\otimes_\Q\R$ with $Q_0(\tau)$  positive,  viewed as an 
 element of $G_0(\R)$,  acts  on $X_0$ with two fixed points  
 $$
 x_0^\pm(\tau_0)\in X_0^\pm,
 $$
  and the  $(0,0)$-current on $X_0$ defined by
 $$
 \xi_0(\tau_0) = \mathbf{g}_0\big(x_0^+(\tau_0) , Q_0(\tau_0) \big) + \mathbf{g}_0\big(x_0^-(\tau_0) , Q_0(\tau_0) \big) 
 $$
 is a Green current for the $0$-cycle $x_0^+(\tau_0) + x_0^-(\tau_0)$.   Given $\tau\in V\otimes_\Q\R$ with 
 $Q(\tau)$ totally positive,  the fixed points of $\tau$ acting on $X$ are 
$$
x^+(\tau)=(x_0^+(\tau_1),x_0^+(\tau_2)) \qquad x^-(\tau)=(x_0^-(\tau_1),x_0^-(\tau_2)), 
$$
 and the $(1,1)$-current  on $X$ defined by
$$
\xi(\tau) = \mathbf{g} \big( x^+(\tau), Q(\tau) \big)  +  \mathbf{g} \big( x^-(\tau) , Q(\tau) \big) 
$$
is a Green current for the $0$-cycle $x^+(\tau)+x^-(\tau)$.    Set $L=V\cap\co_B$.  
For a totally positive $v\in F\otimes_\Q\R$ the current
\begin{equation}\label{alpha current}
\Xi(\alpha,v) =  \sum_{  \substack{ \tau \in L \\ Q(\tau)=\alpha} }  \xi (  v^{1/2}\tau)  
=    \sum_{  \substack{ \tau \in L \\ Q(\tau)=\alpha} }   
\big( \mathbf{g} ( x^+(\tau), \alpha v ) + \mathbf{g} ( x^-(\tau), \alpha v )\big)
\end{equation}
is $\Gamma$-invariant and so descends to a $(1,1)$-current on the orbifold $\mathcal{M}(\C)$ 
which we denote in the same way.    Decompose $L=L^\sing\sqcup L^\nonsing$ in which 
$L^\sing$ (resp.~ $L^\nonsing$) is the subset consisting of those $\tau$ for which the vectors  
$\tau_1, \tau_2\in V_0\otimes_\Q\R$ are linearly dependent (resp.~ linearly independent).  Note that 
\begin{equation}\label{in diag}
\tau\in L^\sing \iff x^\pm (\tau) \in X_0.
\end{equation}
  We now decompose $L=L^\good \sqcup L^\bad$ in which 
$$
L^\bad =  \Gamma L^\sing \qquad L^\good =L\smallsetminus L^\bad
$$
and define $(1,1)$-currents  $\Xi^\good(\alpha,v)$ and  $\Xi^\bad(\alpha,v)$ on $X$ exactly as in (\ref{alpha current}), 
but with $L$ replaced by $L^\good$ and $L^\bad$, respectively.  Both currents are invariant under the action of 
$\Gamma$ and so define currents on the orbifold $\mathcal{M}(\C)$.

\begin{Prop}\label{Prop:green construction}
The currents $\Xi(\alpha,v)$, $\Xi^\good(\alpha,v)$, and $\Xi^\bad(\alpha,v)$  are Green 
currents for the $0$-cycles $\mathcal{C}_\Q$, $\mathcal{C}_\Q^\good$, and  $\mathcal{C}_\Q^\bad$ on 
$\mathcal{M}$, respectively.
\end{Prop}

\begin{proof}
It suffices to prove any two of the three claims.  As in \cite[\S 3.2]{howardA} there is an isomorphism of zero dimensional orbifolds
$$
\mathcal{Y}(\alpha) (\C) \iso\Big[ \Gamma\big\backslash
 \bigsqcup_{  \substack{  \tau\in L \\ Q(\tau)=\alpha  }   } \{ x^+(\tau), x^-(\tau) \} \Big].
$$
This shows that the pullback of the cycle $\mathcal{C}_\Q$ to $X$ is equal to the formal sum
$$
\sum_{ \substack{  \tau\in L \\ Q(\tau)=\alpha  }  } (x^+(\tau) +  x^-(\tau)),
$$
which has $\Xi(\alpha,v)$ as a Green current.  The pullback of $\mathcal{C}_\Q^\bad$ to $X$ consists 
of that portion of the above sum whose support is contained in the $\Gamma$-orbit of $X_0$.  This is the formal sum 
$$
\sum_{ \substack{  \tau\in L^\bad \\ Q(\tau)=\alpha  }  } (x^+(\tau) +  x^-(\tau)),
$$
which has $\Xi^\bad(\alpha,v)$ as a Green current.
\end{proof}


\section{The Hodge bundle}
\label{s:hodge}


In this section we define the hodge line bundles $\omega_0^\mathrm{Hdg}$ and $\omega^\mathrm{Hdg}$ on $\mathcal{M}_0$
and $\mathcal{M}$, respectively, and prove that these line bundles are isomorphic to the canonical bundles of these
stacks.   We then construct explicit metrics on these bundles, and prove that the restriction of $\omega^\mathrm{Hdg}$ to 
$\mathcal{M}_0$ is isomorphic, as a metrized line bundle, to $\omega_0^\mathrm{Hdg} \otimes \omega_0^\mathrm{Hdg}$.    This  will be a 
key ingredient in the arithmetic adjunction formula proved in \S \ref{s:adjunction}.

Abbreviate $\delta_F=\varpi_1\varpi_2^\sigma-\varpi_2\varpi_1^\sigma$,
where $\{\varpi_1,\varpi_2\}$ is our fixed  $\Z$-basis of $\co_F$.  Note that 
$-\delta_F\delta_F^\sigma = \delta_F^2 =  d_F$, and that $\delta_F\co_F=\mathfrak{D}_F$.  
For an abelian scheme $A\map{}S$ with identity section $e:S\map{}A$, the \emph{co-Lie algebra} of 
$A$ is the locally free $\co_S$-module  
 $$
  \mathrm{coLie}(A/S) = e^*\Omega^1_{A/S} \iso \Hom_{\co_S}(\mathrm{Lie}(A/S),\co_S)
  $$   
 of rank equal to the relative dimension of $A$.  Define the \emph{Hodge bundle} on $\mathcal{M}_0$ by
$$
\omega_0^\mathrm{Hdg} = \wedge^2 \mathrm{coLie}(\mathbf{A}_0^\mathrm{univ}/\mathcal{M}_0)
$$
 where the exterior product is taken in the category of $\co_{\mathcal{M}_0}$-modules, and 
 $\mathbf{A}_0^\mathrm{univ}$ is the universal QM abelian surface over $\mathcal{M}_0$.  
 Define the Hodge bundle on $\mathcal{M}$ by
$$
\omega^\mathrm{Hdg} =\wedge^2 \wedge^2_{\co_F} \mathrm{coLie}(\mathbf{A}^\mathrm{univ}/\mathcal{M})
$$
 where the $\wedge^2$ means  exterior square in the category of $\co_{\mathcal{M}}$-modules 
 and $\wedge^2_{\co_F}$ means exterior square  in the category of $\co_{\mathcal{M}}\otimes_{\Z}\co_F$-modules.  
 The Hodge bundle is essentially the determinant bundle of $\mathrm{coLie}(\mathbf{A}^\mathrm{univ}/\mathcal{M})$.
Indeed,   if  $\mathcal{L}$ is any $\co_{\mathcal{M}_{/\C}}\otimes_\Q F$-module 
 that is locally free of rank two, the splitting  $\C\otimes_\Q F\iso \C \times \C$ induces a decomposition 
 $\mathcal{L}\iso \mathcal{L}^{(1)} \oplus \mathcal{L}^{(2)}$. There is then an isomorphism of line bundles
$$
\wedge^2 \wedge^2_{\co_F} \mathcal{L}  \map{} \wedge^4\mathcal{L}
$$
on $\mathcal{M}_{/\C}$ determined by
$$
(s_1\wedge t_1) \wedge (s_2\wedge t_2) \mapsto s_1\wedge t_1\wedge s_2\wedge t_2
$$
for local sections $s_1,t_1$ of $\mathcal{L}^{(1)}$ and $s_2,t_2$ of $\mathcal{L}^{(2)}$.  
Taking   $\mathcal{L}=\mathrm{coLie}(\mathbf{A}^\mathrm{univ}/\mathcal{M})$ shows
\begin{equation}\label{simple hodge}
\omega^\mathrm{Hdg}_{/\C}  \iso \wedge^4\mathrm{coLie}(\mathbf{A}^\mathrm{univ}/\mathcal{M})_{/\C}.
\end{equation}

\begin{Lem}\label{Lem:hodge split}
Recalling the closed immersion $i:\mathcal{M}_0\map{}\mathcal{M}$ of \S \ref{s:moduli},  
there is an  isomorphism of invertible $\co_{\mathcal{M}_0}$-modules
$$
i^*\omega^\mathrm{Hdg}\iso \omega_0^\mathrm{Hdg}\otimes\omega_0^\mathrm{Hdg}.
$$
\end{Lem}

\begin{proof}
Using the canonical isomorphism
$$
i^*\mathrm{Lie}(\mathbf{A}^\mathrm{univ}/\mathcal{M}) \iso \mathrm{Lie}(\mathbf{A}_0^\mathrm{univ}/\mathcal{M}_0) \otimes_\Z\co_F
$$
of $\co_{\mathcal{M}_0}\otimes_\Z\co_F$-modules,  we deduce that
$$
i^*\mathrm{coLie}(\mathbf{A}^\mathrm{univ}/\mathcal{M}) \iso 
\mathrm{coLie}(\mathbf{A}_0^\mathrm{univ} / \mathcal{M}_0 ) \otimes_\Z\mathfrak{D}_F^{-1}.
$$
If $\mathcal{L}_0$ is an $\co_{\mathcal{M}_0}$-module that is locally free of rank two,  there is an isomorphism
$$
\wedge^2_{\co_F}(\mathcal{L}_0\otimes_\Z\mathfrak{D}_F^{-1})\map{} ( \wedge^2\mathcal{L}_0)  \otimes_\Z  \co_F
$$
defined by 
$$
(s\otimes\delta )\wedge(s' \otimes\delta' ) \mapsto  (s \wedge s' )\otimes  (\delta  \delta' \cdot d_F)
$$
and an isomorphism
$$
\wedge^2\big( ( \wedge^2\mathcal{L}_0)  \otimes_\Z  \co_F\big) \map{} 
(\wedge^2\mathcal{L}_0)\otimes(\wedge^2\mathcal{L}_0) \otimes (\wedge^2\co_F)
$$
defined by
$$
\big( ( s\wedge t) \otimes x \big) \wedge \big( ( s'\wedge t') \otimes x' \big) 
\mapsto    ( s\wedge t)  \otimes ( s' \wedge t' )     \otimes (x\wedge x').
$$
Using the $\Z$-basis $\{\varpi_1 , \varpi_2\}$ of $\co_F$ to identify $\wedge^2\co_F\iso \Z$ via
$$
(a\varpi_1 + b\varpi_2)\wedge(c\varpi_1 + d\varpi_2) \mapsto ad-bc
$$
and applying the above isomorphisms with $\mathcal{L}_0=\mathrm{coLie}(\mathbf{A}_0^\mathrm{univ}/\mathcal{M}_0)$, 
we find isomorphisms
\begin{eqnarray*}
i^*\omega^\mathrm{Hdg}  &\iso& 
\wedge^2 \wedge^2_{\co_F}(   \mathrm{coLie}(\mathbf{A}_0^\mathrm{univ} / \mathcal{M}_0 ) \otimes_\Z\mathfrak{D}_F^{-1} )  \\
& \iso &
\wedge^2\big( ( \wedge^2 \mathrm{coLie}(\mathbf{A}_0^\mathrm{univ}/\mathcal{M}_0 ) ) \otimes_\Z \co_F\big)  \\
&\iso &
\big( \wedge^2\mathrm{coLie}(\mathbf{A}_0^\mathrm{univ}/\mathcal{M}_0 ) \big)  \otimes
 \big(  \wedge^2\mathrm{coLie}(\mathbf{A}_0^\mathrm{univ}/\mathcal{M}_0 )  \big)  \\
&\iso & \omega_0^\mathrm{Hdg} \otimes\omega_0^\mathrm{Hdg} . 
\end{eqnarray*}
\end{proof}

 Let  $\omega_0 = \omega_{\mathcal{M}_0/\Z}$ and   $\omega = \omega_{\mathcal{M}/\Z}$  
 be the canonical bundles on $\mathcal{M}_0$ and $\mathcal{M}$, respectively.   
 There is a canonical morphism of $\co_{\mathcal{M}_0}$-modules $\Omega^1_{\mathcal{M}_0}\map{}\omega_0$, 
 which is an isomorphism when restricted to the smooth locus of $\mathcal{M}_0\map{}\Spec(\Z)$, 
 and hence we may identify $\omega_{0/\C}\iso \Omega^1_{\mathcal{M}_{0/\C}}$.  Similarly there is a 
 canonical morphism of $\co_{\mathcal{M}}$-modules $\Omega^2_{\mathcal{M}} \map{}\omega$, 
 which is an isomorphism over the smooth locus of $\mathcal{M}\map{}\Spec(\Z)$, and so we may identify 
 $\omega_{/\C}\iso \Omega^2_{\mathcal{M}_{/\C}}$.

\begin{Prop}\label{Prop:kodaira-spencer}
There are isomorphisms 
$$
\omega_0\iso \omega_0^\mathrm{Hdg} \qquad \omega\iso \omega^\mathrm{Hdg}
$$
of line bundles on $\mathcal{M}_0$ and $\mathcal{M}$, respectively.
\end{Prop}

\begin{proof}
The first isomorphism is  \cite[Proposition 3.2]{kudla04a} (compare also with \cite[\S 1.0]{katz78}).  
We will give a slightly different construction, which is better suited to the calculations to be performed in the 
proof of Proposition \ref{Prop:canonical metrics}.  Suppose that $U\map{}\mathcal{M}_0$ is an \'etale morphism 
with $U$ a scheme, and that $U$ is smooth over $\Spec(\Z)$.  That is to say, $U$ is an \'etale open subset of the 
smooth locus of $\mathcal{M}_0$.  The morphism $U\map{}\mathcal{M}_0$ determines  a  principally polarized 
QM  abelian surface  $(A_0,i_0,\lambda_0)$  over  $U$, and the order  $\co_{B_0}$ acts naturally on the right on 
each of the $\co_U$-modules $\mathrm{coLie}(A_0/U)$,  $\mathrm{Lie}(A^\vee_0/U)$, and 
$H^1_{\mathrm{DR}}(A_0/U)$.   Given an $\co_{B_0}  \otimes_\Z \co_U$-linear map
$$
\phi:  \mathrm{coLie}(A_0/U) \map{} \mathrm{Lie}(A^\vee_0/U) 
$$
we will attach to $\phi$ a skew-symmetric $\co_U$-bilinear pairing $Q_\phi$ on $\mathrm{coLie}(A_0/U)$ 
in such a way that  the construction $\phi\mapsto Q_\phi$ determines an isomorphism   
\begin{equation}\label{screw pairing}
 \Hom_{\co_{B_0}\otimes_\Z\co_U }(  \mathrm{coLie}(A_0/U), \mathrm{Lie}(A^\vee_0/U) ) 
 \iso \Hom_{\co_U}(\wedge^2\mathrm{coLie} (A_0/U) , \co_U ).
\end{equation}
 Indeed, from the proof of  \cite[Proposition 3.2]{kudla04a}  one deduces the existence of  a unique $\co_U$-linear map
$$
\Phi: \mathrm{coLie}(A_0/U) \map{} \mathrm{Lie}(A_0/U) 
$$
satisfying $\Phi(x\cdot b) = b^\iota\cdot \Phi(x)$ for all $b\in\co_{B_0}$ and making  the diagram 
$$
\xymatrix{
{  \mathrm{coLie}(A_0/U)   } \ar[d]_{\Phi}\ar[drr]^{\phi}  \\
{ \mathrm{Lie}(A_0/U) } \ar[r]_{s\cdot }   &  {  \mathrm{Lie}(A_0/U)  } \ar[r]_{\lambda_0}  & {  \mathrm{Lie}(A_0^\vee/U) }
}
$$
commute, and one shows that the pairing 
$$
Q_\phi(x,y)=\langle x, \Phi(y) \rangle
$$
 has the desired properties.  Here $s\in\co_{B_0}$ is the trace free element chosen in the definition of the 
 involution $b^*=s^{-1}b^\iota s$ on $B_0$, $b\mapsto b^\iota$ is the main involution, and the pairing 
 $\langle\cdot, \cdot\rangle$ is the tautological pairing between $\mathrm{coLie}(A_0/U)$ and $\mathrm{Lie}(A_0/U)$.

Denote  by  $T_{U/\Z}=\Hom_{\co_U}( \Omega^1_{U/\Z} , \co_U )$ the  tangent sheaf of $U$.   For  
any local section $D$ of $T_{U/\Z}$ the Gauss-Manin connection determines an $\co_{B_0}$-linear 
morphism of coherent $\co_U$-modules
$$
\nabla(D): H^1_{\mathrm{DR}}(A_0/U) \map{} H^1_{\mathrm{DR}}(A_0/U).
$$
Combining this with the short exact sequence
$$
0\map{} \mathrm{coLie}(A_0/U) \map{} H^1_{\mathrm{DR}}(A_0/U) \map{} \mathrm{Lie}(A^\vee_0/U)\map{}0
$$
of Hodge theory yields the \emph{Kodaira-Spencer isomorphism}
\begin{equation}\label{simple KS}
T_{U/\Z} \map{} \Hom_{\co_{B_0} \otimes_\Z \co_U }(  \mathrm{coLie}(A_0/U), \mathrm{Lie}(A^\vee_0/U) )
\end{equation}
defined by sending $D$ to the composition 
$$
\mathrm{coLie}(A_0/U) \map{} H^1_{\mathrm{DR}}(A_0/U) \map{\nabla(D)} 
H^1_{\mathrm{DR}}(A_0/U) \map{}  \mathrm{Lie}(A^\vee_0/U).
$$
More details on this construction can be found in \cite[\S 1]{katz72}.  Composing   
the  Kodaira-Spencer  isomorphism with the isomorphism (\ref{screw pairing})   and dualizing yields an isomorphism
$$
\wedge^2\mathrm{coLie}(A_0/U) \map{} \Omega^1_{U/\Z},
$$
and thus over the smooth locus of $\mathcal{M}_0$ there is an isomorphism 
$\omega_0^\mathrm{Hdg}\iso \omega_0$ of line bundles.   Using the regularity of $\mathcal{M}_0$ 
and the fact that the nonsmooth locus of $\mathcal{M}_0$ lies in codimension two, one shows 
that this isomorphism extends uniquely across all of $\mathcal{M}_0$.

Now suppose that $U\map{}\mathcal{M}$ is an \'etale open subset of the smooth locus of $\mathcal{M}$, and  let  
$(A,i,\lambda)$ be the corresponding  $\mathfrak{D}_F^{-1}$-polarized QM abelian fourfold over  $U$.      
We claim that, as above, there is an isomorphism
\begin{eqnarray}\lefteqn{
\Hom_{\co_{B}\otimes_\Z\co_U}(  \mathrm{coLie}(A/U), \mathrm{Lie}(A^\vee/U) ) \nonumber  } \\
& & \iso  \Hom_{\co_F\otimes_\Z\co_U}( \wedge^2_{\co_F}\mathrm{coLie} (A/U) , \co_F\otimes_\Z\co_U ),
\label{screw pairing II}
\end{eqnarray}
which we will denote by $\phi\mapsto Q_\phi$.  The construction of $Q_\phi$ is essentially the same as 
that considered earlier.   There is a unique perfect  $\co_F$-bilinear pairing 
$$
\mathrm{coLie}(A/U) \otimes_{\co_U} \mathrm{Lie}(A/U) \map{} \mathfrak{D}_F^{-1}\otimes_\Z\co_U
$$
such that the composition
$$
\mathrm{coLie}(A/U) \otimes_{\co_U}  \mathrm{Lie}(A/U) \map{} 
\mathfrak{D}_F^{-1}\otimes_\Z\co_U \map{\mathrm{Tr}_{F/\Q} \otimes\mathrm{id} } \co_U
$$
is the tautological pairing.   This pairing defines the first arrow in the   $\co_F$-bilinear pairing 
$$
\mathrm{coLie}(A/U) \otimes_{\co_U} ( \mathrm{Lie}(A/U)\otimes_{\co_F}\mathfrak{D}_F^{-1} ) 
 \map{} \mathfrak{D}_F^{-2}\otimes_\Z\co_U \map{d_F\otimes \mathrm{id}}\co_F\otimes_\Z\co_U,
$$
which we denote by $\langle \cdot , \cdot \rangle $.   View the polarization $\lambda$ 
as an isomorphism  $$A\otimes_{\co_F} \mathfrak{D}_F^{-1} \iso A^\vee .$$  For each $\co_B\otimes_\Z \co_U$-linear
 $$
 \phi:  \mathrm{coLie}(A/U) \map{}  \mathrm{Lie}(A^\vee/U) 
 $$
  there is a unique $\co_U$-linear map
$$
\Phi: \mathrm{coLie}(A/U) \map{} \mathrm{Lie}(A/U) \otimes_{\co_F}\mathfrak{D}_F^{-1}
$$
satisfying $\Phi(x\cdot b) = b^\iota\cdot \Phi(x)$ for all $b\in\co_{B}$ and making  the diagram 
$$
\xymatrix{
{  \mathrm{coLie}(A/U)   } \ar[d]_{\Phi}\ar[drr]^{\phi}  \\
{ \mathrm{Lie}(A/U) \otimes_{\co_F}\mathfrak{D}_F^{-1} } \ar[r]_{s\cdot }   &
  {  \mathrm{Lie}(A/U)\otimes_{\co_F}\mathfrak{D}_F^{-1}  } \ar[r]_{\lambda}  & {  \mathrm{Lie}(A^\vee/U) }
}
$$
commute, and  the pairing 
$$
Q_\phi(x,y)=\langle x, \Phi(y) \rangle
$$
  has the desired properties.    The $\Z$-module homomorphism
$$
\co_F\map{\delta_F^{-1}}\mathfrak{D}_F^{-1}\map{\mathrm{Tr}_{F/\Q}} \Z
$$
induces an isomorphism
$$
 \Hom_{\co_F\otimes_\Z\co_U}( \wedge^2_{\co_F}\mathrm{coLie} (A/U) , \co_F\otimes_\Z\co_U )  \map{}
    \Hom_{\co_U} (  \wedge^2_{\co_F}\mathrm{coLie} (A/U) , \co_U ),
$$
which when composed with (\ref{screw pairing II}) and the Kodaira-Spencer isomorphism
$$
T_{U/\Z} \map{} \Hom_{\co_{B}\otimes_\Z \co_U}(\mathrm{coLie}(A/U) , \mathrm{Lie}(A^\vee/U))
$$
 yields an isomorphism
$$
\wedge^2_{\co_F}  \mathrm{coLie} (A/U)  \iso \Omega^1_{U/\Z}.
$$
Thus over the smooth locus of $\mathcal{M}$ there is an isomorphism $\omega^\mathrm{Hdg}\iso \omega$
 which, again using the regularity of $\mathcal{M}$ and the fact that the nonsmooth locus of $\mathcal{M}$ 
 lies in codimension two, extends uniquely across all of $\mathcal{M}$.
\end{proof}

\begin{Rem}
In the sequel we freely identify $\omega_0$ with $\omega_0^\mathrm{Hdg}$ and $\omega$ with $\omega^\mathrm{Hdg}$ 
using the isomorphisms of Proposition \ref{Prop:kodaira-spencer}. 
\end{Rem}

Let $\sigma_0$ and $\tau_0$ denote the coordinate functions on $\C^2$.  For each $z_0\in X_0$ 
the holomorphic $2$-form $d\sigma_0\wedge d\tau_0$ on $\C^2$ defines a holomorphic $2$-form 
on the QM abelian surface $\mathbf{A}_{0,z_0}$ constructed in \S \ref{s:generic fiber}, hence an 
element of the stalk at $z_0$ of the pullback of $\omega^\mathrm{Hdg}_{0/\C}$ to $X_0$.  As $z_0$ varies 
$$
\epsilon_0=d\sigma_0\wedge d\tau_0
$$ 
defines a nonvanishing global section of the pullback of $\omega_{0/\C}$ to $X_0$.  Similarly, if 
$\sigma_1, \tau_1,\sigma_2,\tau_2$ are the coordinate functions on $\C^2\times\C^2$, then 
$d\sigma_i\wedge d\tau_i$ for $i=1,2$ defines a section of the pullback of  
$\wedge^2_{\co_F}\mathrm{coLie}(\mathbf{A}^\mathrm{univ}/\mathcal{M})_{/\C}$ to $X$, and hence 
$$
\epsilon =  \delta_F^{-1} \cdot  (d\sigma_1\wedge d\tau_1) \wedge (d\sigma_2\wedge d\tau_2)
$$ 
defines a section of the pullback of $\omega^\mathrm{Hdg}_{/\C}$ to $X$.  Tracing through the 
above constructions shows that the isomorphism of Lemma  \ref{Lem:hodge split} satisfies
\begin{equation}\label{hodge pullback}
i^*\epsilon\mapsto   d_F \cdot \epsilon_0\otimes\epsilon_0 .
\end{equation}

 We now metrize $\omega_0^\mathrm{Hdg}$ and $\omega^\mathrm{Hdg}$.   For a point $z_0\in \mathcal{M}_0(\C)$ 
 and a vector $u_0$ in the fiber of $\omega_0^\mathrm{Hdg}$ at $z_0$,  define
\begin{equation}\label{metric I}
 ||u_0||^2_{z_0} = \frac{1}{2^4\pi^3 e^{\gamma_{\mathrm{Euler}}}}
  \left| \int_{\mathbf{A}^\mathrm{univ}_{0,z_0}(\C)} u_0 \wedge \overline{u}_0 \ \right| 
 \end{equation}
 where $\gamma_{\mathrm{Euler}} = 0.5772\ldots$ is Euler's constant.  Denote by 
 $\widehat{\omega}_0$ the line bundle $\omega_0^\mathrm{Hdg}$ equipped with the above metric, 
 and note that our $\widehat{\omega}_0$ is precisely the metrized Hodge bundle constructed 
 by Kudla-Rapoport-Yang in \cite[Definition 3.4]{kudla04a}.  As in \cite[(3.15)]{kudla04a} the 
 explicit construction of $\mathbf{A}_{z_0}$ given in \S \ref{s:generic fiber},  together with the easy calculation
$$
\mathrm{Vol}(M_2(\R)/\co_{B_0}) = \mathrm{disc}(B_0)
$$
(the volume is with respect to the Haar measure on $M_2(\R)$ normalized so that 
$\mathrm{Vol}(M_2(\R)/M_2(\Z))=1$)  shows that after pulling back $\omega_0^\mathrm{Hdg}$ to $X_0$
\begin{equation}\label{hodge norm one}
||\epsilon_0||_{z_0}^2 =   \frac{ 1 }{ 2^2 \pi^3 e^{\gamma_{\mathrm{Euler}}}} \mathrm{Im}(z_0)^2 \mathrm{disc}(B_0).
\end{equation}
We metrize $\omega^\mathrm{Hdg}$ in a similar way.  For a point $z\in\mathcal{M}(\C)$ and a vector $u$ 
in the fiber of $\omega^\mathrm{Hdg}$ at $z$,  we use the isomorphism (\ref{simple hodge}) to view $u$ as a 
holomorphic $4$-form on the QM abelian fourfold $\mathbf{A}_z^\mathrm{univ}$ and define
\begin{equation}\label{metric II}
 ||u||^2_{z} = \frac{1}{2^8\pi^6 e^{2\gamma_{\mathrm{Euler}}}} \left| \int_{\mathbf{A}^\mathrm{univ}_z(\C)} u\wedge \overline{u} \ \right| .
 \end{equation}
Pulling back $\omega^\mathrm{Hdg}$ to $X$, and using the volume calculation
$$
\mathrm{Vol}\big( (M_2(\R)\times M_2(\R) ) /\co_{B}\big) = \mathrm{disc}(B_0)^2 \cdot d_F^2
$$
and the construction  of $\mathbf{A}_z$ of \S \ref{s:generic fiber}, we then compute
\begin{equation}\label{hodge norm two}
||\epsilon||_z^2 =   \frac{ d_F^2 }{ 2^4 \pi^6 e^{2 \gamma_{\mathrm{Euler}}}} 
\mathrm{Im}(z_1)^2 \mathrm{Im}(z_2)^2 \mathrm{disc}(B_0)^2.
\end{equation}
Comparing (\ref{hodge pullback}), (\ref{hodge norm one}), and (\ref{hodge norm two}) we find that the 
isomorphism of Lemma \ref{Lem:hodge split} preserves the metrics defined above.  That is to say, the 
isomorphism of Lemma \ref{Lem:hodge split} induces an isomorphism of metrized line bundles
\begin{equation}\label{metric splitting}
i^* \widehat{\omega} \iso \widehat{\omega}_0 \otimes\widehat{\omega}_0 .
\end{equation}

\begin{Prop}\label{Prop:canonical metrics}
The metrics (\ref{metric I}) and (\ref{metric II}) on $\omega_0^\mathrm{Hdg}$ and $\omega^\mathrm{Hdg}$ 
induce metrics on the sheaves of  top degree holomorphic differential forms $\Omega^1_{X_0}$ and $\Omega^2_X$,
and these metrics are determined by the formulas
$$
 ||dz_0||^2 =   \frac{ 1 } {\pi  e^{ \gamma_{\mathrm{Euler}}}  } \cdot \mathrm{Im}(z_0)^2 \cdot  \mathrm{disc}(B_0)  
$$
and
$$
 ||dz_1 \wedge dz_2 ||^2 =  \frac{ d_F }{ \pi^2 e^{ 2 \gamma_{\mathrm{Euler}}}}   
 \cdot \mathrm{Im}(z_1)^2  \mathrm{Im}(z_2)^2 \cdot  \mathrm{disc}(B_0)^2 ,
$$
respectively.
\end{Prop}

\begin{proof}
Return to the notation of \S \ref{s:generic fiber}, and  in particular recall the family of principally polarized 
QM abelian surfaces  $(A_{0,z_0}, i_{0,z_0} , \lambda_{0,z_0})$ parametrized by $z_0=x_0+iy_0 \in X_0$.   
 By Hodge theory there is an isomorphism of short exact sequences
$$
\xymatrix{
{  0 }  \ar[r]  &   { \mathrm{coLie}(A_{0,z_0}/\C) }  \ar[r] \ar[d] & { H^1_{\mathrm{DR}}(A_{0,z_0}/\C ) } \ar[r]\ar[d]
 &  { H^1(A_{0,z_0}, \co_{A_{0,z_0}}) }  \ar[r] \ar[d]&  { 0 }   \\
 {  0 }  \ar[r]  &   { H^{1,0} (A_{0,z_0}/\C) }  \ar[r]  & { H^1_{\mathrm{DR}}(A_{0,z_0}/\C ) } \ar[r]
 &  { H^{0,1} (A_{0,z_0}/\C) }  \ar[r] &  { 0 },
}
$$
and the cohomology of the exponential  sequence
$$
0\map{}2 \pi i\Z \map{}\co_{A_{0,z_0}} \map{f\mapsto e^f} \co_{A_{0,z_0}}^\times\map{}0
$$
induces the first isomorphism in 
$$
\mathrm{Lie}(A^\vee_{0,z_0}/\C) \iso  H^1(A_{0,z_0} , \co_{A_{0,z_0}}) \iso H^{0,1} (A_{0,z_0}/\C).
$$
 Recall that we have fixed an isomorphism $B_0\otimes_\Q\R\iso M_2(\R)$ and defined an isomorphism of 
 $\R$-vector spaces $\rho_{0,z_0}:M_2(\R)\map{} \C^2$  by 
$$
\rho_{0,z_0}(A)= A\cdot \left[ \begin{matrix} z_0 \\ 1 \end{matrix}\right].
$$
If we give $M_2(\R)$ the complex structure under which multiplication by $i$ is equal to right multiplication by 
$$
J_{z_0} = \frac{1}{y_0} \left(\begin{matrix}  y_0 & x_0 \\  & 1 \end{matrix}\right)
 \left(\begin{matrix}   & -1 \\ 1 &  \end{matrix}\right)
  \left(\begin{matrix}  1 & -x_0 \\  & y_0 \end{matrix}\right),
$$
then $\rho_{0,z_0}$ is an isomorphism of complex vector spaces.  By definition of $A_{0,z_0}$ 
there are isomorphisms of smooth manifolds
\begin{equation}\label{trivialization}
M_2(\R)/\co_{B_0} \map{\rho_{0,z_0}} \C^2/\rho_{0,z_0}(\co_{B_0})   \iso A_{0,z_0}.
\end{equation}
Let $\sigma_0$ and $\tau_0$ be the standard coordinate functions on $\C^2$, so that $\{d\sigma_0 , d\tau_0\}$ 
is a basis for  $H^{1,0}(A_{0,z_0}/\C)$ and $\{ d\overline{\sigma}_0, d\overline{\tau}_0\}$ is a basis for 
$H^{0,1}(A_{0,z_0}/\C )$.   Under the  isomorphisms (\ref{trivialization})  these differentials correspond 
to the smooth $1$-forms on $M_2(\R)/\co_{B_0}$
\begin{eqnarray}\label{coordinates}
d\sigma_0 &=& z_0 da_{11} + da_{12} \\
d\tau_0 &=& z_0 da_{21} + da_{22} \nonumber \\
d\overline{\sigma}_0 &=& \overline{z}_0 da_{11} + da_{12}  \nonumber \\
d\overline{\tau}_0 &=& \overline{z}_0 da_{21} + da_{22}  \nonumber
\end{eqnarray}
where $a_{ij}$ are the usual coordinates on $M_2(\R)$.  The basis of $\mathrm{Lie}(A_{0,z_0}/\C)$ 
 dual to the basis $\{d\sigma_0, d\tau_0\}$ of $\mathrm{coLie}(A_{0,z_0}/\C)$ is $\{e,f\}$  where
$$
e  = \left(\begin{matrix} 0 & 1\\ 0& 0\end{matrix}\right) 
\qquad
f  = \left(\begin{matrix} 0 & 0 \\ 0& 1 \end{matrix}\right) .
$$

Recall from \cite[\S 3.1]{howardA} the alternating form $\psi_0$ on $B_0$ defined by 
$$
\psi_0(x,y) = \frac{1}{\mathrm{disc}(B_0)} \mathrm{Tr}(xsy^*) =  \frac{1}{\mathrm{disc}(B_0)} \mathrm{Tr}(xy^\iota s).
$$
Extend $\psi_0$ $\R$-linearly to $B_0\otimes_\Q\R\iso M_2(\R)$,  and define a Hermitian 
form on $M_2(\R)$ (with respect to the complex structure determined by $J_{z_0}$)
$$
 H_{z_0}(x,y) = \pm \big( \psi_0(x J_{z_0} , y ) + i\psi_0(x,y) \big)
$$
where the sign  is chosen so that $H_{z_0}$ is positive definite.  As $z_0$ varies the sign for which 
this holds is constant on each connected component of $X_0$, and is different on the two components; by 
replacing $s$ by $-s$ if necessary, we may assume that the sign is $+1$ on the component with 
$\mathrm{Im}(z_0) >0$, and to simplify notation we  assume from now on that $\mathrm{Im}(z_0)>0$.  
Using (\ref{trivialization}) to identify 
$$
M_2(\R)\iso  \C^2\iso \mathrm{Lie}(A_{0,z_0}/\C),
$$ 
we view $H_{z_0}$ as a nondegenerate Riemann form on $   \mathrm{Lie}(A_{0,z_0}/\C) $.  
If we identify $H^{0,1} (A_{0,z_0}/\C)$ with the space of conjugate linear functionals on 
$\mathrm{Lie}(A_{0,z_0}/\C)$  then the  $\C$-linear isomorphism
$$
\mathrm{Lie}(A_{0,z_0} /\C)  \map{\lambda_{0,z_0}}    \mathrm{Lie}(A^\vee_{0,z_0} /\C )
$$
factors as
$$
 \mathrm{Lie}(A_{0,z_0} /\C)  \map{} H^{0,1}(A_{0,z_0}/\C)  \map{}  \mathrm{Lie}(A^\vee_{0,z_0} /\C )
$$
where the first arrow  takes the vector $v$ to the conjugate linear functional 
$$
w\mapsto \pi  H_{z_0}(v,w).
$$   
The factor of $\pi = (2\pi i)/(2 i)$ appears because of the ``experimental error" of $2i$ at the bottom of 
\cite[p.~87]{mumford70} and the fact that our exponential sequence is shifted from Mumford's by a 
factor of $2\pi i$.  Direct calculation  now shows that
$$
\begin{array}{ccc}
\pi H_{z_0}(se,e) =  & 0  & =   -\pi y_0^{-1} \cdot d\overline{\tau}_0(e) \\
\pi H_{z_0}( se,  f   ) = & -\pi  y_0^{-1} & =   -\pi y_0^{-1} \cdot d\overline{\tau}_0(f) \\
\pi H_{z_0}( sf,e ) = &   -\pi y_0^{-1} & =  -\pi y_0^{-1} \cdot  d\overline{\sigma}_0(e) \\
\pi H_{z_0}( sf,f  )   = &  0  & =   -\pi y_0^{-1} \cdot d\overline{\sigma}_0(e) .
\end{array}
$$
This implies that the composition 
$$
  \mathrm{Lie}(A_{0,z_0}/\C )  \map{ s\cdot }   \mathrm{Lie}(A_{0,z_0}/\C )  \map{\lambda_{0,z_0} }
\mathrm{Lie}(A^\vee_{0,z_0} /\C )  
$$
used in the proof of Proposition \ref{Prop:kodaira-spencer} satisfies
$$
e  \mapsto \frac{-\pi}{ y_0} \cdot d\overline{\tau}_0
\qquad
 f \mapsto \frac{-\pi} { y_0 } \cdot d\overline{\sigma}_0 .
$$

We next compute the Gauss-Manin connection 
$$
\nabla(d/dz_0) : H^1_{\mathrm{DR}}(A_{0,z_0}/\C) \map{}  H^1_{\mathrm{DR}}(A_{0,z_0}/\C).
$$
 Differentiating the equations (\ref{coordinates}) with respect to $z_0$  shows that  $\nabla(d/dz_0)$ satisfies 
 \begin{eqnarray*}
d\sigma_0 &\mapsto &  \frac{1}{z_0-\overline{z}_0} (d\sigma_0 - d\overline{\sigma}_0)  \\
d\tau_0  &\mapsto & \frac{1}{z_0-\overline{z}_0} (d\tau_0 - d\overline{\tau}_0)   .
\end{eqnarray*}
The image 
$$
\phi_{z_0}\in\Hom_{B_0\otimes_\Q\C} ( \mathrm{coLie}(A_{0,z_0}/\C )  , \mathrm{Lie}(A^\vee_{0,z_0}/\C ) )
$$ 
of $d/dz_0$ under the Kodaira-Spencer isomorphism (\ref{simple KS}) is equal to the composition
$$
H^{1,0}(A_{0,z_0}/\C) \map{} H^1_{\mathrm{DR}}(A_{0,z_0}/\C) \map{\nabla(d/dz_0)}  
H^1_{\mathrm{DR}}(A_{0,z_0}/\C) \map{} H^{0,1}(A_{0,z_0}/\C)
$$
and so has the explicit form
$$
\phi_{z_0}( d\sigma_0) = \frac{-1}{2i y_0} \cdot d\overline{\sigma}_0 
\qquad
\phi_{z_0}( d\tau_0) = \frac{-1}{2i y_0} \cdot d\overline{\tau}_0 .
$$
The map $\Phi_{z_0}$ making the diagram
$$
\xymatrix{
{  \mathrm{coLie}(A_{0,z_0}/\C )  } \ar[rrd]^{\phi_{z_0}} \ar[d]_{\Phi_{z_0}}   \\ 
{  \mathrm{Lie}( A_{0,z_0}/\C )  } \ar[r]_{s\cdot}   & {  \mathrm{Lie}(A_{0,z_0}/\C)  } 
\ar[r]_{\lambda_{0,z_0}}  & {  \mathrm{Lie} (A^\vee_{0,z_0} /\C ) }
}
$$
commute is then 
$$
\Phi_{z_0}(d\sigma_0) = \frac{1}{2\pi i} f \qquad \Phi_{z_0}(d\tau_0) = \frac{1}{2\pi i} e,
$$
and  the pairing $Q_{\phi_{z_0}}$ on $\mathrm{coLie}(A_{0,z_0}/\C )$ defined in the proof of 
Proposition  \ref{Prop:kodaira-spencer}   is completely determined by the single value
$$
Q_{\phi_{z_0}}( d\sigma_0, d\tau_0 )  =  \langle d\sigma_0 , \Phi_{z_0}(d\tau_0) \rangle =  
\frac{ 1 }{2 \pi i }  \langle d\sigma_0 , e \rangle =  \frac{1}{2\pi i} .
$$
We deduce that the isomorphism $\omega_0\iso \omega_0^\mathrm{Hdg}$ of Proposition 
\ref{Prop:kodaira-spencer}, when pulled back to an isomorphism of line bundles on $X_0$, satisfies 
$$
dz_0  \mapsto  2\pi i \cdot d\sigma_0 \wedge d\tau_0 .
$$
Applying (\ref{hodge norm one}) shows that
$$
||dz_0 ||^2 = 4\pi^2 \cdot ||\epsilon_0||^2 =
 \frac{1}{\pi e^{ \gamma^{\mathrm{Euler}}}} \cdot \mathrm{Im}(z_0)^2 \cdot \mathrm{disc}(B_0)
$$
as desired.  Similar calculations show that the isomorphism $\omega\iso \omega^\mathrm{Hdg}$ of 
Proposition \ref{Prop:kodaira-spencer}, when pulled back to an isomorphism of line bundles on $X$, satisfies 
$$
dz_1\wedge dz_2  \mapsto \frac{  (2 \pi i)^2}{ \delta_F} (d\sigma_1\wedge d\tau_1) \wedge (d\sigma_2\wedge d\tau_2),
$$
and hence  (\ref{hodge norm two}) implies
$$
||dz_1\wedge dz_2 ||^2 = \frac{16 \pi^4}{d_F} \cdot ||\epsilon ||^2 
= \frac{d_F}{\pi^2 e^{ 2 \gamma^{\mathrm{Euler}}}}\cdot \mathrm{Im}(z_1)^2 \mathrm{Im}(z_2)^2 \mathrm{disc}(B_0)^2.
$$

\end{proof}


\section{The adjunction formula}
\label{s:adjunction}


In this section we prove the \emph{arithmetic adjunction formula}, Theorem \ref{Thm:adjunction}.
This theorem gives an explicit formula for the linear functional (\ref{arithmetic degree II}) evaluated at a horizontal 
irreducible arithmetic cycle  on $\mathcal{M}$ intersecting $\mathcal{M}_0$ improperly 
(\emph{i.e.}~completely contained in  $\mathcal{M}_0$).   
This formula is one the main ingredients in the proof of Theorem \ref{Big result}.

We return to the notation of \S \ref{s:generic fiber}.    Fix a totally positive $v\in F\otimes_\Q\R$ and write 
$(v_1,v_2)$ for the image of $v$ in $\R\times \R$.  Similarly, for any $\gamma\in\Gamma$,  write 
$(\gamma_1,\gamma_2)$ for the image of $\gamma$ in $G_0(\R)\times G_0(\R)$.  For   an irreducible 
horizontal cycle $\mathcal{D}$ of codimension two on $\mathcal{M}$,   define a Green current for $\mathcal{D}$
$$
\Xi(\mathcal{D} ,v ) = \sum_{P\in \mathcal{D} (\C)} e_P^{-1}   \sum_{\gamma\in \Gamma} \mathbf{g}(\gamma  x, v ).
$$
On the right  $x\in X$ is any point lying above $P\in\mathcal{M}(\C)$ under the  orbifold presentation (\ref{orbifold}) of $\mathcal{M}(\C)$.
Denote by
\begin{equation}\label{augment}
\widehat{\mathcal{D}}(v)\in \widehat{\mathrm{CH}}^2(\mathcal{M})  
\end{equation}
the arithmetic cycle class of the pair $(\mathcal{D}, \Xi(\mathcal{D} , v) )$.  Given a metrized line bundle 
$\widehat{\mathcal{F}}$ on $\mathcal{M}_0$ and an irreducible cycle $j : \mathcal{D} \map{} \mathcal{M}_0$ 
of codimension one,  the \emph{Arakelov degree}  
\begin{equation}\label{KRY degree}
\widehat{\deg} (\mathcal{D} , j^*\widehat{\mathcal{F}})
\end{equation} 
is defined in \cite[Chapter 2]{KRY}, and  the \emph{Arakelov height}
$$
h_{\widehat{\mathcal{F}}} : Z^1(\mathcal{M}_0) \map{}\R
$$
is defined by linearly extending  $\mathcal{D}\mapsto \widehat{\deg} (\mathcal{D} , j^*\widehat{\mathcal{F}})$ 
to all codimension one cycles with rational coefficients.   If instead $\widehat{\mathcal{F}}$ is a metrized line 
bundle on $\mathcal{M}$ and $j : \mathcal{D} \map{}\mathcal{M}$ is an irreducible cycle of codimension two, then 
$\widehat{\deg} (\mathcal{D} , j^*\widehat{\mathcal{F}})$ is defined in the same way as (\ref{KRY degree}), and 
$$
h_{\widehat{\mathcal{F}}} : Z^2(\mathcal{M}) \map{}\R
$$
is the $\Q$-linear extension of  $\mathcal{D}\mapsto \widehat{\deg} (\mathcal{D} , j^*\widehat{\mathcal{F}}).$

\begin{Lem}\label{Lem:autos}
Suppose $w\in X$ and $\gamma\in \Gamma$ satisfy both $w\in X_0$ and $\gamma w\in X_0$.  Then $\gamma\in\Gamma_0$.
\end{Lem}

\begin{proof}
Pick any two points $w,w' \in X_0$ and let $P_0=(\mathbf{A}_{0,w},\lambda_{0,w})$ and 
$P_0'=(\mathbf{A}_{0,w'},\lambda_{0,w'})$ be the objects of $\mathcal{M}_0(\C)$ constructed in \S \ref{s:generic fiber}.  
Thus there is a canonical bijection 
\begin{equation}\label{gamma inc a}
\{\gamma_0\in \Gamma_0 : \gamma_0 w=w' \} \iso \mathrm{Iso}_{\mathcal{M}_0(\C)}( P_0, P_0' ).
\end{equation}
If we then let $P$ and $P'$ be the images of $P_0$ and $P_0'$ in $\mathcal{M}(\C)$, there  is a canonical bijection 
\begin{equation}\label{gamma inc b}
\{\gamma\in \Gamma : \gamma w=w' \} \iso \mathrm{Iso}_{\mathcal{M}(\C)}(P, P').
\end{equation}
According to  \cite[Lemma 3.1.1]{howardA} the evident function from the right hand side of 
(\ref{gamma inc a}) to the right hand side of (\ref{gamma inc b}) is a bijection, and hence so is the evident function
$$
\{\gamma_0\in \Gamma_0 : \gamma_0 w=w' \}\map{} \{\gamma\in \Gamma : \gamma w=w' \}. 
$$
\end{proof}

For any $z_0\in X_0$  define
$$
\vartheta_u(z_0) = \sum_{ \substack{  \gamma \in \Gamma_0\backslash \Gamma \\ \gamma\not\in\Gamma_0 } }  
g_u^0(\gamma_1 z_0,\gamma_2 z_0).
$$
The preceding lemma implies that $\gamma_1 z_0 \not=\gamma_2 z_0$ in each term on the right, so that  
each term in the infinite sum is defined.  Using the methods of \cite[\S 6.5--6.6]{hejhal} and the rapid decay of 
$g_u^0$ away from the diagonal \cite[Remark 7.3.2]{KRY} one can show that the summation   converges 
uniformly on compact subsets of $X_0$, and so defines a smooth function on the orbifold $[\Gamma_0\backslash X_0]$.   
For a positive $u\in\R$, denote by $\widehat{\co}_{\mathcal{M}_0}(u)$ the structure sheaf of $\mathcal{M}_0$, 
endowed with the metric defined by
$$
-\log ||1||^2_P = \vartheta_u(P)
$$
for every $P\in\mathcal{M}_0(\C)$.  Here $1$ denotes the constant function $1$ on the orbifold $\mathcal{M}_0(\C)$.      
For any irreducible cycle $\mathcal{D}_0$ on $\mathcal{M}_0$ we have, from the definition of Arakelov height, the relation
\begin{equation}\label{arakelov height}
h_{\widehat{\co}_{\mathcal{M}_0}(u) }(\mathcal{D}_0) = \frac{1}{2} \sum_{P\in\mathcal{D}_0(\C)}e_P^{-1} \vartheta_u(P).
\end{equation}

We next metrize the line bundle
$$
 \mathcal{L} \define \co_{\mathcal{M}}(\mathcal{M}_0) 
$$ 
on $\mathcal{M}$. If we denote by $s$ the constant function $1$ on $\mathcal{M}$ viewed as a 
global section of $\mathcal{L}$,   there is a unique smooth metric $||\cdot||$  on $\mathcal{L}$ satisfying 
\begin{equation}
\label{green metric}
-\log || s  ||_P^2 =   \log(4u \cdot d_F \mathrm{disc}(B_0)) + \sum_{ \gamma\in \Gamma_0\backslash \Gamma}  
g^0_{u}( \gamma_1 x_1 , \gamma_2  x_2) 
\end{equation}
for every $P\in\mathcal{M}(\C)\smallsetminus \mathcal{M}_0(\C)$ and $x\in X$ lying above $P$.  
To see that the metric extends smoothly across $\mathcal{M}_0(\C)$ one uses \cite[(7.3.16)]{KRY} 
to show that near a point of $X_0$ the right hand side has the form 
$$
g_u^0(x_1,x_2)+\mathrm{smooth} = -\log|x_1-x_2|^2 + \mathrm{smooth}.
$$
Let $\widehat{\mathcal{L}}(u)$ denote  the line bundle $\mathcal{L}$ on $\mathcal{M}$ endowed with the above metric.    
The pullback $ i^*\mathcal{L}$ is the \emph{normal bundle} of the closed immersion $i:\mathcal{M}_0\map{}\mathcal{M}$  
and the classical adjunction formula \cite[Theorem 6.4.9]{liu} provides a canonical isomorphism
\begin{equation}\label{class adjunction}
\omega_0 \iso i^*\mathcal{L} \otimes i^*\omega
\end{equation}
in which $\omega_0$ and $\omega$ are the canonical bundles on $\mathcal{M}_0$ and $\mathcal{M}$ as in \S \ref{s:hodge}.  
The following proposition is our first form of the arithmetic adjunction formula.

\begin{Prop}\label{Prop:first adjunction}
There is an isomorphism of metrized line bundles on $\mathcal{M}_0$
\begin{equation}\label{class adjunction III}
\widehat{\omega}_0 \otimes\widehat{\co}_{\mathcal{M}_0} (u) \iso i^*\widehat{\mathcal{L}}(u) \otimes i^*\widehat{\omega} .
\end{equation}
\end{Prop}

\begin{proof}
Let $L_0$ be the pullback of $i^*\mathcal{L}$ to a line bundle on $X_0$, so that $L_0$ is isomorphic to the 
pullback of $\co_X(X_0)$ to $X_0$.   The function $f(z_1,z_2) = (z_1-z_2)^{-1}$ on $X$ defines a global 
nonvanishing section of $\co_X(X_0)$, which in turn restricts to a global nonvanishing section $\sigma_0$ of 
$L_0$.  Under the metric on $L_0$ determined by the metric (\ref{green metric}) on $i^*\mathcal{L}$,  this section 
has norm (using \cite[(7.3.16)]{KRY} for the final equality)
\begin{eqnarray}
-\log ||\sigma_0||^2_{z_0}   &=& -  \lim_{x\to z_0} \log|| s/(x_1-x_2)  ||^2_x \nonumber \\
&  =  & \log(4u \cdot d_F\mathrm{disc}(B_0)) +  \vartheta_u(z_0) +  \lim_{x\to z_0} ( g_u^0(x_1,x_2)  + \log |x_1-x_2  |^2 ) \nonumber \\
& =& \log(d_F\mathrm{disc}(B_0)) +  \vartheta_u(z_0) -  \gamma_{\mathrm{Euler}}  -  
\log\left(  \frac{\pi}{\mathrm{Im}(z_0)^2} \right) \label{obvious section}
\end{eqnarray}
where $z_0\in X_0$, and in each limit $x\in X\smallsetminus X_0$.
  The isomorphism of line bundles (\ref{class adjunction}) can be viewed as an isomorphism 
  \begin{equation}\label{class adjunction II}
  \omega_0\otimes\co_{\mathcal{M}_0} \iso i^*\mathcal{L} \otimes \omega,
  \end{equation}
which  pulls back to the isomorphism of line bundles 
$$
\Omega^1_{X_0} \otimes \co_{X_0} \iso L_0 \otimes i^*\Omega^2_X
$$
on $X_0$ determined by $dz_0 \otimes 1\leftrightarrow \sigma_0 \otimes (dz_1\wedge dz_2).$  
Comparing (\ref{obvious section}) with Proposition \ref{Prop:canonical metrics} gives
\begin{eqnarray*}
-\log|| dz_0 \otimes 1||_{z_0}^2 &=&  
\gamma_{\mathrm{Euler}} + \log\left( \frac{\pi}{\mathrm{disc}(B_0) \cdot \mathrm{Im}(z_0)^2}\right) + \vartheta_u(z_0)
 \\
&=& -\log|| \sigma_0 \otimes (dz_1\wedge dz_2)  ||_{z_0}^2,
\end{eqnarray*}
and therefore the isomorphism (\ref{class adjunction II}) respects the metrics of (\ref{class adjunction III}).
\end{proof}

\begin{Cor}\label{Cor:second adjunction}
There is an isomorphism of metrized line bundles on $\mathcal{M}_0$
$$
 \widehat{\co}_{\mathcal{M}_0} (u) \iso i^*\widehat{\mathcal{L}}(u)\otimes\widehat{\omega}_0.
$$
In particular for  any irreducible horizontal cycle $\mathcal{D}$ on $\mathcal{M}_0$
$$
h_{i^*\widehat{\mathcal{L}}(u)}(\mathcal{D}) +  h_{\widehat{\omega}_0} (\mathcal{D})   =  
 \frac{1}{2} \sum_{P\in\mathcal{D}(\C)} e_P^{-1} \vartheta_u(x_0).
$$
\end{Cor}

\begin{proof}
The first claim is immediate from Proposition \ref{Prop:first adjunction} and the isomorphism of 
metrized line bundles (\ref{metric splitting}).  The second claim is then just a restatement of (\ref{arakelov height}).
\end{proof}

\begin{Lem}\label{Lem:messy integrals}
For any  irreducible horizontal cycle $\mathcal{D}$ on $\mathcal{M}$ of codimension two
\begin{eqnarray*}\lefteqn{
\widehat{\deg}_{\mathcal{M}_0}  \widehat{\mathcal{D}}(v)   =  h_{\widehat{\mathcal{L}}(v_1) }  ( \mathcal{D}) 
- \frac{1}{2} \deg_\Q(\mathcal{D}) \log(4 v_1 d_F  \mathrm{disc}(B_0))    } \\
& &  + \frac{1}{2} \sum_{P\in \mathcal{D}(\C)} e_P^{-1}  \sum_{\gamma\in   \Gamma_0 \backslash  \Gamma} 
 \int_{X_0} \mathbf{g}_0( \gamma_2 x_2,v_2)\wedge \Phi_0( \gamma_1 x_1,v_1)
 \end{eqnarray*}
where $x\in X$ is any point above  $P$ under $\mathcal{M}(\C)\iso [\Gamma\backslash X]$, and 
\begin{equation}\label{generic degree}
\deg_\Q(\mathcal{D}) = \sum_{P\in \mathcal{D}(\C)} e_P^{-1}.
\end{equation}
\end{Lem}

\begin{proof}
Kudla's  function $g_u^0$ on $X$ satisfies the Green equation
$$
dd^c g_u^0 + \delta_{X_0} = c_u^0
$$
for some smooth $\Gamma_0$-invariant $(1,1)$-form $c_u^0$ on $X$, and from the explicit calculation of 
$dd^c g_u^0$ in the proof of \cite[Proposition 7.3.1]{KRY} we see that 
$$
c_u^0  = \phi_u^0\cdot  \pi_1^*\mu_0 + \phi_u^0 \cdot \pi_2^*\mu_0 +
 \alpha_u \cdot dz_1\wedge d\overline{z}_2 +\beta_u \cdot  d\overline{z}_1\wedge dz_2
$$
for smooth functions $\alpha_u$ and $\beta_u$ on $X$.   Define a $\Gamma_0$-invariant function
$$
G^0_u(x)  =   \log(4u \cdot d_F \mathrm{disc}(B_0)) + \sum_{ \gamma\in \Gamma_0\backslash \Gamma}  
g^0_{u}( \gamma_1 x_1 , \gamma_2  x_2) 
$$
on $X\smallsetminus \Gamma X_0$,  and view $G_u^0$ as a $(0,0)$-current on the orbifold 
$\mathcal{M}(\C)\iso [\Gamma\backslash X]$.  As  $G^0_u$ satisfies the  Green equation
$$
dd^c G^0_u + \delta_{\mathcal{M}_0} =  \sum_{\gamma\in \Gamma_0\backslash \Gamma} \gamma^*c^0_u,
$$
we may consider the arithmetic cycle class $\widehat{\mathcal{M}}_0(u)\in \widehat{\mathrm{CH}}^1(\mathcal{M})$ 
determined by  $(\mathcal{M}_0,G_u^0)$.  Comparing with (\ref{green metric}), we note that 
$\widehat{\mathcal{M}}_0(u)$ is the arithmetic Chern class   (in the sense of \cite[\S 2.1.2]{BGS}) 
of the metrized line bundle  $\widehat{\mathcal{L}}(u)$.

Recall from \cite[Lemma 3.4.3]{gillet-soule90} that there is a canonical isomorphism of $\Q$-vector spaces
$$
\widehat{\mathrm{CH}}^1(\Spec(\Z))\map{}\R.
$$
  The lemma is really a special case of  \cite[Proposition 2.3.1]{BGS}, which relates both sides of the 
  stated equality to the value of $ \widehat{\mathcal{D}}(v)\cdot \widehat{\mathcal{M}}_0(v_1)$, 
  where the product  is the arithmetic intersection 
\begin{equation}\label{arithmetic intersection}
\widehat{\mathrm{CH}}^2(\mathcal{M}) \times \widehat{\mathrm{CH}}^1(\mathcal{M}) \map{} \widehat{\mathrm{CH}}^1(\Spec(\Z))\map{}  \R.
\end{equation}
Indeed, what we call the arithmetic degree $\widehat{\deg}_{\mathcal{M}_0}\widehat{D}(v)$  is equal to 
the intersection pairing $(\widehat{\mathcal{D}}(v) \mid \mathcal{M}_0)$, where the pairing
$$
\widehat{\mathrm{CH}}^2(\mathcal{M})\times Z_2(\mathcal{M})  \map{} \widehat{\mathrm{CH}}^1(\Spec(\Z)) \map{} \R
$$
 is defined by \cite[(2.3.1)]{BGS}.  On the other hand, by \cite[Proposition 2.3.1(vi)]{BGS} the Arakelov height 
 $h_{\widehat{\mathcal{L}}(u)}(\mathcal{D})$ is equal to the intersection pairing $(\widehat{\mathcal{M}}_0(u)\mid \mathcal{D})$, 
 where now we use the pairing
 $$
\widehat{\mathrm{CH}}^1(\mathcal{M})\times Z_1(\mathcal{M})  \map{} \widehat{\mathrm{CH}}^1(\Spec(\Z)) \map{} \R.
$$
  Comparing \cite[(2.3.1)]{BGS} with \cite[(2.3.3)]{BGS} shows that   
$$
  \widehat{\deg}_{\mathcal{M}_0} \widehat{\mathcal{D}}(v)   
=   \widehat{\mathcal{D}}(v)\cdot \widehat{\mathcal{M}}_0(u)
-   \frac{1}{2} \sum_{P\in\mathcal{D}(\C) } e_P^{-1} \int_{ X }  G_{u}^0\wedge\Phi_1(x,v) \wedge \Phi_2(x,v) .
$$
Taking $u=v_1$ and using
\begin{eqnarray*}\lefteqn{
 \int_{ X }   \log( 4u \mathrm{disc}(B_0) ) 
 \wedge\Phi_1(x,v) \wedge\Phi_2(x,v)
    }   \\
 & = &   \log( 4u \mathrm{disc}(B_0) )
 \left(  \int_{ X_0 }  \phi^0_{v_1}(x_1,z_1)  d\mu_0(z_1)\right) \left(\int_{X_0}\phi^0_{v_2}(x_2,z_2)  d\mu_0(z_2)  \right) \\
 &=&  \log( 4u \mathrm{disc}(B_0) )
\end{eqnarray*}
shows that   $\widehat{\mathcal{D}}(v) \cdot \widehat{\mathcal{M}}_0 (v_1)$ is equal to 
\begin{eqnarray*}\lefteqn{
\widehat{\deg}_{\mathcal{M}_0} \widehat{\mathcal{D}}(v)  + \frac{1}{2} \deg_\Q(\mathcal{D})  \log( 4v_1 \mathrm{disc}(B_0) )  }  \\
& & +   \frac{1}{2} \sum_{P\in\mathcal{D}(\C) } e_P^{-1} \sum_{\gamma\in\Gamma_0\backslash \Gamma} 
  \int_{ X }  g_{v_1}^0(\gamma_1 z_1,\gamma_2 z_2) \phi^0_{v_1}(x_1,z_1)  \phi^0_{v_2}(x_2,z_2) d\mu_0(z_1) d\mu_0(z_2) . 
    \end{eqnarray*}
 On the other hand, we may use the symmetry of the pairing (\ref{arithmetic intersection}) to reverse the roles of $\mathcal{M}_0$ and 
 $\mathcal{D}$, and deduce that
 $$
h_{ \widehat{\mathcal{L}}(v_1) }(\mathcal{D}) = 
  \widehat{\mathcal{M}}_0 ( v_1 ) \cdot  \widehat{\mathcal{D}}(v)  -  
  \frac{1}{2} \sum_{P\in\mathcal{D}(\C) } e_P^{-1} \sum_{ \gamma \in\Gamma_0\backslash \Gamma} 
   \int_{ X }  \gamma^*c_{v_1}^0 \wedge \mathbf{g}(x,v).
$$
 The integral can be rewritten, using $(\gamma^{-1})^*\mathbf{g}(x, v) = \mathbf{g}(\gamma x,v)$, as
\begin{eqnarray*}\lefteqn{
 \int_{ X  }  c_{v_1}^0 \wedge \mathbf{g}(\gamma x,v) }  \\
&=& 
\int_{  \{\gamma_1 x_1\} \times X_0 }  
c_{v_1}^0 \wedge \mathbf{g}_2(\gamma x,v)  
+ \int_{ X }  c_{v_1}^0 \wedge \mathbf{g}_1( \gamma x,v)\wedge \Phi_2(\gamma x,v) \\
 &=& 
 \int_{  \{\gamma_1 x_1\} \times X_0 }  
 \phi_{v_1}^0 \pi_2^*\mu_0 \wedge \mathbf{g}_2(\gamma x,v)    +  \int_{ X }  
 \phi^0_{v_1} \pi_1^*\mu_0 \wedge \mathbf{g}_1( \gamma x,v)\wedge \Phi_2(\gamma x,v) \\
 &=& 
 \int_{ X_0  }  \phi_{v_1}^0 (\gamma_1 x_1,z_2)  g^0_{v_2}(\gamma_2 x_2,z_2)  d\mu_0(z_2) \\
 & &   + \int_{ X }  g^0_{v_1}( \gamma_1 x_1, z_1 ) \phi^0_{v_1}(z_1,z_2) 
  \phi^0_{v_2}(\gamma_2 x_2,z_2)  d\mu_0(z_1) d\mu_0(z_2)  \\
 &=&    \int_{ X_0  }  \mathbf{g}_0(\gamma_2 x_2, v_2) \wedge \Phi_0(\gamma_1 x_1, v_1)     \\
 & &   + \int_{ X }    g^0_{v_1}( z_1, z_2 ) \phi^0_{v_1}(z_1, \gamma_1 x_2)  
  \phi^0_{v_2}(\gamma_2 x_2,z_2)  d\mu_0(z_1) d\mu_0(z_2)\\ 
   &=&   \int_{ X_0  }  \mathbf{g}_0(\gamma_2 x_2, v_2) \wedge \Phi_0(\gamma_1 x_1, v_1)   \\
 & &   + \int_{ X }    g^0_{v_1}(\gamma_1 z_1, \gamma_2 z_2 ) \phi^0_{v_1}(z_1,  x_2)  
  \phi^0_{v_2}( x_2,z_2)  d\mu_0(z_1) d\mu_0(z_2).
\end{eqnarray*}
For the second to last equality we have used the following observation: for each fixed $z_2\in X_0$ 
there is a $T\in G_0(\R)$ whose action on $X_0$ interchanges $z_2$ and $\gamma_1 x_1$. Hence
\begin{eqnarray*}
\int_{X_0}   g_{v_1}( \gamma_1 x_1,  z_1   )  \phi_{v_1}( z_1 ,  z_2)    d\mu_0(z_1) 
&=&   
\int_{X_0}    g_{v_1}( z_2, T^{-1} z_1   )  \phi_{v_1}( z_1 ,  z_2)   d\mu_0(z_1)  \\
&=&   
\int_{X_0}    g_{v_1}( z_2 ,  z_1   )  \phi_{v_1}(  T z_1 ,  z_2)   d\mu_0(z_1)  \\
&=&   
\int_{X_0}    g_{v_1}( z_2,  z_1   )  \phi_{v_1}( z_1 ,  \gamma_1 x_1)   d\mu_0(z_1).
\end{eqnarray*}
Comparing the two formulas for 
$\widehat{\mathcal{D}}(v) \cdot \widehat{\mathcal{M}}_0(v_1) = \widehat{\mathcal{M}}_0(v_1) \cdot  \widehat{\mathcal{D}}(v) $ 
proves the claim.
\end{proof}

\begin{Rem}
The proof of Lemma \ref{Lem:messy integrals} makes extensive use of the arithmetic 
intersection theory for \emph{schemes}  developed in  \cite{gillet-soule90} and in \cite{BGS}, while we are working 
with the \emph{stacks} $\mathcal{M}_0$ and $\mathcal{M}$.  This can be justified by using a trick of Bruinier-Burgos-K\"uhn 
\cite{bruinier-burgos-kuhn} to deduce an adequate intersection theory for $\mathcal{M}$ by writing (for every $N\in\Z^+$) 
the stack $\mathcal{M}_{/\Z[1/N]}$  as the quotient of a $\Z[1/N]$-scheme $M_{\Z[1/N]}$ by the action 
of a finite group, and using compatibility of the  Gillet-Soul\'e intersection theory for $M_{\Z[1/N]}$ as $N$ varies.  
See for example the construction of 
$\widehat{\deg}_{\mathcal{M}_0}$ given in \cite[\S 2.3]{howardA}.
\end{Rem}

Now suppose we start with an irreducible horizontal cycle $\mathcal{D}$ of codimension one on 
$\mathcal{M}_0$.  Viewing $\mathcal{D}$ as a codimension two cycle on $\mathcal{M}$,  we may 
form the arithmetic cycle class  $\widehat{\mathcal{D}}(v)$ of (\ref{augment}).  Our  final form of the 
arithmetic adjunction formula will compute the arithmetic degree of  $\widehat{\mathcal{D}}(v)$  along 
$\mathcal{M}_0$ in terms of purely archimedean data and the quantity $h_{\widehat{\omega}_0}(\mathcal{D})$.  
The essential point is that while  the arithmetic degree of  $\widehat{\mathcal{D}}(v)$  along $\mathcal{M}_0$ 
depends on the positions of $\mathcal{M}_0$ and $\mathcal{D}$ inside of the ambient threefold $\mathcal{M}$, 
the Arakelov height $h_{\widehat{\omega}_0}(\mathcal{D})$ depends only on the position of $\mathcal{D}$ in 
$\mathcal{M}_0$, and makes no reference to the threefold $\mathcal{M}$.   In the applications $\mathcal{D}$ will 
be chosen in such a way that the quantity $h_{\widehat{\omega}_0}(\mathcal{D})$ has already been calculated by 
Kudla-Rapoport-Yang \cite{kudla04a}.   To state the formula we need the following notation.  As in \cite[Lemma 7.5.4]{KRY} 
define a function $J:\R^+\map{}\R^+$ by
$$
J(t)=\int_0^\infty w^{-1} e^{-tw}[(w+1)^{1/2}-1]\ dw
$$
so that for any fixed $x_0\in X_0$
\begin{equation}\label{J evaluation}
\int_{X_0} \mathbf{g}_0(x_0,v_2) \wedge \Phi_0(x_0,v_1) = \log\left( \frac{v_1+v_2}{v_2} \right)  -  J(4\pi v_1 + 4\pi v_2 ).
\end{equation}

\begin{Thm}[Arithmetic adjunction]\label{Thm:adjunction}
Suppose  $\mathcal{D}$ is an irreducible horizontal cycle  of codimension one on $\mathcal{M}_0$.  
Viewing $\mathcal{D}$ as a cycle on $\mathcal{M}$, let $$\widehat{\mathcal{D}}(v)\in\widehat{\mathrm{CH}}^2(\mathcal{M})$$ 
be the arithmetic cycle class of (\ref{augment}).    Then  
\begin{eqnarray*}
h_{\widehat{\omega}_0} (\mathcal{D})    +  \widehat{\deg}_{\mathcal{M}_0}   \widehat{\mathcal{D} } (v)  & = &
    \frac{1}{2} \deg_\Q (\mathcal{D})  \cdot \log\left(   \frac{ v_1+v_2   }{4v_1v_2  \cdot d_F\mathrm{disc}(B_0)  } \right)  \\
 & &     -     \frac{1}{2} \deg_\Q (\mathcal{D})   \cdot   J(4\pi v_1 + 4\pi v_2)     \\
& & +   \frac{1}{2}\sum_{ P\in\mathcal{D}(\C) } e_P^{-1}  
\sum_{  \substack{\gamma\in \Gamma_0\backslash \Gamma \\ \gamma\not\in \Gamma_0 }}  
 \int_{X_0} \mathbf{g}_0(\gamma_1 x_0 , v_1) * \mathbf{g}_0(\gamma_2 x_0 , v_2) .
 \end{eqnarray*}
 In the integral $x_0\in X_0$ is any point above $P$ under  
 $\mathcal{M}_0(\C)\iso [\Gamma_0\backslash X_0]$, and $\deg_\Q$ is defined by (\ref{generic degree}).
  \end{Thm}

\begin{proof}
The crux of the proof is a trivial observation: the height $h_{\widehat{\mathcal{L}}(u )} (\mathcal{D})$ 
depends only on the pullback of the metrized line bundle $\widehat{\mathcal{L}}(u)$ to $\mathcal{D}$, 
which can be computed by first pulling back from $\mathcal{M}$ to $\mathcal{M}_0$ and then from 
$\mathcal{M}_0$ to $\mathcal{D}$.  In other words
$$
h_{i^*\widehat{\mathcal{L}}(u) } (\mathcal{D}) = h_{\widehat{\mathcal{L}}(u )} (\mathcal{D})
$$
where $\mathcal{D}$ is regarded as a cycle on $\mathcal{M}_0$ on the left, and as a cycle on 
$\mathcal{M}$ on the right.  Thus Lemma \ref{Lem:messy integrals} shows that 
$\widehat{\deg}_{\mathcal{M}_0}   \widehat{\mathcal{D} } (v)$ is equal to 
\begin{eqnarray*}\lefteqn{
 h_{i^*\widehat{\mathcal{L}}(v_1) }  ( \mathcal{D}) - \frac{1}{2} \deg_\Q(\mathcal{D}) \log(4 v_1 d_F \mathrm{disc}(B_0))    } \\
& &  + \frac{1}{2} \sum_{P\in \mathcal{D}(\C)} e_P^{-1}  \sum_{\gamma\in   \Gamma_0 \backslash  \Gamma}
  \int_{X_0} \mathbf{g}_0( \gamma_2 x_0,v_2)\wedge \Phi_0( \gamma_1 x_0,v_1). 
\end{eqnarray*}
Using Corollary \ref{Cor:second adjunction} and (\ref{J evaluation}) this can be written as
\begin{eqnarray*}\lefteqn{
 -  h_{\widehat{\omega}_0 }  ( \mathcal{D}) + \frac{1}{2} \deg_\Q(\mathcal{D}) 
 \log\left(   \frac{ v_1+v_2   }{4v_1v_2  \cdot d_F  \mathrm{disc}(B_0)    }\right)  
  - \frac{1}{2}  \deg_\Q(\mathcal{D}) J( 4\pi v_1 + 4\pi v_2)  }    \qquad   \\
& &  + \frac{1}{2} \sum_{P\in \mathcal{D}(\C)} e_P^{-1}  \Big( \vartheta_{v_1}(x_0)  
+ \sum_{ \substack{\gamma\in   \Gamma_0 \backslash  \Gamma \\ \gamma\not\in\Gamma_0 } }
  \int_{X_0} \mathbf{g}_0( \gamma_2 x_0,v_2)\wedge \Phi_0( \gamma_1 x_0,v_1) \Big).
\end{eqnarray*}
The final quantity in parentheses is 
\begin{eqnarray*}\lefteqn{ 
 \sum_{ \substack{\gamma\in   \Gamma_0 \backslash  \Gamma \\ \gamma\not\in\Gamma_0 } } 
\Big( g^0_{v_1}(\gamma_1x_0,\gamma_2 x_0) +  
 \int_{X_0}   \mathbf{g}_0( \gamma_2 x_0,v_2)\wedge \Phi_0( \gamma_1 x_0,v_1) \Big) } \\
 & = &
 \sum_{ \substack{\gamma\in   \Gamma_0 \backslash  \Gamma \\ \gamma\not\in\Gamma_0 } } 
 \int_{X_0}  \Big(  \mathbf{g}_0(\gamma_1x_0, v_1) \wedge \delta_{ \{\gamma_2 x_0 \} } +  
     \mathbf{g}_0( \gamma_2 x_0,v_2)\wedge \Phi_0( \gamma_1 x_0,v_1) \Big) \\
 & = &
  \sum_{ \substack{\gamma\in   \Gamma_0 \backslash  \Gamma \\ \gamma\not\in\Gamma_0 } } 
 \int_{X_0}   \mathbf{g}_0(\gamma_1x_0, v_1) *    \mathbf{g}_0( \gamma_2 x_0,v_2)  
 \end{eqnarray*}
completing the proof.
\end{proof}


\section{Unramified intersection theory}
\label{s:good reduction}


The arithmetic adjunction formula of the previous section will allow us to compute the 
degree along $\mathcal{M}_0$ of those horizontal components of $\mathcal{Y}(\alpha)$ that 
intersect $\mathcal{M}_0$ improperly, and now we turn to the intersection theory of the remaining 
components of $\mathcal{Y}(\alpha)$.     Much of the hard work in these calculations is contained in \cite{howardB}.  
In this section we consider intersection  multiplicities in characteristics prime to   the discriminant of $B_0$.  
The case of characteristic dividing the discriminant of $B_0$ will be treated in the next section.

Fix a totally positive $\alpha\in\co_F$ and abbreviate $\mathcal{Y}=\mathcal{Y}(\alpha)$ and 
$\mathcal{Y}_0=\mathcal{Y}_0(\alpha)$.  Let $p$ be  a  prime that does not divide the discriminant of 
$B_0$, and fix an isomorphisms of stacks  $\mathcal{M}_{/\Z_p}\iso [H\backslash M]$ with $M$ a  $\Z_p$-scheme    
and $H$ a  finite group of automorphisms of $M$.  Set
$$
Y=\mathcal{Y}\times_\mathcal{M} M 
\qquad M_0=\mathcal{M}_0\times_{\mathcal{M}} M
 \qquad Y_0=\mathcal{Y}_0\times_{\mathcal{M}} M$$
so that there is a cartesian diagram of $\Z_p$-schemes
$$
\xymatrix{
{  Y_0 } \ar[r]^j \ar[d]_{\phi_0}  &  {  Y  } \ar[d]^{\phi}  \\
{ M_0 } \ar[r]_i  & { M}.
}
$$
The scheme $Y$ has dimension at most one (see \cite[Proposition 3.3.1]{howardA}).  
The scheme $Y_0$ has dimension zero if $F(\sqrt{-\alpha})/\Q$ is not biquadratic 
(see the proof of \cite[Lemma 5.1.2]{howardA}) but otherwise $Y_0$ may have components of dimension one.  
For every nonzero $T\in\mathrm{Sym}_2(\Z)^\vee$ and nonzero $t\in \Z$ set 
$$
Z(T)=\mathcal{Z}(T)\times_{\mathcal{M}_0} M_0 \qquad Z(t) =\mathcal{Z}(t)\times_{\mathcal{M}_0} M_0.
$$
If $\det(T)\not=0$ then the scheme $Z(T)$ is zero dimensional (by \cite[Theorem 3.6.1]{KRY}).  
If $\det(T)=0$ then $Z(T)$ has dimension at most one and every irreducible component is horizontal 
(by \cite[Proposition 3.4.5]{KRY} and \cite[Lemma 6.4.1]{KRY}).   The scheme $Z(t)$ has dimension at most one and 
every irreducible component is horizontal (again by   \cite[Proposition 3.4.5]{KRY}).   The decomposition 
(\ref{moduli decomp}) induces a decomposition
\begin{equation}\label{scheme decomp}
Y_0 = \bigsqcup_{T\in\Sigma(\alpha)} Z(T).
\end{equation}

For any Noetherian scheme $X$ let $\mathbf{K}_0(X)$ be the Grothendieck group of the category of coherent $\co_X$-modules.  
If $J\map{}X$ is a proper morphism,  let $\mathbf{K}_0^J(X)$ be the Grothendieck group of the category of coherent  $\co_X$-
modules that are supported on the image of $J$.  If $\mathcal{F}$ is a coherent $\co_X$-module we denote by $[\mathcal{F}]$ the 
corresponding class in $\mathbf{K}_0(X)$.  Let $\Pi(Y)$ denote the set of all irreducible  components of $Y$ of dimension one, and 
endow each such component with its reduced subscheme structure.  If $D\in\Pi(Y)$ has generic point $\eta$,  and $\mathcal{F}$ is a 
coherent $\co_Y$-module, define the \emph{multiplicity of $\mathcal{F}$ along $D$} to be the length of the stalk
$$
\mathrm{mult}_D(\mathcal{F}) = \length_{\co_{Y,\eta}}(\mathcal{F}_\eta).
$$
The multiplicity is finite (as $\co_{Y,\eta}$ is Artinian)  and depends only on the class  of $\mathcal{F}$ in $\mathbf{K}_0(Y)$, not on the 
sheaf $\mathcal{F}$ itself.  Define
$$
[\mathcal{F}]_D= \mathrm{mult}_D(\mathcal{F})\cdot [\co_D] \in\mathbf{K}_0(D).
$$
As the inclusion $D\map{}Y$ is finite, push forward of sheaves is an exact functor from coherent $\co_D$-modules to coherent 
$\co_Y$-modules, and induces a homomorphism $\mathbf{K}_0(D)\map{}\mathbf{K}_0(Y)$.  Thus we may also view 
$[\mathcal{F}]_D\in\mathbf{K}_0(Y)$.  It follows from \cite[Lemma 2.2.2]{howardA} that for any class $[\mathcal{F}] \in\mathbf{K}_0(Y)$ 
there is  a canonical decomposition in $\mathbf{K}_0(Y)$
\begin{equation}\label{up to small}
[\mathcal{F}]= [\mathcal{F}]^\mathrm{small} + \sum_{D\in \Pi(Y)} [\mathcal{F}]_D 
\end{equation}
in which  $[\mathcal{F}]^\mathrm{small}$  lies in the image of $\mathbf{K}_0^J(Y)\map{}\mathbf{K}_0(Y)$ for some closed subscheme  
$J\map{}Y$  of dimension zero.

\begin{Def}\label{Def:good proper}
 We say that an irreducible component $D\in\Pi(Y)$ is \emph{improper} if it is  contained in the closed subscheme $Y_0$, and is 
 \emph{proper} otherwise.   Thus the term ``proper" is shorthand for ``meets $M_0$ properly."
\end{Def}

 Let $\Pi^\proper(Y)\subset \Pi(Y)$ be the subset of proper components,  and write $\Pi(Y)$ as a disjoint union
$$
\Pi(Y)=\Pi^\good(Y) \cup \Pi^\bad(Y) \cup \Pi^\vertical(Y)
$$
in which $\Pi^\vertical(Y)$ is the subset of vertical  components,  $\Pi^\good(Y)$ is the subset of proper horizontal components, 
and $\Pi^\bad(Y)$ is the subset of improper  horizontal components. 
 As $Y_0$ has no vertical components of dimension one (by the decomposition (\ref{scheme decomp}) and the corresponding property 
 of $Z(T)$ noted above)  
 $$ 
 \Pi^\proper(Y) = \Pi^\good(Y)\cup \Pi^\vertical(Y).
 $$  
Give
$$
Y^\good = \bigcup_{D\in\Pi^\good(Y)} D \subset Y
$$
its reduced subscheme structure and define $Y^\bad$, $Y^\vertical$, and $Y^\proper$ similarly. 
For any $[\mathcal{F}]\in\mathbf{K}_0(Y)$, set
$$ 
[\mathcal{F}]^\good =  \sum_{D\in \Pi^\good(Y)}  [\mathcal{F}]_D \in\mathbf{K}_0(Y^\good)
$$ 
and define $ [\mathcal{F}]^\bad $,  $ [\mathcal{F}]^\vertical $, and $[\mathcal{F}]^\proper$ similarly. In  $\mathbf{K}_0(Y)$ we have 
the relation
$$
[\mathcal{F}]= [\mathcal{F}]^\mathrm{small} +  [\mathcal{F}]^\good + [\mathcal{F}]^\bad + [\mathcal{F}]^\vertical,
$$
and in $\mathbf{K}_0(Y^\proper)$ we have   
\begin{equation}\label{simple prop decomp}
[\mathcal{F}]^\proper = [\mathcal{F}]^\good + [\mathcal{F}]^\vertical.
\end{equation}
  For any coherent $\co_{Y^\proper}$-module $\mathcal{F}$ (which we also view as a coherent $\co_Y$-module supported on 
  $Y^\proper$) and any $y\in Y_0(\F_p^\alg)$ (which we also view   as an element of $Y(\F_p^\alg)$, $M_0(\F_p^\alg)$, or 
  $M(\F_p^\alg)$ as needed) define
\begin{eqnarray}\label{serre def}
I_{\co_{Y_0,y}} (  \mathcal{F} ,   \co_{M_0}    )
 &=&  \sum_{\ell \ge 0} (-1)^\ell  \cdot  \length_{\co_{Y_0,y}} \Tor_\ell^{ \co_{M,y} }(  \mathcal{F}_y, \co_{{M_0},y} ) \\
 &=&  \sum_{\ell \ge 0} (-1)^\ell  \cdot  \length_{\co_{Y_0,y}} \Tor_\ell^{ \co_{Y,y} }(  \mathcal{F}_y, \co_{{Y_0},y} ) . \nonumber
\end{eqnarray}
The rule  $  [\mathcal{F}]\mapsto I_{\co_{Y_0,y}} (  \mathcal{F} ,   \co_{M_0}    )$ determines a  homomorphism 
$\mathbf{K}_0(Y^\proper)\map{}\Z$.

\begin{Prop} 
\label{Prop:good nonsingular}
For every  nonsingular $T\in\Sigma(\alpha)$ and $y\in Z(T)(\F_p^\alg)$
$$
 I_{\co_{Y_0,y}} (  [\co_Y]^\good ,   \co_{M_0}    ) +  I_{\co_{Y_0,y}} ( [\co_Y]^\vertical ,   \co_{M_0}    )
  =  \length_{\co_{Z(T)},y} (  \co_{Z(T),y} ).
$$
\end{Prop}

\begin{proof}
For any $y\in Z(T)(\F_p^\alg)$,  the decomposition (\ref{scheme decomp}) implies $\co_{Y_0,y}\iso \co_{Z(T),y}$.  As 
noted earlier, the hypothesis that $T$ is nonsingular implies that $Z(T)$ has dimension zero, and hence    $\co_{Y_0,y}$ is Artinian.  It 
follows that  the right hand side of (\ref{serre def}) is defined for \emph{every} coherent $\co_Y$-module $\mathcal{F}$, not merely for 
coherent $\co_{Y^\proper}$-modules, and that 
$$
\mathcal{F} \mapsto I_{\co_{Y_0,y}} (  \mathcal{F} ,   \co_{M_0}    )
$$
extends to a homomorphism $\mathbf{K}_0(Y)\map{}\Z$.  Furthermore \cite[Lemma 4.2.2]{howardA} implies that 
$[\mathcal{F}]^\mathrm{small}$ lies in the kernel of this homomorphism, and so (\ref{up to small}) implies
\begin{equation}\label{gns display I}
 I_{\co_{Y_0,y}} (  \mathcal{F} ,   \co_{M_0}    ) = \sum_{D\in\Pi(Y)} I_{\co_{Y_0,y}} (  [\mathcal{F}]_D ,   \co_{M_0}    ).
\end{equation}
 If $D\in\Pi(Y)$ is an improper component then $D\subset Y_0$, and so
 $$
\dim Z(T)=0 \implies D\not\subset Z(T) \implies y\not\in D(\F_p^\alg).
 $$
Hence  if we view $\co_D$ as a coherent $\co_{M}$-module,  the stalk  $\co_{D,y}$ is trivial.  Taking $\mathcal{F}=\co_Y$ 
in (\ref{gns display I}) now gives
\begin{equation}\label{gns display II}
 I_{\co_{Y_0,y}} ( \co_Y , \co_{M_0} ) = I_{\co_{Y_0,y}} ( [\co_Y]^\proper , \co_{M_0} ).  
\end{equation}

The local ring $\co_{Y,y}$ is   Cohen-Macaulay of dimension one by  \cite[Lemma 3.3.4]{howardA} and  \cite[Corollary 3.3.9]{howardA}, 
and so by the argument leading to (\ref{no tor})  
$$
\Tor_\ell^{\co_{M,y}}(\co_{Y,y},\co_{M_0,y})=0
$$  
for $\ell>0$.  Therefore
$$
I_{\co_{Y_0,y}}(\co_Y,\co_{M_0}) = \length_{\co_{Y_0,y}}(\co_{Y_0,y}) =  \length_{\co_{Z(T)},y}(\co_{Z(T),y})
$$
which, when combined with (\ref{simple prop decomp}) and  (\ref{gns display II}), completes the proof.
\end{proof}

For the remainder of \S \ref{s:good reduction}, suppose that $T\in \Sigma(\alpha)$ is singular, and denote by $t_1$ and $t_2$ the 
diagonal entries of $T$.  As $t_1t_2$ is a square,  there are relatively prime integers $n_1$ and $n_2$   satisfying 
\begin{equation}\label{little t}
t_1=n_1^2 \cdot t \qquad t_2=n_2^2 \cdot t
\end{equation}
for some nonzero $t\in\Z$,  uniquely determined by $T$.   Each of $n_1$ and $n_2$ is uniquely determined up to sign, and 
(directly from the definition of $\Sigma(\alpha)$)  these signs may be chosen so that 
$$
\alpha = (n_1\varpi_1 + n_2 \varpi_2)^2\cdot  t.
$$
This implies that the field extension $F(\sqrt{-\alpha})/\Q$ is  biquadratic, $t>0$, the field
$$
K \define  \Q(\sqrt{-t})
$$ 
 is one of the two quadratic imaginary subfields of  $F(\sqrt{-\alpha})$, and   $t$ is the largest integer with the property
 $$
 \co_F[\sqrt{-\alpha}] \subset \co_F[\sqrt{-t}].
 $$
 Furthermore \cite[Lemma 6.4.1]{KRY} provides an isomorphism of stacks
\begin{equation}\label{degenerate cycle}
 \mathcal{Z}(t)\iso \mathcal{Z}(T),
\end{equation}
which takes the triple $(\mathbf{A}_0, \lambda_0 , s_0)$ to the quadruple  $(\mathbf{A}_0, \lambda_0 , n_1s_0, n_2s_0)$.  Let   $m_0$ 
be the ramification index of $p$ in $K/\Q$,  abbreviate $d_K=\mathrm{disc}(K/\Q)$, and define $n\in\Z^+$ by    $4t=-n^2 d_K$  so that 
$n$ is the conductor of   $ \Z[\sqrt{-t}] $.    Let 
$$
 \chi =\left(\frac{d_K}{p} \right)\in\{-1,0,1\}.
 $$

\begin{Conj}\label{Conj:intersection conjecture} 
As above, let $T\in\Sigma(\alpha)$ be singular.  For every   $y\in  Z(T)(\F_p^\alg)$ 
\begin{equation}
\label{good component conjecture}
I_{\co_{Y_0,y}} (  [\co_Y]^\good  ,   \co_{M_0}    )   + 
I_{\co_{Y_0,y}} (  [\co_Y]^\vertical  ,   \co_{M_0}    )  
=  \frac{1}{2} \cdot  \Gamma_p(T) \cdot  \ord_p\left(\frac{4 \alpha\alpha^\sigma }{t}\right)
\end{equation}
where
$$
\Gamma_p(T) = m_0\cdot \frac{ p^{\ord_p(n)+1} -1 }{p-1}  -\chi  m_0 \cdot \frac{p^{\ord_p(n)} -1 }{p-1}.
$$
\end{Conj}

The motivation for the conjecture is simply that (\ref{good component conjecture}) is what is needed for the equality of Proposition 
\ref{Prop:main unramified} below, and hence also the main result Theorem \ref{Thm:main result},  
to hold without unwanted hypotheses.  
In what follows we will prove Conjecture \ref{Conj:intersection conjecture} in many cases; \emph{e.g.}~if 
$p$ is split in  $F$, or if $p$ is odd and  inert in $F$.  
These proofs make essential use of the calculations of the companion paper \cite{howardB}.

Keep $T$ and $y$ as in Conjecture \ref{Conj:intersection conjecture}.  Writing $X$ for any one of $Y$, $Y_0$, $M$, or $M_0$, and 
viewing $y$ as a geometric point of $X$,  abbreviate $R_X$ for the completed strictly Henselian  local ring of $X$ at $y$.  As in the 
proof of Proposition \ref{Prop:good nonsingular}, the local ring $\co_{Y,y}$ is Cohen-Macaulay of dimension one, 
and hence the same is  true of $R_Y$.  In particular $R_Y/\mathfrak{p}$ has dimension one for every minimal prime 
$\mathfrak{p}$ of $R_Y$.  Mimicking  Definition \ref{Def:good proper}, a minimal prime $\mathfrak{p}$ of $R_Y$ 
 is \emph{improper} if it  lies in the image of $\Spec(R_{Y_0})\map{}\Spec(R_Y)$.   We say that 
 $\mathfrak{p}$ is \emph{proper} otherwise, and abbreviate $\Pi^\proper(R_Y)$ for the 
set of proper minimal primes of $R_Y$.  For a minimal prime $\mathfrak{p}\subset R_Y$ abbreviate
$$
\mathrm{mult}(\mathfrak{p})= \length_{R_{Y,\mathfrak{p}}} ( R_{Y,\mathfrak{p}} ) .
$$
Routine commutative algebra shows that the left hand side of (\ref{good component conjecture}) can be computed after passing from 
$\co_{Y,y}$ to $R_Y$:
\begin{eqnarray*}\lefteqn{
I_{\co_{Y_0,y}}( [\co_Y]^\proper,\co_{M_0})  } \\ & &  =
 \sum_{ \mathfrak{p} \in \Pi^\proper(R_Y) }  \mathrm{mult}(\mathfrak{p}) \cdot  
  \sum_{\ell\ge 0}(-1)^\ell \length_{R_{Y_0}} \Tor_\ell^{R_M}( R_Y/\mathfrak{p}, R_{M_0}) .
\end{eqnarray*}
The argument leading to  (\ref{no tor}) shows that only the $\ell=0$ term contributes to the inner sum, and thus 
\begin{equation}\label{be wise}
I_{\co_{Y_0,y}}( [\co_Y]^\proper,\co_{M_0})  =   \sum_{ \mathfrak{p} \in \Pi^\proper(R_Y) } 
 \mathrm{mult}(\mathfrak{p} ) \cdot \length_{R_{Y_0}} (R_Y/\mathfrak{p} \otimes_{R_M} R_{M_0} ) .
\end{equation}

Keeping the notation above,  let $W=W(\F_p^\alg)$ be the ring of Witt vectors of $\F_p^\alg$, and let $\mathbf{Art}$ be the category of 
local Artinian $W$-algebras with residue field $\F_p^\alg$.   Using (\ref{scheme decomp}) and (\ref{degenerate cycle}), the point 
$y\in Z(T)(\F_p^\alg)$ determines triples
$$
(\mathbf{A}_0,\lambda_0, s_0) \in \mathcal{Z}(t)(\F_p^\alg)\qquad
(\mathbf{A},\lambda,t_\alpha) \in \mathcal{Y}(\F_p^\alg)
$$
related by $\mathbf{A}\iso \mathbf{A}_0\otimes\co_F$ and 
$$
t_\alpha=s_0 \otimes (n_1 \varpi_1 + n_2\varpi_2) \in \End(\mathbf{A}_0)\otimes_\Z\co_F\iso \End(\mathbf{A}) .
$$ 
 Let $\mathbf{A}_{0,p}$ and $\mathbf{A}_p$ be the $p$-Barsotti-Tate groups of $\mathbf{A}_0$ and $\mathbf{A}$.  The action of 
 $$\co_{B_0,p} \iso M_2(\Z_p)$$  (for any  $\Z$-module $S$ we abbreviate  $S_p=S\otimes_\Z\Z_p$)  allows us to decompose 
$$
\mathbf{A}_{0,p} \iso \mathfrak{g}_0\times\mathfrak{g}_0
$$
with $\mathfrak{g}_0$ a $p$-Barsotti-Tate group of dimension one and height two equipped with an action of the quadratic 
$\Z_p$-order $\Z_p[s_0]\iso \Z_p[x]/(x^2+t)$.  Similarly 
$$
\mathbf{A}_{p}  \iso \mathfrak{g} \times\mathfrak{g}
$$
where $\mathfrak{g}\iso \mathfrak{g}_0\otimes_{\Z_p}\co_{F,p}$ is equipped with an action of the quadratic  $\co_{F,p}$-order 
$\co_{F,p} [t_\alpha] \iso \co_{F,p}[x]/(x^2+\alpha) $.   The quotient map $M\map{}\mathcal{M}_{/\Z_p}$ is \'etale, and so  
$R_M$ is isomorphic to the completion of the strictly Henselian local
ring of $\mathcal{M}$ at $y$.  It follows from the Serre-Tate theory  that the formal $W$-scheme $\Spf(R_M)$  classifies 
deformations of $\mathfrak{g}$ with its $\co_{F,p}$-action  to objects of $\mathbf{Art}$ (the polarization $\lambda$ lifts uniquely to any 
deformation of $\mathbf{A}$ by \cite[p.~51]{breen-labesse} or \cite{vollaard}).  Similarly $R_Y$ classifies deformations of $\mathfrak{g}$ 
with its $\co_{F,p}[t_\alpha]$-action, $R_{M_0}$ classifies deformations of $\mathfrak{g}_0$, and $R_{Y_0}$ classifies deformations of 
$\mathfrak{g}_0$ with its $\Z_p[s_0]$-action.

Note that $\mathfrak{g}_0$ is either isomorphic to $\Q_p/\Z_p \times\mu_{p^\infty}$ (the \emph{ordinary case}) or to the unique 
connected $p$-Barsotti-Tate group of dimension one and height two (the \emph{supersingular case}).  The endomorphism $s_0$ of 
$\mathfrak{g}_0$ induces an embedding of $K_p\iso \Q_p[s_0]$ into $\End(\mathfrak{g}_0)\otimes_{\Z_p}\Q_p$, from which we see that  
$\chi=1$ implies  that $\mathfrak{g}_0$ is ordinary, while $\chi\not=1$ implies that $\mathfrak{g}_0$ is supersingular.

\begin{Prop}
\label{Prop:good singular I}
If $T\in\Sigma(\alpha)$ is singular and $\chi=1$ then (\ref{good component conjecture}) holds for every $y\in  Z(T)(\F_p^\alg)$.
\end{Prop}

\begin{proof}
  If we define rank one $\Z_p$-modules
 $$
P_0=\Hom(  \Q_p/\Z_p , \mathfrak{g}_0  ) 
\qquad
P_0^\vee=\Hom( \mathfrak{g}_0 , \mu_{p^\infty} ) 
 $$ 
 and $\mathbf{P}_0= P_0\otimes_{\Z_p} P_0^\vee$, then  the theory of Serre-Tate coordinates as in \cite[Theorem 7.2]{goren} provides a 
 canonical isomorphism of functors on $\mathbf{Art}$
$$
\Spf(R_{M_0} ) \iso \Hom_{\Z_p}( \mathbf{P}_0, \widehat{\mathbb{G}}_m ).
$$
If $S$ is an object of $\mathbf{Art}$ and $\phi: \mathbf{P}_0 \map{}\widehat{\mathbb{G}}_m(S)$ represents a deformation of 
$\mathfrak{g}_0$ to $S$,  the endomorphism $s_0$ of $\mathfrak{g}_0$ lifts to this deformation if and only if $\phi$ satisfies
$$
\phi( (s_0 x)\otimes y ) = \phi( x\otimes(s_0 y))
$$
for every $x\in P_0$ and $y\in P_0^\vee$.  It follows that there is an isomorphism of functors on $\mathbf{Art}$
$$
\Spf(R_{Y_0} ) \iso \Hom_{\Z_p}( \mathbf{P}_0 / \mathfrak{c}_0 \mathbf{P}_0, \widehat{\mathbb{G}}_m )
$$
where  $\mathfrak{c}_0=n \Z_p$ is the conductor of the order $\Z_p[s_0]$.   If we define rank one $\co_{F,p}$-modules
$$
P=\Hom(  \Q_p/\Z_p , \mathfrak{g}  ) 
\qquad
P^\vee=\Hom( \mathfrak{g} , \mu_{p^\infty} ) 
 $$ 
and $\mathbf{P}=P\otimes_{\co_{F,p}} P^\vee$, then similarly there are isomorphisms
$$
\Spf(R_{M} ) \iso \Hom_{\Z_p}( \mathbf{P}, \widehat{\mathbb{G}}_m )
$$
and 
$$
\Spf(R_{Y} ) \iso \Hom_{\Z_p}( \mathbf{P} / \mathfrak{c} \mathbf{P}, \widehat{\mathbb{G}}_m )
$$
where $\mathfrak{c}\subset \co_{F,p}$ is the conductor of $\co_{F,p}[t_\alpha]$.  Note that there are canonical isomorphisms
$$
P\iso P_0\otimes_{\Z_p} \co_{F,p} \qquad  P^\vee\iso P_0^\vee\otimes_{\Z_p} \omega
$$
and $\mathbf{P}\iso \mathbf{P}_0\otimes_{\Z_p} \omega$, where $\omega=\Hom_{\Z_p} ( \co_{F,p} ,\Z_p )$.  After fixing an isomorphism 
$\mathbf{P}_0\iso \Z_p$, the commutative diagram of functors on $\mathbf{Art}$ 
\begin{equation}\label{pole I}
\xymatrix{
{   \Spf(R_{Y_0})     } \ar[r] \ar[d] &  {\Spf(R_Y) }  \ar[d]   \\
{   \Spf(R_{M_0})   }  \ar[r]   &  {\Spf(R_M) } 
}
\end{equation}
becomes identified with 
\begin{equation}\label{pole II}
\xymatrix{
{   \Hom_{\Z_p}(\Z_p/\mathfrak{c}_0,   \widehat{\mathbb{G}}_m )  } \ar[r] \ar[d] &   
\Hom_{\Z_p}(\omega/\mathfrak{c}\omega,   \widehat{\mathbb{G}}_m )    \ar[d]   \\
{  \Hom_{\Z_p}(\Z_p,   \widehat{\mathbb{G}}_m )   }  \ar[r]   &  { \Hom_{\Z_p}(\omega,   \widehat{\mathbb{G}}_m )     } .
}
\end{equation}
Here the horizontal arrows are obtained by dualizing the $\Z_p$-module map $\mathrm{Tr}:\omega\map{}\Z_p$ 
defined by $\mathrm{Tr}(f)=f(1)$.

Set $c_0=\ord_p(n)$, so that $\mathfrak{c}_0=p^{c_0} \Z_p$, and define nonnegative integers $c_1,c_2$ as follows:
\begin{enumerate}
\item
if $F_p\iso \Q_p\times\Q_p$ then  $(p^{c_1},p^{c_2})\in\Z_p\times\Z_p\iso  \co_{F,p}$ generates the conductor of the order 
$\co_{F,p}[t_\alpha]$,
\item
if $F_p$ is an unramified field extension of $\Q_p$ then  $c_1=c_2=c_0$, 
\item
if $F_p$ is a ramified field extension of $\Q_p$ then let $\varpi_F$ be a uniformizer of $F_p$ and let $\varpi_F^c$ generate the 
conductor of the order $\co_{F,p}[t_\alpha]$.  If $c$ is even  define $c_1=c_2=c/2$; if $c$ is odd  define $c_1=(c-1)/2$ and 
$c_2=(c+1)/2$.
\end{enumerate}
One can check that in all cases $c_0=\mathrm{min}\{c_1,c_2\}$,
$$
 \ord_p(4\alpha\alpha^\sigma) - \ord_p(t)  = 2c_1+2c_2 - 2c_0,
$$
and $\omega$ admits a $\Z_p$-basis $\{e_1,e_2\}$ such that $\mathrm{Tr}(e_i)=1$ and $\{p^{c_1} e_1 , p^{c_2} e_2\}$ is a 
$\Z_p$-basis of  $\mathfrak{c}\omega$.  Using this basis of $\omega$ one identifies the diagram (\ref{pole II}) with
$$
\xymatrix{
{  \mu_{p^{c_0}}    } \ar[r] \ar[d] &  {   \mu_{p^{c_1}} \times  \mu_{p^{c_2}}     }  \ar[d]   \\
{  \widehat{\mathbb{G}}_m   }  \ar[r]   &  { \widehat{\mathbb{G}}_m\times  \widehat{\mathbb{G}}_m} 
}
$$
where the horizontal arrows are the diagonal maps and the vertical arrows are the natural inclusions.  The original diagram (\ref{pole I}) 
is now identified with
$$
\xymatrix{
{ \Spf(W[[x_0]]/(f_{c_0}(x_0))  ) } \ar[r]\ar[d]   &  {  \Spf(W[[x_1,x_2]] / (f_{c_1}(x_1),f_{c_2}(x_2)) ) } \ar[d]  \\
{  \Spf(W[[x_0]] ) } \ar[r]  & {  \Spf( W[[x_1,x_2]] )}
}
$$
where
$$
f_c(x) = (x+1)^{p^c}-1
$$
and the horizontal arrows are determined by $x_i\mapsto x_0$.

The calculation of the right hand side of (\ref{be wise}) is now reduced to a pleasant exercise. Let $M$ be the fraction field of $W$ and 
fix an embedding $M\map{}\C_p$. For each nonnegative integer $k$ set $\varphi_k(x)=\Phi_{p^k}(x+1)$ where $\Phi_{p^k}$ is the 
$p^k$-cyclotomic polynomial, let $X_k$ denote the roots of $\varphi_k(x)$ in $\C_p$, set $M_k=M(X_k)$, and let $W_k$ be the ring of 
integers of $M_k$.   The minimal primes of 
\begin{eqnarray*}
R_Y  &\iso& W [[x_1,x_2]]/( f_{c_1}(x_1), f_{c_2}(x_2)  )  \\
&\iso&   W[x_1]/(f_{c_1})\otimes_{W} W[x_2]/(f_{c_2})
\end{eqnarray*}
 are indexed by the $\Aut(\C_p/M)$-orbits of the set  
$$
(\mu_{p^{c_1}} \times \mu_{p^{c_2}}) (\C_p) =\bigsqcup_{ \substack{ 0\le k_1\le c_1  \\  0\le k_2\le c_2  }} (X_{k_1}\times X_{k_2} )
$$
 by the rule that attaches to the orbit $[\pi_1,\pi_2]$ of the pair $(\pi_1,\pi_2)\in X_{k_1}\times X_{k_2} $  the kernel $\mathfrak{p}$ of the 
 unique $W$-algebra homomorphism $R_Y \map{}\C_p$ taking $x_i\mapsto \pi_i$.  Assuming for simplicity that  $k_1\le k_2$, the 
 localization of $R_Y$ at $\mathfrak{p}$ is isomorphic to the cyclotomic field  $M_{k_2}$, and so $\mathrm{mult}(\mathfrak{p})=1$.   
 Under the above indexing the  proper minimal primes of $R_Y$ correspond to those orbits of the form $[\pi_1,\pi_2]$ with 
 $\pi_1\not=\pi_2$.  
 
 If $\mathfrak{p}\in\Pi^\proper(R_Y)$ is  indexed by  the orbit $[\pi_1,\pi_2]\subset X_{k_1}\times X_{k_2}$  then we will say that 
 $\mathfrak{p}$ has \emph{type} $(k_1,k_2)$.  Let $\Pi^\proper_{k_1,k_1}(R_Y)\subset\Pi^\proper(R_Y)$ be the subset of components 
 having type $(k_1,k_2)$.   Suppose   $k_1 < k_2$, fix some $\mathfrak{p}\in\Pi_{k_1,k_2}^\proper(R_Y)$, and let $[\pi_1,\pi_2]$ be the 
 corresponding orbit. There is an isomorphism $R_Y/\mathfrak{p}\iso W_{k_2}$
defined by $x_i\mapsto \pi_{i}$ and  isomorphisms
$$
R_Y/\mathfrak{p} \otimes_{R_M} R_{M_0} \iso W_{k_2}/ (\pi_2-\pi_1)\iso \F_p^\alg.
$$
 The second isomorphism is due to the fact that $\pi_2-\pi_1$ is a uniformizing parameter of $W_{k_2}$.    It is easy to see that the 
 number of $\Aut(\C_p/M)$-orbits in $X_{k_1}\times X_{k_2}$ is $|X_{k_1}|=[M_{k_1}:M]$, and applying the same reasoning in the case 
 $k_2<k_1$ shows that for  any $k_1\not= k_2$ we have
$$
\sum_{\mathfrak{p}\in\Pi^\proper_{k_1,k_2}(R_Y) } \mathrm{mult}(\mathfrak{p}) \cdot 
\length_{R_Y}(R_Y/\mathfrak{p} \otimes_W R_{M_0})  =|\Pi^\proper_{k_1,k_2}(R_Y) | = [M_{\mathrm{min}\{k_1,k_2\} } : M].
$$

Next fix  $0\le k\le c_0$ and one element $\pi\in X_k$.  The minimal primes of $R_Y$ contained in $\Pi_{k,k}^\proper(R_Y)$ 
correspond to the orbits $[\pi,\pi']$ as $\pi'$ ranges over $X_k\smallsetminus\{\pi\}$, and we find
\begin{eqnarray*}\lefteqn{
\sum_{\mathfrak{p}\in\Pi_{k,k}^\proper(R_Y) }  \mathrm{mult}(\mathfrak{p}) \cdot 
\length_{R_Y}(R_Y/\mathfrak{p} \otimes_W R_{M_0})  } \\
&=&
\sum_{\substack{ \pi'\in X_k \\ \pi'\not=\pi} } \length_{W_k} \big(W_k/(\pi-\pi')\big) \\
  &=& \ord_{\pi}(\mathrm{Diff}(M_k/ M ) )
\end{eqnarray*}
where $\mathrm{Diff}(M_k / M )$ is the relative different.    Combining (\ref{be wise}) with the equalities
$$
[M_k: M] = \left\{ \begin{array}{ll}1 & \mathrm{if\ }k=0 \\ p^{k-1}(p-1) & \mathrm{if\ }k>0 \end{array} \right.
$$
and (denoting by $\varpi_k$  any uniformizer of $M_k$ and using  \cite[Proposition 7.8.5]{KRY})
$$
\ord_{\varpi_k}(\mathrm{Diff}(M_k / M ) ) = 
\left\{\begin{array}{ll} 0 & \mathrm{if\ }k=0 \\ p^{k-1} (kp-k-1)  & \mathrm{if\ }k>0 \end{array} \right. 
$$
an elementary calculation  gives
\begin{eqnarray*}
I_{\co_{Y_0,y}}( [\co_Y]^\proper,\co_{M_0})  &=&  \sum_{ \mathfrak{p} \in \Pi^\proper(R_Y) } \mathrm{mult}(\mathfrak{p}) 
\cdot \length_{R_Y}(R_Y/\mathfrak{p} \otimes_W R_{M_0})      \\
&=& \sum_{ \substack{  0\le k_i\le c_i  \\  k_1\not=k_2  } } [ M_{\mathrm{min}\{k_1,k_2\} } : M ]  + \sum_{ 0\le k\le c_0 } 
\ord_{\varpi_k}(\mathrm{Diff}(M_k/M) ) \\
 &=&  p^{c_0} \cdot (c_1+c_2-c_0)  \\
 &=& \frac{1}{2}  \cdot  p^{\ord_p(n)}  \cdot  \ord_p\left(\frac{4\alpha\alpha^\sigma}{t}\right) 
 \end{eqnarray*}
as claimed.
 \end{proof}

\begin{Prop}
\label{Prop:good singular II}
If $T\in\Sigma(\alpha)$ is singular and  $p$ splits in $F$ then (\ref{good component conjecture}) holds for every $y\in  Z(T)(\F_p^\alg)$.
\end{Prop}

\begin{proof}
After Proposition \ref{Prop:good singular I} we may assume that $\chi\not=1$, so that $K_p/\Q_p$ is a quadratic field extension, and   we 
are in the supersingular case.   As $\End(\mathfrak{g}_0)$ is the maximal order in a nonsplit quaternion algebra over $\Q_p$, the 
embedding  $$\Z_p[x]/(x^2+t)\map{}\End(\mathfrak{g}_0)$$ determined by $s_0$ extends to an embedding 
$\co_{K,p}\map{}\End(\mathfrak{g}_0)$, and the action of $\co_{K,p}$ on $\mathrm{Lie}(\mathfrak{g}_0)$ is through a $\Z_p$-algebra 
homomorphism $\co_{K,p}\map{}\F_p^\alg$.   If we let $W_0$ be the completion of the strict Henselization of $\co_{K,p}$ with respect to 
this map, then $W_0$ is naturally a $W$-algebra satisfying
$$
W_0\iso \left\{ \begin{array}{ll}   W & \mathrm{if\ }\chi=-1 \\ \co_{K,p}\otimes_{\Z_p} W & \mathrm{if\ }\chi=0. \end{array} \right.
$$
The field $K_p$ is naturally a subfield of $M_0=\mathrm{Frac}(W_0)$, and by local class field theory $\co_{K,p}^\times$ is isomorphic to 
$\Gal(\widehat{K}^\mathrm{ab}/M_0)$, where $\widehat{K}^\mathrm{ab}$ is the completion of the maximal   abelian extension of $K$.  
Let $M_0\subset M_k\subset \widehat{K}^\mathrm{ab}$ be the subextension characterized by
$$
\co^\times_{K,p}/(\Z_p+p^k\co_{K,p})^\times \iso \Gal(M_k/M_0)
$$
and let $W_k$ be the integer ring of $M_k$.  As in \cite[\S 4.1]{howardB} there is an isomorphism of $W$-algebras $R_{M_0}\iso W[[x]]$, 
while the Gross-Keating theory of quasi-canonical lifts implies that 
$$
R_{Y_0}\iso W[[x]]/( f_{c_0}(x) )
$$
where $c_0=\ord_p(n)$ (so that $p^{c_0}$ generates the conductor of the quadratic $\Z_p$-order $\Z_p[s_0]$) and 
$$
f_{c_0}(x) = \prod_{k=0}^{c_0} \varphi_k(x)
$$
with each   $\varphi_k(x)$ an Eisenstein polynomial satisfying $W[[x]]/(\varphi_k(x)) \iso W_k$.

Fix an isomorphism $\co_{F,p}\iso \Z_p\times\Z_p$ and let $(p^{c_1}, p^{c_2})\in \co_{F,p}$ generate the conductor of the quadratic 
$\co_{F,p}$-order $\co_{F,p}[t_\alpha]$.  As in the proof of Proposition \ref{Prop:good singular I}, we have 
$c_0=\mathrm{min}\{c_1,c_2\}$ and 
$$
\ord_p\left(\frac{4\alpha\alpha^\sigma}{t}\right) = 2c_1+2c_2 -2c_0 + \ord_p(d_K).
$$
The induced splitting $\mathfrak{g}\iso \mathfrak{g}_0\times\mathfrak{g}_0$  determines isomorphisms 
$R_M\iso R_{M_0} \widehat{\otimes}_W R_{M_0}$ and 
$$
R_Y\iso  W[[x_1]]/( f_{c_1}(x_1) ) \widehat{\otimes}_W W[[x_2]]/(f_{c_2}(x_2)),
$$ 
and so the commutative diagram 
$$
\xymatrix{
{  \Spf(   R_{Y_0}  ) }  \ar[r]\ar[d]   &   {  \Spf( R_Y)  } \ar[d] \\
{   \Spf(  R_{M_0}  )   }  \ar[r]  &  {  \Spf( R_M )  } 
}
$$
of functors on $\mathbf{Art}$ can be identified with the commutative diagram
$$
\xymatrix{
{  \Spf(   W[[x_0]]/(f_{c_0}(x_0) )  ) }  \ar[r]\ar[d]   &   {  \Spf(W[[x_1,x_2]] /( f_{c_1}(x_1) , f_{c_2} (x_2) ))  } \ar[d] \\
{   \Spf(  W[[x_0]]  )   }  \ar[r]  &  {  \Spf( W[[x_1,x_2]] )  } 
}
$$
in which the horizontal arrows are determined by $x_i\mapsto x_0$.   Imitating the proof of Proposition \ref{Prop:good singular II}  
shows that
\begin{eqnarray*}
I_{\co_{Y_0,y}}( [\co_Y]^\proper,\co_{M_0})  &=&  \sum_{ \mathfrak{p} \in \Pi^\proper(R_Y) } \mathrm{mult}(\mathfrak{p}) 
\cdot \length_{R_Y}(R_Y/\mathfrak{p} \otimes_W R_{M_0})   \\
&=& \sum_{ \substack{  0\le k_i\le c_i  \\  k_1\not=k_2  } } [ M_{\mathrm{min}\{k_1,k_2\} } : M ]  + \sum_{ 0\le k \le c_0 } 
\ord_{\varpi_k}(\mathrm{Diff}(M_k/M) )
 \end{eqnarray*}
 where in the final sum $\varpi_k$ is a uniformizer of $M_k$.  The final sum can be computed using the formulas of 
 \cite[Proposition 7.7.7]{KRY} and \cite[Proposition 7.8.5]{KRY}.  If $\chi=-1$ then for all $k>0$
 $$
 [M_k:M] = p^{k-1}(p+1)
 $$
 and 
 $$
\ord_{\varpi_k}(\mathrm{Diff}(M_k/M) ) = kp^{k-1}(p+1) - \frac{p^k+p^{k-1} -2}{p-1}.
$$
If instead $\chi=0$, then for all $k \ge 0$ we have $[M_k:M] = 2p^k$ and 
$$
\ord_{\varpi_k}(\mathrm{Diff}(M_k/M) )  = 2kp^k - 2\frac{p^k-1}{p-1} +p^k\cdot\ord_p(d_K).
$$
Tedious but elementary calculation then results in
\begin{eqnarray*}\lefteqn{
I_{\co_{Y_0,y}}( [\co_Y]^\proper,\co_{M_0}) } \\
 &=& m_0  \left( \frac{p^{c_0+1} -1 }{p-1} -\chi \frac{p^{c_0} -1}{p-1} \right) \cdot \left(c_1+c_2-c_0+\frac{\ord_p(d_K)}{2}\right)   \\
 &=& \frac{m_0}{2} \cdot \left( \frac{p^{c_0+1} -1 }{p-1} -\chi \frac{p^{c_0} -1}{p-1} \right)\cdot
  \ord_p\left(\frac{4\alpha\alpha^\sigma}{t} \right) .
\end{eqnarray*}
\end{proof}

We are left to verify (\ref{good component conjecture}) in the supersingular case with $p$ nonsplit in $F$.  This is much harder than the 
cases treated in  Propositions \ref{Prop:good singular I} and \ref{Prop:good singular II}, and the bulk of the proof is contained in the 
separate article  \cite{howardB}.

\begin{Prop}
\label{Prop:good singular III}
Suppose that $T\in\Sigma(\alpha)$ is singular,   that  $p$ is odd, and that $p\co_F$ is relatively prime to   
$\gcd(\alpha\co_F,\mathfrak{D}_F)$.  Then (\ref{good component conjecture}) holds for every $y\in  Z(T)(\F_p^\alg)$.
\end{Prop}

\begin{proof}
After Propositions \ref{Prop:good singular I} and \ref{Prop:good singular II}, we may assume that $p$ is nonsplit in $F$ and that 
$\chi\not=1$ (so that we are in the supersingular case).  Set $c_0=\ord_p(n)$.

First assume that  $p$ is inert in $F$.    Combining (\ref{be wise}) with \cite[Theorem D]{howardB} 
 (where $K_p$ is denoted $E_0$) gives
$$
I_{\co_{Y_0},y}( [\co_Y]^\proper , \co_{M_0}  ) = c_0\left( \frac{p^{c_0 +1} -1}{p-1} + \frac{p^{c_0}-1}{p-1}  \right)
$$
if $\chi=-1$, and
$$
I_{\co_{Y_0},y}( [\co_Y]^\proper , \co_{M_0}  ) = (2c_0 +1 )\frac{p^{c_0+1}-1}{p-1}
$$
if $\chi=0$.  If we set $\mathfrak{p}=p\co_F$ then 
$$
\ord_\mathfrak{p}(\alpha)=\ord_\mathfrak{p}(\alpha^\sigma) = \ord_p(t)
$$
which, when combined with $4t=-n^2d_K$ and $p\not=2$,  implies that
$$
 \ord_p\left(\frac{4\alpha\alpha^\sigma}{t} \right) = 2c_0 +(\chi+1).
$$
Thus (\ref{good component conjecture}) holds.

We are left with the case of $p$ ramified in $F$.  Write $p\co_F=\mathfrak{p}^2$ and note that our hypotheses on $p$ imply that 
$\mathfrak{p}$ is relatively prime to $\alpha\co_F$.  It follows that   $t$ is relatively prime to $p$, and hence from $4t=-n^2d_K$ that 
$c_0=0$ and that $K$ is unramified at $p$.  In particular  the right hand side of (\ref{good component conjecture}) is equal to $0$.   We 
may now invoke \cite[Proposition 5.1.1]{howardB}, which asserts that $R_{Y_0}\iso R_Y\iso W$.  Therefore   
$\Pi^\proper(R_Y)=\emptyset$ and the right hand side of (\ref{be wise}) (and so also the left hand side of 
(\ref{good component conjecture})) is equal to $0$.  
\end{proof}

Define a codimension two cycle on $M$ by
\begin{equation}\label{good cycle}
C_p^\good = \sum_{D\in\Pi^\good(Y)} \mathrm{mult}_{D}(\co_Y)  \cdot \phi(D),
\end{equation}
where  $\phi(D)$ is viewed as a closed subscheme of $M$ with its reduced subscheme structure.  Define $C_p^\bad$ and 
$C_p^\vertical$ in the same way, replacing $\Pi^\good(Y)$ with $\Pi^\bad(Y)$ or $\Pi^\vertical(Y)$, respectively.  Each of $C_p^\good$, 
$C_p^\bad$, or $C_p^\vertical$ is $H$-invariant and so (by \cite[Lemma 4.2]{gillet84}) determines a cycle on $\mathcal{M}_{/\Z_p}$, 
which we  denote by $\mathcal{C}_p^\good$, $\mathcal{C}_p^\bad$, or $\mathcal{C}_p^\vertical$.     The sum 
$\mathcal{C}_p^\good+\mathcal{C}_p^\bad+\mathcal{C}_p^\vertical$ represents  the cycle class 
$\mathcal{C}_p\in\mathrm{CH}^2_{\mathcal{Y}_{/\Z_p}}(\mathcal{M}_{/\Z_p})$ constructed in \cite[\S 3.3]{howardA}.

\begin{Prop}\label{Prop:main unramified}
Assume that at least one of the following hypotheses holds:
\begin{enumerate}
\item $F(\sqrt{-\alpha})/\Q$ is not biquadratic,
\item $p$ splits in $F$, 
\item  $p$ is odd and $p\co_F$ is relatively prime to $\gcd(\alpha\co_F,\mathfrak{D}_F)$.
 \end{enumerate}
 Then
\begin{eqnarray*} \lefteqn{
I_p( \mathcal{C}_p^\good, \mathcal{M}_0 )  +
I_p(  \mathcal{C}_p^\vertical, \mathcal{M}_0 )   = 
 \frac{ 1 }{2} \sum_{  \substack{ T\in\Sigma(\alpha)  \\   \det(T)=0  }  } 
\deg_\Q(\mathcal{Z}(t))  \cdot  \ord_p\left(\frac{ 4 \alpha\alpha^\sigma}{t}\right)   }   \hspace{2cm}  \\
 & & +
 \sum_{  \substack{T \in \Sigma(\alpha)   \\ \det(T)\not=0  } }  
 \sum_{y\in \mathcal{Z}(T)(\F_p^\alg)} e_y^{-1} \cdot  \length_{  \co_{\mathcal{Z}(T),y}^\mathrm{sh} } 
 ( \co^{\mathrm{sh}}_{\mathcal{Z}(T),y} )   .
  \end{eqnarray*}
On the right hand side  $t$ is defined by (\ref{little t}) and $\co^{\mathrm{sh}}_{\mathcal{Z}(T),y}$ is the strictly Henselian local ring of 
$\mathcal{Z}(T)$ at $y$.  The rational number $\deg_\Q(\mathcal{D})$ is defined by (\ref{generic degree}) for an irreducible cycle 
$\mathcal{D}$ of codimension two on $\mathcal{M}$, and extended linearly to all codimension two cycles. 
\end{Prop}

\begin{proof}
If (a) holds than $\Sigma(\alpha)$ contains no singular matrices by \cite[Lemma 3.1.5(c)]{howardA}, 
and so (\ref{scheme decomp}) and  Proposition \ref{Prop:good nonsingular} imply
 \begin{eqnarray*}\lefteqn{ 
 I_p(C_p^\good, M_0) + I_p(C_p^\vertical,M_0) } \\
 &  =  & \sum_{y\in Y_0(\F_p^\alg) }  \big( I_{\co_{Y_0,y}} (  [\co_Y]^\good  ,   \co_{M_0}    )     +  
  I_{\co_{Y_0,y}} (  [\co_Y]^\vertical  ,   \co_{M_0}    )  \big)  \nonumber \\
& = &  \sum_{ T\in\Sigma(\alpha)   } \sum_{y\in Z(T)(\F_p^\alg)} 
\length_{  \co_{Z(T),y}  } ( \co_{Z(T),y}  )  .
 \end{eqnarray*}
If (b) holds then combining   (\ref{scheme decomp}) with  Propositions \ref{Prop:good nonsingular} and  \ref{Prop:good singular II} gives
 \begin{eqnarray*}\lefteqn{ 
 I_p(C_p^\good, M_0) + I_p(C_p^\vertical,M_0) } \\
 &  =  & \sum_{y\in Y_0(\F_p^\alg) }  \big( I_{\co_{Y_0,y}} (  [\co_Y]^\good  ,   \co_{M_0}    )     +  
  I_{\co_{Y_0,y}} (  [\co_Y]^\vertical  ,   \co_{M_0}    )  \big)  \\
& = &  \sum_{ \substack{  T\in\Sigma(\alpha)   \\  \det(T)\not=0   } } \sum_{y\in Z(T)(\F_p^\alg)} 
\length_{  \co_{Z(T),y}  } ( \co_{Z(T),y}  )   \\
& & + \sum_{ \substack{  T\in\Sigma(\alpha)   \\  \det(T) =0   } } 
  \frac{1}{2} \cdot  \Gamma_p(T)   \cdot \ord_p\left(\frac{ 4 \alpha\alpha^\sigma}{t}\right)  \cdot | Z(T)(\F_p^\alg)|. 
 \end{eqnarray*}
 If (c) holds then one obtains the same equalities by replacing Proposition \ref{Prop:good singular II} with Proposition 
 \ref{Prop:good singular III}.    In all cases the desired result  follows by dividing the above equalities by  $|H|$ and using the equality
$$
| Z(t)(\Q^\alg) |  =  \Gamma_p(T)  \cdot  |Z(T)(\F_p^\alg)|
$$
of \cite[Proposition 7.7.7(ii)]{KRY} for each singular $T\in\Sigma(\alpha)$.
\end{proof}


\section{Ramified intersection theory}
\label{s:bad reduction}


We continue with the work of the previous section, but now work in characteristic dividing the discriminant of $B_0$.  
The situation is complicated by the fact that $\mathcal{Y}(\alpha)$ may have vertical components of dimension $2$,
which must be removed and replaced, following the constructions of \cite{howardA}, by new vertical components of dimension $1$.
We intersect $\mathcal{M}_0$ against these new components, and against those horizontal components of $\mathcal{Y}(\alpha)$
that meet $\mathcal{M}_0$ properly. Our calculations rely heavily on the earlier work of Kudla-Rapoport-Yang \cite{kudla00,kudla04a}.

As in \S \ref{s:good reduction}, fix a totally positive $\alpha\in\co_F$.  
Fix a prime $p$ that divides $\mathrm{disc}(B_0)$ and recall from the introduction our hypothesis that all such primes are split in $F$.   
Let $W=W(\F_p^\alg)$ be the ring of Witt vectors of 
$\F_p^\alg$, let  $\A_f$ be the ring of finite adeles of $\Q$, and let $\A_f^p$ be  the prime-to-$p$ part of $\A_f$.  Define  compact open 
subgroups of $G_0(\A_f)$ and  $G(\A_f)$ by
$$
U_0^\mathrm{max}=\widehat{\co}_{B_0}^\times
\qquad
U^\mathrm{max} =\{ b\in  \widehat{\co}_B^\times   :  \mathrm{Nm}(b)\in \widehat{\Z}^\times \},
$$
 and  choose a normal compact 
open subgroup  $U\subset U^\mathrm{max}$  of the form $U=U_p U^p$ with $U_p\subset G(\Q_p)$ and $U^p\subset G(\A_f^p)$.  We 
assume $U_p=U_p^\mathrm{max}$.    For sufficiently small such $U$ there is an isomorphism of DM stacks 
$\mathcal{M}_{/\Z_p}\iso [H\backslash M]$, where $H=U^\mathrm{max}/U$, and $M$ is the $\Z_p$-scheme representing the functor that 
assigns to a $\Z_p$-scheme $S$ the set of isomorphism classes of  $\mathfrak{D}_F^{-1}$-polarized QM abelian fourfolds over $S$ 
equipped with a $U$-level structure in the sense of \cite[\S 3.1]{howardA}.  Having chosen such a presentation of $\mathcal{M}_{/\Z_p}$ 
let $M_0$, $Y_0$, $Y$, $Z(t)$, and $Z(T)$ have the same meaning as in \S \ref{s:good reduction}.

Recall that for a Noetherian scheme $X$, we let $\mathbf{K}_0(X)$ be the Grothendieck group of the category of coherent 
$\co_X$-modules, and that the class of such a coherent $\mathcal{F}$ in $\mathbf{K}_0(X)$ is denoted $[\mathcal{F}]$.  
We denote by $\mathbf{K}_0^\vertical(X)$  the Grothendieck group of the category of  locally $\Z_p$-torsion coherent $\co_X$-modules.  
As $X$ is quasi-compact every locally $\Z_p$-torsion coherent $\co_X$-module $\mathcal{F}$ satisfies $p^n\mathcal{F}=0$ for some 
sufficiently large $n$.  As in \S \ref{s:good reduction} let  $\Pi^\good(Y)$ (respectively $\Pi^\bad(Y)$) be the set of horizontal components 
of $Y$  that are not contained in $Y_0$ (respectively are  contained in $Y_0$), and view each such component as a closed subscheme 
of $Y$ with its reduced subscheme structure.  Let $Y^\good\subset Y$ be the union of all $D\in\Pi^\good(Y)$ with its reduced 
subscheme structure and define $Y^\bad$ similarly.  Using the notation of \S \ref{s:good reduction} define classes in 
$\mathbf{K}_0(Y^\good)$ and $\mathbf{K}_0(Y^\bad)$ by
$$
[\co_Y]^\good = \sum_{  D\in\Pi^\good(Y)  } [\co_Y]_D
\qquad
[\co_Y]^\bad = \sum_{  D\in\Pi^\bad(Y)  } [\co_Y]_D.
$$

While the horizontal components of $Y$ are all of dimension one,  $Y$ may have vertical components of dimension two. Thus while we 
may define codimension two cycles $\mathcal{C}_p^\good$ and $\mathcal{C}_p^\bad$ on $\mathcal{M}_{/\Z_p}$ exactly as in 
(\ref{good cycle}), the construction of a codimension two cycle $\mathcal{C}_p^\vertical$ will  proceed by the roundabout construction of 
an auxiliary class $[\mathfrak{O}_Y]^\vertical \in\mathbf{K}^\vertical_0(Y)$ to serve as a substitute for the naive class $[\co_Y]^\vertical$.  
To construct this class, recall some notation and constructions from \cite[\S 4]{howardA}.    Fix a principally polarized QM abelian 
surface $(\mathbf{A}^*_0,\lambda_0^*)$ over $\F_p^\alg$,   set 
$$
(\mathbf{A}^*,\lambda^*) = (\mathbf{A}^*_0,\lambda^*_0)\otimes\co_F,
$$
 and define totally definite quaternion algebras over $\Q$ and $F$, respectively,
$$
\overline{B}_0=\End^0(\mathbf{A}^*_0 ) \qquad \overline{B}=\End^0(\mathbf{A}^*)
$$
so that $\overline{B}_0\otimes_\Q F \iso \overline{B}$.  Let $\overline{G}_0\subset \overline{G}$ be the algebraic groups over $\Q$ 
defined in the same way as $G_0\subset G$, but with $B_0$ and $B$ replaced by $\overline{B}_0$ and $\overline{B}$.   
Let $\widehat{\Lambda}_0 $ and $\widehat{\Lambda}$ be the profinite completions of $\co_{B_0}$ and $\co_B$, respectively, and let 
$$
\widehat{\Lambda}_0^p \iso\widehat{\co}_{B_0}^p
\qquad
\widehat{\Lambda}^p\iso \widehat{\co}_B^p
$$ 
be their prime-to-$p$ parts.  As in \cite[\S 4.1]{howardA}, fix an isomorphism of $\A_f^p$-modules
$$
\nu_0^*: \widehat{\Lambda}_0^p\map{} \mathrm{Ta}^p(A_0^*)
$$
where on the right $\mathrm{Ta}^p$ is the prime-to-$p$ adelic Tate module of the underlying abelian variety 
$A_0^*$ of $\mathbf{A}_0^*$.   This isomorphism is assumed to respect the left $\co_{B_0}$ action on both sides, and to identify the 
Weil pairing on the right induced by $\lambda_0$ with the pairing $\psi_0$ on the left defined in \cite[\S 3.1]{howardA}.  By tensoring 
with $\co_F$, the choice of $\nu_0$ induces an isomorphism
$$
\nu^*: \widehat{\Lambda}^p\map{} \mathrm{Ta}^p(A^*).
$$
Each  $g\in\overline{B}_0\otimes_\Q\A_f^p$ acts as a $B_0$-linear endomorphism  of  $\mathrm{Ta}^p(A_0^*)\otimes_\Z\Q$, and so 
also acts (using $\nu_0$) as a $B_0$-linear endomorphism  of $\widehat{\Lambda}_0^p\otimes_\Z\Q$.    As the action of $B_0$ on 
$\widehat{\Lambda}_0^p\otimes_\Z\Q$ is by left multiplication, the  endomorphism of $\widehat{\Lambda}_0^p\otimes_\Z\Q$ 
determined by $g$ is  given by right multiplication by some $\iota_0(g)\in  B_0\otimes_\Q\A_f$.  In this way the choice of $\nu_0^*$ 
determines a bijection
$$
\iota_0: \overline{G}_0(\A_f^p)\map{}G_0(\A_f^p),
$$
which satisfies $\iota_0(xy)=\iota_0(y)\iota_0(x)$.  Similarly $\nu^*$ determines a bijection
$$
\iota:\overline{G}(\A_f^p)\map{}G(\A_f^p)
$$
satisfying $\iota(xy)=\iota(y)\iota(x)$.  The induced bijection between subgroups of $\overline{G}_0(\A_f^p)$ and subgroups of 
$G_0(\A_f^p)$ is denoted $\overline{H}^p\leftrightarrow H^p$, and similarly with $G_0$ replaced by $G$.

 Let  $\mathfrak{G}^*_0$   denote the $p$-divisible group of $\mathbf{A}^*_0$ equipped with its action of $\co_{B_0}\otimes_\Z\Z_p$.   
 We denote by $\mathfrak{h}_m$ Drinfeld's formal $W$-scheme representing the functor that assigns to every $W$-scheme $S$ on 
 which $p$ is locally nilpotent the set $\mathfrak{h}_m(S)$ of isomorphism classes of pairs  $(\mathfrak{G}_0,\rho_0)$, in which  
 $\mathfrak{G}_0$ is a special formal $\co_{B_0}\otimes_\Z\Z_p$-module (in the sense of \cite[\S II.2]{boutot-carayol}) of dimension two 
 and height four over $S$ and  
$$
\rho_0 : \mathfrak{G}^*_0\times_{\F_p^\alg}  S_{/\F_p^\alg} \map{} \mathfrak{G}_0\times_S S_{/\F_p^\alg}
$$ 
is a height $2m$ quasi-isogeny of $p$-divisible groups over $S_{/\F_p^\alg}$  respecting the action of $\co_{B_0}\otimes_\Z \Z_p$.     
The group $\overline{G}_0(\Q_p)$ acts on the formal $W$-scheme
$$
\mathfrak{X}_0=\bigsqcup_{m\in\Z}\mathfrak{h}_m
$$ 
by 
$$
\gamma\cdot (\mathfrak{G}_0,\rho_0)=  (\mathfrak{G}_0, \rho_0\circ \gamma^{-1}).
$$
As $p$ splits in $F$, fix an isomorphism $F\otimes_\Q\Q_p\iso \Q_p\times\Q_p$.  This, in turn, determines an isomorphism
$$
\overline{G}(\Q_p)\iso \{ (x,y)\in \overline{G}_0(\Q_p)\times \overline{G}_0(\Q_p)  : \mathrm{Nm}(x)=\mathrm{Nm}(y) \}
$$ 
and hence an action of $\overline{G}(\Q_p)$ on 
$$
\mathfrak{X}=\bigsqcup_{m\in\Z} (\mathfrak{h}_m\times_W\mathfrak{h}_m).
$$
For any $W$-scheme $X$ denote by $\widehat{X}$ the formal completion of $X$ along its special fiber, a formal $W$-scheme.

 There is a \v Cerednik-Drinfeld style  isomorphism of formal $W$-schemes
\begin{equation}\label{CD}
\widehat{M}_{/W} \iso \overline{G}(\Q)\backslash \mathfrak{X}\times \overline{G}(\A_f^p) /\overline{U}^p,
\end{equation}
and from \cite[\S 4]{howardA} we have a commutative diagram of formal $W$-schemes
 \begin{equation}\label{formal cartesian}
 \xymatrix{
 &    {\mathfrak{M}^1  }  \ar[dr]   \\
    {\widehat{Y}_{/W} }  \ar[dr]\ar[ur]   &  & {\mathfrak{M} }  \ar[r] & {\widehat{M}_{/W}}    \\
 &  {\mathfrak{M}^2 }  \ar[ur]
 }
 \end{equation}
 in which the square is cartesian and all arrows in the square are closed immersions.   We quickly recall the definitions of 
 $\mathfrak{M}$ and $\mathfrak{M}^k$.  Let $\overline{V}$ denote the $F$-vector space of elements of $\overline{B}$ having reduced 
 trace zero.  We equip $\overline{V}$ with the $F$-valued quadratic form $\overline{Q}(\tau)=-\tau^2$ and let $\overline{G}(\Q)$ act on 
 $\overline{V}$ by conjugation.  Similarly let $\overline{V}_0$ be the trace zero elements of $\overline{B}_0$ equipped with the 
 conjugation action of $\overline{G}_0(\Q)$ and with the quadratic form $\overline{Q}_0(\tau_0)=-\tau_0^2$.  For each 
 $\tau\in \overline{V}\otimes_\Q \Q_p$ write $(\tau_1,\tau_2)$ for the image of $\tau$ under
 \begin{equation}\label{p splitting}
 \overline{V}\otimes_\Q \Q_p \iso  (\overline{V}_0\otimes_\Q\Q_p) \times (\overline{V}_0\otimes_\Q\Q_p).
 \end{equation}
  For every $W$-scheme $S$ on which $p$ is locally nilpotent and every $\tau_0\in \overline{V}_0\otimes_\Q\Q_p$, viewed as a 
  quasi-endomorphism of  $\mathfrak{G}_0^* \times_{\F_p^\alg} S_{/\F_p^\alg} $,  let  
 $$
 \mathfrak{h}_m(\tau_0)(S)\subset\mathfrak{h}_m(S)
 $$ 
 be the subset consisting of those pairs $(\mathfrak{G}_0,\rho_0)$ for which the quasi-endomorphism
$$
\rho_0 \circ \tau_0\circ \rho_0^{-1} \in \End(\mathfrak{G}_0\times_S S_{/\F_p^\alg} ) \otimes_{\Z_p}\Q_p
$$
lies in the image of 
$$
\End(\mathfrak{G}_0)    \map{} \End (\mathfrak{G}_0\times_S S_{/\F_p^\alg})  .
$$
The functor $\mathfrak{h}_m(\tau_0)$ is represented  by a closed formal subscheme of $\mathfrak{h}_m$, and we define
$$
\mathfrak{X}_0(\tau_0)=\bigsqcup_{m\in\Z} \mathfrak{h}_m(\tau_0).
$$
For each $\tau \in \overline{V}$ define closed formal subschemes of $\mathfrak{X}$  by
 $$
 \mathfrak{X}^1(\tau)= \bigsqcup_{m\in\Z} (\mathfrak{h}_m(\tau_1)\times_W\mathfrak{h}_m)
 \qquad
 \mathfrak{X}^2(\tau)= \bigsqcup_{m\in\Z} (\mathfrak{h}_m\times_W\mathfrak{h}_m(\tau_2) )
 $$
 and 
 $$
 \mathfrak{X}(\tau)=\mathfrak{X}^1(\tau)\times_{\mathfrak{X}}\mathfrak{X}^2(\tau).
 $$
Define a right $\overline{U}^p$-invariant  compact open subset $\Omega(\tau)\subset \overline{G}(\A_f^p)$ by
$$
\Omega(\tau) =  \{ g \in \overline{G}(\A_f^p) : \widehat{\Lambda}^p \cdot \iota(g^{-1}\tau g) \subset \widehat{\Lambda}^p  \}.
$$
Thus  $\gamma\cdot \Omega(\tau)=\Omega(\gamma\tau\gamma^{-1})$ for each $\gamma\in\overline{G}(\Q)$, and there is an 
isomorphism of formal $W$-schemes 
\begin{equation}\label{super CD}
\widehat{Y}_{/W}\iso \overline{G}(\Q)\backslash  \bigsqcup_{  \substack{\tau \in \overline{V}  \\ \overline{Q}(\tau)=\alpha } }    
\left( \mathfrak{X} ( \tau ) \times \Omega(\tau)  / \overline{U}^p \right) .
\end{equation}
The formal $W$-scheme $\mathfrak{M}$  in (\ref{formal cartesian}) is defined by replacing $\mathfrak{X}(\tau)$ by $\mathfrak{X}$ in the 
right hand side of (\ref{super CD}), and $\mathfrak{M}^k$ is defined by replacing $\mathfrak{X}(\tau)$ by $\mathfrak{X}^k(\tau)$.

If $\mathrm{Frob}$ denotes the (arithmetic) Frobenius automorphism $W\map{}W$ then the formal schemes $\mathfrak{M}$ and 
$\mathfrak{M}^k$ come equipped with isomorphisms of formal $W$-schemes (described at the beginning of \cite[\S 4.2]{howardA})
$$
\mathfrak{M}^{\mathrm{Frob}}\iso \mathfrak{M}
\qquad
\mathfrak{M}^{k,\mathrm{Frob}} \iso \mathfrak{M}^k,
$$
which are compatible with the morphisms in (\ref{formal cartesian}), and with the evident isomorphisms 
$$
(\widehat{Y}_{/W})^\mathrm{Frob} \iso \widehat{Y}_{/W} \qquad
(\widehat{M}_{/W})^\mathrm{Frob} \iso \widehat{M}_{/W}.
$$
In other words, the entire diagram (\ref{formal cartesian}) is invariant under base change by $\mathrm{Frob}$. By Grothendieck's 
theories of faithfully flat descent and formal GAGA,  any coherent $\co_{\widehat{Y}_{/W}}$-module that is invariant under 
$\mathrm{Frob}$ descends to a coherent $\co_Y$-module.
In particular if $\mathfrak{F}^1$ and $\mathfrak{F}^2$ are coherent sheaves on $\co_{\mathfrak{M}^1}$ and $\co_{\mathfrak{M}^2}$ 
respectively, each invariant under $\mathrm{Frob}$, then for every $\ell$ the coherent $\co_{\mathfrak{M}}$-module 
$\Tor_\ell^{\co_{\mathfrak{M}}}(\mathfrak{F}^1,\mathfrak{F}^2)$  is invariant under $\mathrm{Frob}$ and is annihilated by the ideal sheaf 
of the closed formal subscheme 
$$
\widehat{Y}_{/W}\iso \mathfrak{M}^1\times_\mathfrak{M}\mathfrak{M}^2 \hookrightarrow \mathfrak{M}.
$$
Thus we may view $\Tor_\ell^{\co_{\mathfrak{M}}}(\mathfrak{F}^1,\mathfrak{F}^2)$ as a coherent $\co_Y$-module and form
\begin{equation}\label{derived tensor}
[\mathfrak{F}^1\otimes^L_{\co_\mathfrak{M}} \mathfrak{F}^2 ] =
 \sum_{\ell\ge 0} (-1)^\ell [\Tor^{\co_{\mathfrak{M}}}_\ell ( \mathfrak{F}^1 , \mathfrak{F}^2 ) ] \in \mathbf{K}_0(Y).
\end{equation}
For $k\in\{1,2\}$ let $\mathfrak{B}^k$ be the largest ideal sheaf of $\co_{\mathfrak{M}^k}$ that is locally 
$W$-torsion, and define $\mathfrak{A}^k$ by the exactness of 
$$
0\map{}\mathfrak{B}^k  \map{} \co_{\mathfrak{M}^k} \map{}\mathfrak{A}^k \map{}0.
$$
We then have a decomposition in $\mathbf{K}_0(Y)$
\begin{equation}\label{ramified h-v}
[  \co_{\mathfrak{M}^1} \otimes^L_{\co_{\mathfrak{M}}}   \co_{\mathfrak{M}^2} ]  
=  [\mathfrak{O}_Y] ^\horizontal + [\mathfrak{O}_Y] ^\vertical
\end{equation}
in which 
\begin{eqnarray*}
 [\mathfrak{O}_Y] ^\horizontal  &=& [\mathfrak{A}^1\otimes^L_{\co_{\mathfrak{M}}} \mathfrak{A}^2 ]  \\
{  [\mathfrak{O}_Y] ^\vertical  }&=&  [\mathfrak{A}^1  \otimes^L_{\co_{\mathfrak{M}}}   \mathfrak{B}^2 ]  
+  [\mathfrak{B}^1 \otimes^L_{\co_{\mathfrak{M}}} \mathfrak{A}^2 ]   + [\mathfrak{B}^1  \otimes^L_{\co_{\mathfrak{M}}} \mathfrak{B}^2  ] .  
\end{eqnarray*}
 By its construction,  the class $[\mathfrak{O}_Y]^\vertical$ may be viewed also as an element of  $\mathbf{K}^\vertical_0(Y)$.

\begin{Prop}\label{Prop:bad hor decomp}
There is a closed subscheme $J\map{}Y$ of dimension zero such that 
$$
 [\mathfrak{O}_Y] ^\horizontal   -  [\co_Y]^\good  -  [\co_Y]^\bad 
  \in \mathrm{Image}  \left(  \mathbf{K}_0^J(Y)  \map{}  \mathbf{K}_0(Y)  \right).
$$
\end{Prop}

\begin{proof}
Given classes  $[\mathcal{F}]$ and $[\mathcal{G}]$ in $\mathbf{K}_0(Y)$, write $[\mathcal{F}]\sim [\mathcal{G}]$ to mean that
$$
[\mathcal{F}] - [\mathcal{G}] \in \mathrm{Image} \big(  \mathbf{K}_0^J(Y)  \map{} \mathbf{K}_0(Y)  \big)
$$ 
for some closed subscheme $J\map{}Y$ of dimension zero.  Let us say that a coherent $\co_Y$-module $\mathcal{A}$ is 
\emph{essentially horizontal} if every point $\eta\in Y$ in the support of $\mathcal{A}$ is either the generic point of a horizontal 
component of $Y$ or is a closed point of $Y$.  That is, the support of an essentially horizontal $\co_Y$-module $\mathcal{A}$ contains 
no vertical components of dimension greater than zero.  Note that any subquotient of an essentially horizontal $\co_Y$-module is again 
essentially horizontal.  If $\mathcal{A}$ is essentially horizontal then an induction argument using the exact sequence of coherent 
$\co_Y$-modules
$$
0 \map{}  \mathcal{A}'  \map{}  \mathcal{A}  \map{} \mathcal{A}\otimes_{\co_Y} \co_D\map{}0
$$
and the methods of \cite[Lemma 2.2.1]{howardA} show that,  in the notation of \S \ref{s:good reduction},
\begin{equation}\label{horizontal breakdown}
[\mathcal{A}] \sim \sum_{ D\in \Pi^\horizontal(Y) } \mathrm{mult}_D(\mathcal{A}) \cdot [\co_D]
\end{equation}
where $\Pi^\horizontal(Y)$ is the set of irreducible horizontal components of $Y$, each endowed with its reduced subscheme structure.

  Let $\mathcal{B}$ be the subsheaf of locally $\Z_p$-torsion sections of $\co_Y$ and define $\mathcal{A}$ by the exactness of 
$$
0 \map{} \mathcal{B} \map{} \co_Y \map{} \mathcal{A} \map{} 0.
$$
The sheaf  $\mathcal{A}$  has no $\Z_p$-torsion local sections, and  the stalks of $\mathcal{A}$ are $\Z_p$-torsion free.   
To see that $\mathcal{A}$ is essentially horizontal, suppose that $\eta\in Y$ has residue characteristic $p$ and  Zariski closure of 
dimension greater than zero. This implies that no horizontal component of $Y$ passes through $\eta$, and hence that every prime ideal 
of $\co_{Y,\eta}$ has residue characteristic $p$.  We deduce that $\co_{Y,\eta}[1/p]$ is the trivial ring, and so $\co_{Y,\eta}$ is itself 
$\Z_p$-torsion.  But then $\mathcal{A}_\eta$ is both $\Z_p$-torsion and $\Z_p$-torsion free, hence $\mathcal{A}_\eta=0$ as desired.   
As the support of $\mathcal{B}$ is contained in the special fiber of $Y$,  $ \mathrm{mult}_D(\mathcal{A}) = \mathrm{mult}_D(\co_Y) $  for 
every $D\in\Pi^\horizontal(Y)$, and we now deduce from (\ref{horizontal breakdown}) that
$$
[\mathcal{A}]\sim  [\co_Y]^\good  +  [\co_Y]^\bad.
$$

Fix a closed point $x\in \mathfrak{M}$, let $R$ be the completion of the local ring $\co_{\mathfrak{M},x}$, let $N^k$ be the completed 
stalk at $x$ of the $\co_{\mathfrak{M}}$-module $\co_{\mathfrak{M}^k}$, and let $P^k$ be the maximal $W$-torsion free quotient of 
$N_k$ (which is isomorphic as an $R$-module to the completed stalk of $\mathfrak{A}^k$ at $x$).  It follows from the proof of   
\cite[Proposition 4.2.5]{howardA} that  
$$
\Tor_\ell^R(P^1,P^2) =0
$$
for all $\ell >0$.  Thus $\Tor_\ell^{\co_\mathfrak{M}}(\mathfrak{A}^1,\mathfrak{A}^2)$ has trivial stalks, and so
$$
 [\mathfrak{O}_Y] ^\horizontal  = [\mathfrak{A}^1\otimes_{\co_{\mathfrak{M}}} \mathfrak{A}^2 ] .
$$
The proof of  \cite[Proposition 4.2.5]{howardA} also shows  that $P^1\otimes_R P^2$ is free of finite rank over $W$; in other words the 
completed stalks at closed points of the $\co_{Y_{/W}}$-module $\mathfrak{A}^1\otimes_{\co_{\mathfrak{M}}} \mathfrak{A}^2$ are free of 
finite rank over $W$.   It follows that all stalks of the $\co_{Y_{/W}}$-module 
$\mathfrak{A}^1\otimes_{\co_{\mathfrak{M}}} \mathfrak{A}^2$ are $W$-torsion free, from which one easily deduces  from the faithful 
flatness of $\Z_p\map{}W$ that the local sections of the coherent $\co_Y$-module 
$\mathfrak{A}^1\otimes_{\co_{\mathfrak{M}}} \mathfrak{A}^2$ are $\Z_p$-torsion free.  The kernel of the evident surjection of $\co_Y$-
modules
$$
 \co_{\mathfrak{M}^1}\otimes_{\co_\mathfrak{M}} \co_{\mathfrak{M}^2}\map{}  \mathfrak{A}^1\otimes_{\co_{\mathfrak{M}}} \mathfrak{A}^2
$$
is generated by the ideal sheaves $\mathfrak{B}^1\otimes_{\co_\mathfrak{M}} \co_{\mathfrak{M}^2}$ and 
$\co_{\mathfrak{M}^1}\otimes_{\co_\mathfrak{M}} \mathfrak{B}^2$, and so is locally $\Z_p$-torsion.  Using the isomorphism 
$\co_Y\iso \co_{\mathfrak{M}^1}\otimes_{\co_\mathfrak{M}}\co_{\mathfrak{M}^2}$, we therefore deduce that 
$\mathfrak{A}^1\otimes_{\co_\mathfrak{M}} \mathfrak{A}^2$ is the maximal quotient sheaf of $\co_Y$ whose local sections are 
$\Z_p$-torsion free.  In other words
$\mathfrak{A}^1\otimes_{\co_\mathfrak{M}} \mathfrak{A}^2 \iso \mathcal{A}$.  Thus $[\mathfrak{O}_Y]^\horizontal = [\mathcal{A}]$ and
$$
[\mathfrak{O}_Y]^\horizontal - [\co_Y]^\good  -  [\co_Y]^\bad \sim [\mathcal{A}]-[\mathcal{A}] =0.
$$
\end{proof}

 Given a coherent $\co_Y$-module $\mathcal{F}$, we regard $\Tor^{\co_Y}_\ell(\mathcal{F},\co_{Y_0})$ as a coherent $\co_{Y_0}$-
 module and define
$$
[\mathcal{F}\otimes^L_{\co_Y}\co_{Y_0} ] = \sum_{\ell \ge 0} (-1)^\ell [\Tor_\ell^{\co_Y}(\mathcal{F} , \co_{Y_0} ) ] \in\mathbf{K}_0(Y_0).
$$
This class depends only on the class  $[\mathcal{F}]$, not on the sheaf $\mathcal{F}$ itself, and the construction 
$[\mathcal{F}] \mapsto [\mathcal{F}\otimes^L_{\co_Y}\co_{Y_0} ]$ defines a homomorphism $\mathbf{K}_0(Y)\map{}\mathbf{K}_0(Y_0)$, 
as well as homomorphisms
$$
\mathbf{K}_0^\vertical(Y)\map{}\mathbf{K}_0^\vertical(Y_0)
\qquad
\mathbf{K}_0(Y^\good) \map{} \mathbf{K}_0^\vertical(Y_0).
$$
From the decomposition (\ref{moduli decomp}) we deduce an isomorphism 
$$
\mathbf{K}_0^\vertical(Y_0) \iso \bigoplus_{T\in\Sigma(\alpha)} \mathbf{K}_0^\vertical(Z(T)).
$$
Given a $T\in\Sigma(\alpha)$ and a coherent $\co_{Z(T)}$-module $\mathcal{F}_0$ 
that is annihilated by a power of $p$,  define
$$
\chi_T(\mathcal{F}_0) = \sum_{\ell \ge 0}(-1)^\ell \length_{\Z_p} R^\ell\mu_* \mathcal{F}_0 
$$
where $\mu:Z(T)\map{}\Spec(\Z_p)$ is the structure morphism.  If $\mathcal{F}_0$ is supported in dimension zero then it is 
easy to see that
$$
\chi_T(\mathcal{F}_0) = \sum_{y\in Z(T)(\F_p^\alg)} \length_{\co_{Y,y}} \mathcal{F}_{0,y}.
$$
As $\chi_T$  depends only on the class $[\mathcal{F}_0] \in\mathbf{K}_0^\vertical(Z(T))$, we may extend $\chi_T$ to a homomorphism 
$  \mathbf{K}_0^\vertical(Z(T)) \map{} \Z$ and define, for any class $[\mathcal{F}]$ in either $\mathbf{K}_0^\vertical(Y)$ or 
$\mathbf{K}_0(Y^\good)$, the intersection multiplicity of $\mathcal{F}$ and $\co_{M_0}$ at $Z(T)$ by
$$
I_{\co_{Z(T)}} (\mathcal{F} , \co_{M_0} )  =  \chi_T( \mathcal{F}\otimes^L_{\co_Y}\co_{Y_0}  )  .
$$

As in  \cite[\S 4.3]{howardA}, define formal $W$-schemes for $k\in\{1,2\}$
$$
\mathfrak{M}_0= \widehat{M}_{0/W} \times_{ \widehat{M}_{/W} }  \mathfrak{M} 
\qquad
\mathfrak{M}_0^k= \widehat{M}_{0/W} \times_{ \widehat{M}_{/W} }  \mathfrak{M}^k
$$
so that
$$
\widehat{Y}_{0/W} \iso \mathfrak{M}_0\times_{\mathfrak{M}} \widehat{Y}_{/W}  \iso 
\mathfrak{M}^1_0 \times_{\mathfrak{M}_0} \mathfrak{M}^2_0.
$$
These formal schemes  admit \v Cerednik-Drinfeld style uniformizations: if we set 
$\Omega_0(\tau)=\Omega(\tau)\cap \overline{G}_0(\A_f^p)$ then 
\begin{equation}\label{M_0 CD}
\mathfrak{M}_0 \iso \overline{G}_0(\Q)\backslash \bigsqcup_{  \substack{ \tau \in \overline{V}  \\ \overline{Q}(\tau)=\alpha } }
\left(
\mathfrak{X}_0 \times \Omega_0(\tau) \overline{U}^{\mathrm{max},p}/\overline{U}^p
\right)
\end{equation}
where the product $\Omega_0(\tau)\overline{U}^{\mathrm{max},p}$ is taken inside of $\overline{G}(\A_f^p)$.  Similarly 
\begin{equation}\label{M_0 CD II}
\mathfrak{M}^k_0 \iso \overline{G}_0(\Q)\backslash \bigsqcup_{    \substack{ \tau \in \overline{V}  \\ \overline{Q}(\tau)=\alpha  } }
\left(
\mathfrak{X}_0(\tau_k) \times \Omega_0(\tau) \overline{U}^{\mathrm{max},p}/\overline{U}^p
\right).
\end{equation}
Any coherent $\co_Y$-module $\mathcal{F}$ may be viewed as a $\mathrm{Frob}$-invariant coherent 
$\co_{\mathfrak{M}}$-module annihilated by the ideal sheaf of the closed formal subscheme 
$\widehat{Y}_{/W}\map{}\mathfrak{M}$.  For every $\ell\ge 0$ we may then form 
$\Tor_\ell^{\co_{\mathfrak{M}}}(\mathcal{F},\co_{\mathfrak{M}_0})$, a $\mathrm{Frob}$-invariant coherent 
$\co_{\mathfrak{M}_0}$-module annihilated by the ideal sheaf of the closed formal subscheme 
$\widehat{Y}_{0/W}\map{} \co_{\mathfrak{M}_0}$.  As before, using formal GAGA and faithfully flat descent we view 
$\Tor_\ell^{\co_{\mathfrak{M}}}(\mathcal{F},\co_{\mathfrak{M}_0})$ as a coherent $\co_Y$-module, and define
$$
[\mathcal{F} \otimes^L_{\co_{\mathfrak{M}}} \co_{\mathfrak{M}_0} ]
= \sum_{\ell\ge 0} (-1)^\ell [\Tor_\ell^{\co_{\mathfrak{M}}}(\mathcal{F} , \co_{\mathfrak{M}_0} ) ] \in \mathbf{K}_0(Y_0).
$$
It is easy to check that  $[\mathcal{F} \otimes^L_{\co_{\mathfrak{M}}} \co_{\mathfrak{M}_0} ] = [\mathcal{F}\otimes^L_{\co_Y}\co_{Y_0} ] $.

  For any $T\in\mathrm{Sym}_2(\Z)^\vee$ let $\overline{V}_0(T)$ be the set of pairs $(s_1,s_2)\in\overline{V}_0\times\overline{V}_0$ for which 
  (\ref{quadratic matrix}) holds, where $[s_i,s_j]= -\mathrm{Tr}(s_is_j)$ is the bilinear form on $\overline{V}_0$ satisfying 
  $[s,s]= 2 \overline{Q}_0(s)$.  Given a pair $(s_1,s_2)\in \overline{V}_0(T)$ set 
\begin{equation}\label{tau decomp}
\tau=s_1\varpi_1 + s_2\varpi_2\in \overline{V}_0\otimes_\Q F \iso \overline{V}
\end{equation}
and let $(\tau_1,\tau_2)$ be the image of $\tau$ under (\ref{p splitting}).  Assuming that $\det(T)\not=0$, 
Kudla-Rapoport \cite{kudla00} (and Kudla-Rapoport-Yang \cite{kudla04a} when $p=2$) compute the intersection 
multiplicity of the divisors $\mathfrak{h}_m(\tau_1)$ and $\mathfrak{h}_m(\tau_2)$ in the Drinfeld space $\mathfrak{h}_m$.  
This intersection multiplicity depends only the isomorphism class of the rank two quadratic space 
$$
\Z_p\tau_1 + \Z_p\tau_2 =\Z_ps_1 + \Z_p s_2 \subset \overline{V}_0\otimes_\Q\Q_p,
$$
which is determined by  $T$.   Let $e_p(T)$ be this intersection multiplicity, as in \cite[Theorem 6.1]{kudla00}.  
In the notation of \cite[Chapter 7.6]{KRY}, $\nu_p(T)=2e_p(T)$.

\begin{Prop} \label{Prop ram I}
Define a $\Z[1/p]$-lattice $\overline{L}_0\subset\overline{V}_0$ by
$$
\overline{L}_0 = \{ v\in \overline{V}_0  :  \widehat{\Lambda}_0^p \cdot \iota_0(v) \subset \widehat{\Lambda}_0^p \}
$$
and a discrete subgroup $\overline{\Gamma}_0\subset \overline{G}_0(\Q)$ by
$$
\overline{\Gamma}_0 = \{ \gamma\in \overline{G}_0(\Q)  :  \widehat{\Lambda}_0^p 
\cdot \iota_0(g) \subset \widehat{\Lambda}_0^p \} .
$$
Then for every nonsingular $T\in\Sigma(\alpha)$ 
$$
 I_{\co_{Z(T)}} ( [\co_Y]^\good , \co_{M_0} )  +   I_{\co_{Z(T)}} ( [\mathfrak{O}_Y]^\vertical , \co_{M_0} )  
 =  |H|  \cdot e_p(T)  \cdot   | \overline{\Gamma}_0 \backslash  \overline{L}_0(T)  |
$$
where $L_0(T) = \{ (s_1,s_2) \in V_0(T) : s_1,s_2\in L_0 \}$ and $\overline{\Gamma}_0$ acts on $L_0(T)$ by conjugation.
\end{Prop}

\begin{proof}
As $T$ is nonsingular, the scheme $Z(T)$ is supported in characteristic $p$ by \cite[Theorem 3.6.1]{KRY}.  Thus
$$
\mathbf{K}_0^\vertical(Z(T)) = \mathbf{K}_0(Z(T)).
$$
and $I_{\co_{Z(T)}}(\mathcal{F},\co_{M_0})$ is defined for every class $[\mathcal{F}]\in\mathbf{K}_0(Y)$.  
By the decomposition (\ref{scheme decomp}) each $D\in\Pi^\bad(Y)$ is  contained in $Z(T')$ for some $T'\in\Sigma(\alpha)$. 
 In particular $Z(T')$ has nonempty generic fiber, and hence $T'\not=T$.  Thus $D\cap Z(T)=\emptyset$  and the coherent 
 $\co_{Y_0}$-module $\Tor_\ell^{\co_Y}(\co_D,\co_{Y_0})$ has trivial restriction to $Z(T)$.  Thus
\begin{equation}\label{no improper}
I_{\co_{Z(T)}}( [\co_{Y}]^\bad , \co_{M_0}) = \sum_{D\in\Pi^\bad(Y)} \mathrm{mult}_D(\co_Y)\cdot
  \chi_T( \Tor_\ell^{\co_Y}(\co_D,\co_{Y_0}) )  =0.
\end{equation}
If $[\mathcal{F}]$ lies in the image of $\mathbf{K}_0^J(Y)\map{}\mathbf{K}_0(Y)$ for $J\map{}Y$ a closed 
subscheme of dimension zero, then $I_{\co_{Z(T)}}([\mathcal{F}] , \co_{M_0}) =0$ by 
\cite[Lemma 4.2.2]{howardA}; hence  combining (\ref{no improper}) with 
Proposition \ref{Prop:bad hor decomp} and (\ref{ramified h-v}) yields 
\begin{eqnarray*}\lefteqn{
I_{\co_{Z(T)}} ( [\co_Y]^\good , \co_{M_0} )  +   I_{\co_{Z(T)}} ( [\mathfrak{O}_Y]^\vertical , \co_{M_0} )  } \\
& = &
 I_{\co_{Z(T)}} ( [\mathfrak{O}_Y]^\horizontal , \co_{M_0} ) +  I_{\co_{Z(T)}} ( [\mathfrak{O}_Y]^\vertical , \co_{M_0} )  \\
 & = & 
I_{\co_{Z(T)}} ( \co_{\mathfrak{M}^1} \otimes^L_{\co_{\mathfrak{M}}} \co_{\mathfrak{M}^2} , \co_{M_0} )  \\
& = & 
\chi_T \big(( \co_{\mathfrak{M}^1} \otimes^L_{\co_{\mathfrak{M}}} \co_{\mathfrak{M}^2} ) \otimes^L_{ \co_{\mathfrak{M}}  } \co_{\mathfrak{M}_0} ) .
\end{eqnarray*}
The equality
$$
\chi_T \big(( \co_{\mathfrak{M}^1} \otimes^L_{\co_{\mathfrak{M}}} \co_{\mathfrak{M}^2} ) \otimes^L_{ \co_{\mathfrak{M}}  } \co_{\mathfrak{M}_0} ) 
 =  |H|  \cdot e_p(T)  \cdot   | \overline{\Gamma}_0 \backslash  \overline{L}_0(T)  |
$$
is now proved exactly as in  \cite[\S 4.3]{howardA}; see especially Lemma 4.3.2 and Proposition 4.3.4 of [\emph{loc.~ cit.}].
\end{proof}

Suppose $T\in \Sigma(\alpha)$ is singular and denote by $t_1$ and $t_2$ the diagonal entries of $T$.  
Let $n_1$, $n_2$, and $t$ be as in (\ref{little t}) so that $\alpha = (n_1\varpi_1 + n_2 \varpi_2)^2\cdot  t
$ and $K=\Q(\sqrt{-t})$ is a quadratic imaginary subfield of $F(\sqrt{-\alpha})$. Abbreviate 
$d_K= \mathrm{disc}(K/\Q)$ and define $n\in\Z^+$ by   $4t=-n^2d_K$.    As in (\ref{degenerate cycle}) 
there is an isomorphism of stacks $\mathcal{Z}(t)\iso \mathcal{Z}(T)$, which takes the triple 
$(\mathbf{A}_0, \lambda_0 , s_0)$ to the quadruple  $(\mathbf{A}_0, \lambda_0 , n_1s_0, n_2s_0)$.

\begin{Prop} \label{Prop ram II}
For every singular $T\in\Sigma(\alpha)$
\begin{equation}\label{ram intersection I}
 I_{\co_{Z(T)}} ( [\co_Y]^\good , \co_{M_0} )   =  \frac{1}{2} \cdot |Z(T)(\Q_p^\alg)| \cdot \ord_p( d_K ).
\end{equation}
\end{Prop}

\begin{proof} 
  For each $D\in \Pi^\good(Y)$ the scheme $D\times_Y Y_0$ has dimension zero.   
  It follows that the coherent  $\co_{Y_0}$-module  $\Tor_\ell^{\co_Y}(\co_D,\co_{Y_0})$ is 
  supported in dimension zero, and the left hand side of (\ref{ram intersection I}) is equal to 
\begin{equation}\label{ram intersection Ia}
 \sum_{y\in Z(T)(\F_p^\alg)}  \sum_{D\in\Pi^\good(Y)} \mathrm{mult}_D(\co_Y) 
 \cdot \sum_{\ell \ge 0} \length_{\co_{Y,y}}      \Tor^{\co_{Y,y}}_\ell( \co_{D,y}  ,  \co_{Y_0,y} )   .
 \end{equation}

 Given  a $y\in Z(T)(\F_p^\alg)$ let $R_M$ be the completion of the strictly Henselian  local ring of $M$ at $y$, 
 and define $R_Y$, $R_{M_0}$, and $R_{Y_0}$ similarly.   Let  $\Pi^\good(R_Y)$ be the set of minimal primes 
 $ R_Y$ of residue characteristic $0$ that do not come from $R_{Y_0}$ (more precisely: do not contain the 
 kernel of the surjection $R_Y\map{}R_{Y_0}$).    As in (\ref{be wise}) the left hand side of (\ref{ram intersection Ia}) is equal to 
\begin{equation}
 \label{formal horizontal}
\sum_{y\in Z(T)(\F_p^\alg)}   \sum_{  \mathfrak{p}\in\Pi^\good(R_Y) } \mathrm{mult}(\mathfrak{p}) \cdot
 \length_{R_Y} (R_Y/\mathfrak{p} \otimes_{R_M} R_{M_0} ).
\end{equation}
 From the uniformizations (\ref{CD}) and (\ref{super CD}) and similar uniformizations (as in \cite[\S 4]{howardA}) of 
 $\widehat{M}_{0/W}$ and $\widehat{Y}_{0/W}$,  we deduce that there are $m\in\Z$, 
 $x\in \mathfrak{h}_m(\F_p^\alg)$, and $s_0\in \overline{V}_0$ with $\overline{Q}_0(s_0)=t$ for which 
\begin{eqnarray*}
R_{M_0} & \iso & \widehat{\co}_{\mathfrak{h}_m,x}  \\ 
R_{M} & \iso & \widehat{\co}_{\mathfrak{h}_m,x} \widehat{\otimes}_W  \widehat{\co}_{\mathfrak{h}_m,x}   \\ 
R_{Y_0} & \iso & \widehat{\co}_{\mathfrak{h}_m(s_0),x}  \\
R_Y & \iso &   \widehat{\co}_{\mathfrak{h}_m(\tau_1) ,x} \widehat{\otimes}_W  \widehat{\co}_{\mathfrak{h}_m(\tau_2),x} 
\end{eqnarray*}
where $\tau= (n_1\varpi_1 + n_2\varpi_2)s_0$ and, as always, $(\tau_1,\tau_2)$ is the image of $\tau$ under 
(\ref{p splitting}).    Letting $(n_1,n_2)$ be the image of $(n_1\varpi_1 + n_2\varpi_2)$ under 
$\co_F\otimes_\Z\Z_p\iso \Z_p\times\Z_p$, we see that $\tau_i=n_i s_0$, and from the fact that $\gcd(n_1,n_2)=1$ 
we deduce that at least one of $n_1, n_2$ lies in $\Z_p^\times$. Hence
there are natural surjections
$$
\widehat{\co}_{\mathfrak{h}_m(\tau_1),x} \map{} \widehat{\co}_{\mathfrak{h}_m(s_0),x} 
\qquad 
\widehat{\co}_{\mathfrak{h}_m(\tau_2),x} \map{} \widehat{\co}_{\mathfrak{h}_m(s_0),x} 
$$
at least one of which is an isomorphism.  The completed local rings $ \widehat{\co}_{\mathfrak{h}_m,x}$ and 
$\widehat{\co}_{\mathfrak{h}_m(s_0),x}$ are described in detail in \cite{kudla00} (at least for $p\not=2$; for $p=2$ the calculations are in  
\cite{kudla04a}).  In the notation of \cite{kudla00}, the point $x$ may be either \emph{ordinary} or \emph{superspecial}.     From 
Propositions 3.2 and 3.3 of  \cite{kudla00} (and the appendix to  \cite[\S 11]{kudla04a} for  $p=2$), we  see that  there are three mutually 
exclusive possibilities:
\begin{enumerate}
\item
$\widehat{\co}_{\mathfrak{h}_m(s_0),x}$ is $W$-torsion;
\item
$x$ is ordinary, $p$ is inert in $K$, and the quotient of the $W$-algebra $\widehat{\co}_{\mathfrak{h}_m(s_0),x}$ by its ideal of 
$W$-torsion is  isomorphic to  $W$;
\item
$x$ is supersingular, $p$ is ramified in $K$, and  the quotient of the $W$-algebra $ \widehat{\co}_{\mathfrak{h}_m(s_0),x}$ by its ideal 
of $W$-torsion is  isomorphic to $\mathcal{W}$, where $\mathcal{W}$ is the ring of integers in $K\otimes_\Q \mathrm{Frac}(W)$.
\end{enumerate}
The same statements hold verbatim with $s_0$ replaced by $\tau_1$ or by $\tau_2$.

Suppose first that $p$ is unramified in $K$ and choose a $y\in Z(T)(\F_p^\alg)$.  From the above it follows that either $R_{Y_0}$ and 
$R_Y$ are both  $W$-torsion, or the quotients of $R_{Y_0}$ and  $R_Y$ by their  ideals of $W$-torsion are both isomorphic to $W$.   In 
the latter case each of $R_{Y_0}$ and $R_Y$ has a unique prime ideal of characteristic $0$.   In either case  every prime ideal of 
$R_Y$ of residue characteristic $0$ comes from $R_{Y_0}$, and so $\Pi^\good(R_Y)=\emptyset$.  We deduce that if $p$ is unramified 
in $K$ then   $\Pi^\good(R_Y)=\emptyset$  for every $y\in Z(T)(\F_p^\alg)$, and hence the left hand side of (\ref{formal horizontal}) is 
$0$.  From this we see that both sides of (\ref{ram intersection I}) are zero, and we are done.

Now suppose that $p$ is ramified in $K$ and again choose a $y\in Z(T)(\F_p^\alg)$.  Then either $R_{Y_0}$ and $R_Y$ are both 
$W$-torsion, or the quotient of $R_{Y_0}$ by its ideal of $W$-torsion is isomorphic to $\mathcal{W}$ and the quotient of  $R_Y$ by its 
ideal of $W$-torsion is isomorphic to $\mathcal{W}\otimes_W\mathcal{W}$.  Assume we are in the latter case.  Then $R_Y$  has  
exactly two  prime ideals of residue characteristic $0$,  call them $\mathfrak{p}$ and $\mathfrak{q}$.   The prime $\mathfrak{p}$ is the 
kernel of the surjection
$$
R_Y\map{} \mathcal{W} \otimes_W \mathcal{W} \map{a\otimes b\mapsto ab} \mathcal{W} 
$$
 while the prime   $\mathfrak{q}$ is the kernel of the surjection
$$
R_Y\map{} \mathcal{W}\otimes_W \mathcal{W}\map{a\otimes b\mapsto a\overline{b}} \mathcal{W}
$$
in which $b\mapsto \overline{b}$ is the nontrivial $W$-algebra automorphism of $\mathcal{W}$.    Note that the quotient map
$R_Y\map{}R_Y/\mathfrak{p}$ factors through the surjection $R_Y\map{}R_{Y_0}$, but that  $R_Y\map{}R_Y/\mathfrak{q}$ 
does not.  In other words $\mathfrak{p}$ comes from $R_{Y_0}$ while $\mathfrak{q}$ does not, and hence 
$\Pi^\good(R_Y)=\{\mathfrak{q}\}$.   Furthermore
$$
R_Y/\mathfrak{q} \otimes_{R_M} R_{M_0}  \iso R_Y/\mathfrak{q} \otimes_{ (R_{M_0} \widehat{\otimes} R_{M_0} ) } R_{M_0}
 \iso \mathcal{W} \otimes_{ (\mathcal{W}\otimes_W\mathcal{W}) } \mathcal{W}
$$
where in the final tensor product $\mathcal{W}$ is regarded as a $\mathcal{W}\otimes_W\mathcal{W}$ module in two distinct ways: on 
the left through $a\otimes b\mapsto a\overline{b}$ and on the right through $a\otimes b\mapsto ab$.  From standard properties of the 
discriminant (for example \cite[p.~64]{serre79}), we deduce
\begin{eqnarray*}
\length_{R_{Y_0}}(R_Y/\mathfrak{q} \otimes_{R_M} R_{M_0} ) &=&
 \length_{\mathcal{W}}( \mathcal{W} \otimes_{ (\mathcal{W}\otimes_W\mathcal{W}) } \mathcal{W} ) \\
&=& v_W(\mathrm{disc}(\mathcal{W}/W)) \\
&=& \ord_p(d_K)
\end{eqnarray*}
where $v_W$ is the normalized valuation on $W$.  It is easy to see that the localization of $R_Y$ at 
$\mathfrak{q}$ is isomorphic to the fraction field of $\mathcal{W}$, and hence $\mathrm{mult}(\mathfrak{q})=1$.  
Thus, in the case of $p$ ramified in $K$, for every $y\in Z(T)(\F_p^\alg)$ either $R_{Y_0}$ and $R_Y$ are 
both  $W$-torsion or 
$$
\sum_{\mathfrak{p} \in \Pi^\good (R_Y) } \mathrm{mult}(\mathfrak{p}) \cdot 
\length_{R_Y}( R_Y/\mathfrak{p} \otimes_{R_M} R_{M_0} )  = \ord_p(d_K).
$$

Still assuming that $p$ is ramified in $K$, we must count the number of $y\in Z(T)(\F_p^\alg)$ for which $R_{Y_0}$ contains a 
prime ideal of residue characteristic $0$.   By the discussion above, when such a prime ideal exists it is unique and has 
residue field a degree two extension of the fraction field of $W$.  Thus each such $y$ has two distinct lifts to $Y_0(\Q_p^\alg)$, 
each of which must be contained in $Z(T)(\Q_p^\alg)$ by the decomposition (\ref{scheme decomp}).   The number of $y$  for 
which $R_{Y_0}$ contains a prime ideal of residue characteristic $0$  is therefore $\frac{1}{2} |Z(T)(\Q_p^\alg)|$.  It follows that  
(\ref{formal horizontal}) is equal to $\frac{1}{2} |Z(T)(\Q_p^\alg)| \ord_p(d_K)$, and (\ref{ram intersection I}) follows.
\end{proof}

Given a coherent $\co_{\mathfrak{M}_0^k}$-module $\mathfrak{F}_0^k$ for $k \in\{1,2\}$  the sheaf 
$\Tor_\ell^{\co_{\mathfrak{M}_0}}(\mathfrak{F}_0^1,\mathfrak{F}_0^2)$ is annihilated by the ideal sheaf of the closed formal 
subscheme $\widehat{Y}_{0/W}\map{}\mathfrak{M}_0$, and  by formal GAGA may be viewed as a coherent 
$\co_{Y_{0/W}}$-module.  If the sheaves $\mathfrak{F}_0^1$ and $\mathfrak{F}_0^2$ are each invariant under 
$\mathrm{Frob}$ then exactly as in (\ref{derived tensor}) we may form
\begin{equation}\label{derived tensor II}
[\mathfrak{F}^1_0\otimes^L_{\co_{\mathfrak{M}_0}} \mathfrak{F}_0^2 ] 
= \sum_{\ell\ge 0} (-1)^\ell [\Tor^{\co_{\mathfrak{M}_0}}_\ell ( \mathfrak{F}_0^1 , \mathfrak{F}_0^2 ) ] \in \mathbf{K}_0(Y_0).
\end{equation}
If either of $\mathfrak{F}_0^1$ or $\mathfrak{F}_0^2$ is locally $W$-torsion then (\ref{derived tensor II}) also defines a class in 
$\mathbf{K}_0^\vertical(Y_0)$.

\begin{Lem}\label{Lem:CD pullback}
 Let $\mathfrak{B}_0^k$ be the ideal sheaf of locally $W$-torsion sections of $\co_{\mathfrak{M}_0^k}$ 
 and define $\mathfrak{A}_0^k$ by the exactness of 
$$
0\map{} \mathfrak{B}_0^k \map{}\co_{\mathfrak{M}_0^k} \map{} \mathfrak{A}_0^k \map{} 0.
$$
Then for any $T\in\Sigma(\alpha)$
\begin{eqnarray*}\lefteqn{
  I_{\co_{Z(T)}} ( [\mathfrak{O}_Y]^\vertical , \co_{M_0} )  } \\
& = &    \chi_T(\mathfrak{B}_0^1\otimes^L_{ \co_{\mathfrak{M}_0} }  \mathfrak{B}_0^2 )
  +  \chi_T(\mathfrak{B}_0^1\otimes^L_{ \co_{\mathfrak{M}_0} }  \mathfrak{A}_0^2 )
  +  \chi_T(\mathfrak{A}_0^1\otimes^L_{ \co_{\mathfrak{M}_0} }  \mathfrak{B}_0^2 ).
 \end{eqnarray*}
\end{Lem}

\begin{proof}
From the definitions we have
\begin{eqnarray*}\lefteqn{
 I_{\co_{Z(T)}} ( [\mathfrak{O}_Y]^\vertical , \co_{M_0} )    =
 \chi_T\big(  (\mathfrak{B}^1\otimes^L_{\co_{\mathfrak{M}}}
  \mathfrak{B}^2)\otimes^L_{ \co_{\mathfrak{M}} } \co_{\mathfrak{M}_0}  \big) }  \\
 & & +\chi_T\big(  (\mathfrak{B}^1\otimes^L_{\co_{\mathfrak{M}}} \mathfrak{A}^2)
 \otimes^L_{ \co_{\mathfrak{M}} } \co_{\mathfrak{M}_0}  \big)   
  + \chi_T\big(  (\mathfrak{A}^1\otimes^L_{\co_{\mathfrak{M}}} \mathfrak{B}^2)\otimes^L_{ \co_{\mathfrak{M}} } 
  \co_{\mathfrak{M}_0}  \big) .
\end{eqnarray*}
The same  argument used in the proof of  \cite[Lemma 4.3.2]{howardA} shows that
$$
\Tor_i^{\co_{\mathfrak{M}}} (  \mathfrak{B}^k, \co_{\mathfrak{M}_0} )
  \iso  \left\{ \begin{array}{ll}  \mathfrak{B}^k_0  & \mathrm{if\ }i=0 \\ 0 & \mathrm{if\ } i>0  \end{array}\right.
$$
and hence
\begin{eqnarray*}
\chi_T\big(  (\mathfrak{B}^1\otimes^L_{\co_{\mathfrak{M}}} \mathfrak{B}^2)\otimes^L_{ \co_{\mathfrak{M}} } 
\co_{\mathfrak{M}_0}  \big) 
&=& 
\chi_T\big(  (\mathfrak{B}^1\otimes^L_{\co_{\mathfrak{M}}} \co_{\mathfrak{M}_0}) \otimes^L_{ \co_{\mathfrak{M}_0} }  
(\mathfrak{B}^2 \otimes^L_{\co_{\mathfrak{M}}} \co_{\mathfrak{M}_0})    \big)   \\
&=&
\chi_T (\mathfrak{B}^1_0 \otimes^L_{ \co_{\mathfrak{M}_0} }  \mathfrak{B}_0^2 ).
\end{eqnarray*}
Similarly
$$
\Tor_i^{\co_{\mathfrak{M}}} (  \mathfrak{A}^k, \co_{\mathfrak{M}_0} )  \iso 
 \left\{ \begin{array}{ll}  \mathfrak{A}^k_0  & \mathrm{if\ }i=0 \\ 0 & \mathrm{if\ } i>0 \end{array}\right.
$$
implies that 
$$
\chi_T\big(  (\mathfrak{B}^1\otimes^L_{\co_{\mathfrak{M}}} \mathfrak{A}^2)\otimes^L_{ \co_{\mathfrak{M}} }
 \co_{\mathfrak{M}_0}  \big)  =  \chi_T (\mathfrak{B}^1_0 \otimes^L_{ \co_{\mathfrak{M}_0} }  \mathfrak{A}_0^2 )  
$$
and 
$$
\chi_T\big(  (\mathfrak{A}^1\otimes^L_{\co_{\mathfrak{M}}} \mathfrak{B}^2)\otimes^L_{ \co_{\mathfrak{M}} } 
\co_{\mathfrak{M}_0}  \big)  =   \chi_T (\mathfrak{A}^1_0 \otimes^L_{ \co_{\mathfrak{M}_0} }  \mathfrak{B}_0^2 ) 
$$
proving the claim.
\end{proof}

The equality of the following proposition is derived from calculations of Kudla-Rapoport-Yang \cite{kudla04a,KRY}.

\begin{Prop}  \label{Prop ram III}
If  $T\in\Sigma(\alpha)$ is singular  then 
$$
\frac{1}{|H|}   I_{\co_{Z(T)}} ( [\mathfrak{O}_Y]^\vertical , \co_{M_0} )  + 
     \frac{ h_{\widehat{\omega}_0}(\mathcal{Z}^\vertical(t)_p  )  }{\log(p)}  
    =      \frac{ 1 }{2}  \deg_\Q(\mathcal{Z}(t))  \cdot  \ord_p\left(\frac{ 4 \alpha\alpha^\sigma}{td_K}\right)  .
$$
\end{Prop}

\begin{proof}
Using the uniformizations (\ref{M_0 CD}) and (\ref{M_0 CD II}) and the isomorphism
$$
\widehat{Y}_{0/W} \iso \mathfrak{M}_0^1 \times_{\mathfrak{M}_0} \mathfrak{M}_0^2
$$
we find
$$
\widehat{Y}_{0/W} \iso  \overline{G}_0(\Q) \backslash
 \bigsqcup_{  \substack{ \tau\in \overline{V} \\ \overline{Q}(\tau) = \alpha   } }
\left(
\mathfrak{X}_0(\tau) \times \Omega_0(\tau) \overline{U}^{\mathrm{max},p} / \overline{U}^p
\right)
$$
where  $\mathfrak{X}_0(\tau) = \mathfrak{X}_0(\tau_1) \times_{\mathfrak{X}_0} \mathfrak{X}_0(\tau_2)$ is the locus 
in $\mathfrak{X}$ where both of the quasi-endomorphisms $\tau_1$ and $\tau_2$ are integral.  Using the notation of  
(\ref{tau decomp}) there is a bijection 
$$
\{\tau\in\overline{V} : \overline{Q}(\tau) = \alpha \} \map{} \bigsqcup_{T\in\Sigma(\alpha)} \overline{V}_0(T)
$$
given by $\tau\mapsto (s_1,s_2)$.  Using the decomposition (\ref{scheme decomp}) to view $Z(T)$ as a closed subscheme of 
$Y_0$ for each $T\in\Sigma(\alpha)$, we identify $\widehat{Z}(T)_{/W}$ with the open and closed formal subscheme  of 
$\widehat{Y}_{0/W}$
$$
\widehat{Z}(T)_{/W} \iso  \overline{G}_0(\Q) \backslash
 \bigsqcup_{  (s_1,s_2) \in  \overline{V}_0(T) }
\left(
\mathfrak{X}_0(\tau) \times \Omega_0(\tau) \overline{U}^{\mathrm{max},p} / \overline{U}^p
\right).
$$
Now fix a singular $T\in\Sigma(\alpha)$.  Using the isomorphism $Z(T)\iso Z(t)$, we find
$$
\widehat{Z}(T)_{/W} \iso  \overline{G}_0(\Q) \backslash
 \bigsqcup_{  \substack{  s_0 \in  \overline{V}_0 \\ \overline{Q}_0(s_0) = t  } }
\left(
\mathfrak{X}_0(\tau) \times \Omega_0(\tau) \overline{U}^{\mathrm{max},p} / \overline{U}^p
\right)
$$
in which $\tau = (n_1\varpi_1 + n_2\varpi_2) s_0$.   By noting that  the factor $n_1\varpi_1 + n_2\varpi_2\in \co_F$ is not 
divisible by any rational prime (as $\gcd(n_1,n_2)=1$), we deduce that 
\begin{eqnarray*}
\Omega_0(\tau) 
&=&  \{ g\in\overline{G}_0(\A_f^p)  :   \widehat{\Lambda}^p \cdot \iota(g^{-1}\tau g) \in \widehat{\Lambda}^p \} \\
&=&  \{ g\in\overline{G}_0(\A_f^p)  : \iota(g^{-1}\tau g) \in \widehat{\co}_{B}^p \} \\
&=& 
\{ g\in\overline{G}_0(\A_f^p)  : \iota(g^{-1}s_0 g) \in \widehat{\co}_{B_0}^p \}.
\end{eqnarray*}
In particular $\Omega_0(\tau)=\Omega_0(s_0)$.  We now argue as in \cite[\S 11]{kudla04a}.   By the Noether-Skolem theorem 
the action of $\overline{G}_0(\Q)$ on $\{ s_0\in \overline{V}_0 : \overline{Q}_0(s_0) = t \}$ is transitive.  Fixing one 
$s_0\in\overline{V}_0$ with $\overline{Q}_0(s_0)=t$,   embed $K\map{}\overline{B}_0$ via $\sqrt{-t}\mapsto s_0$.  This 
induces a homomorphism $K^\times\map{}\overline{G}_0(\Q)$ and, using the fact that $K^\times$ is the stabilizer of $s_0$ in 
$\overline{G}_0(\Q)$, we deduce
\begin{equation}\label{Z CD}
\widehat{Z}(T)_{/W} \iso  K^\times \backslash
\big(
\mathfrak{X}_0(\tau) \times \Omega_0(s_0) \overline{U}^{\mathrm{max},p} / \overline{U}^p
\big).
\end{equation}

Assume that $p$ is inert in $K$, set $$K^\flat=\{ x\in K^\times: \ord_p(\mathrm{N}_{K/\Q}(x)) = 0 \},$$ and write 
$$\mathfrak{h}_m(\tau) = \mathfrak{h}_m(\tau_1)\times_{\mathfrak{h}_m} \mathfrak{h}_m(\tau_2).$$   
As in \cite[(11.8)]{kudla04a}  rewrite (\ref{Z CD}) as
\begin{equation}\label{sub ZCD}
\widehat{Z}(T)_{/W} \iso \big(\mathfrak{h}_0(\tau)\sqcup \mathfrak{h}_1(\tau) \big) 
\times \big( K^\flat \backslash \Omega_0(s_0)\overline{U}^{\mathrm{max},p} / \overline{U}^p \big).
\end{equation}
For $k\in\{1,2\}$ define the coherent $\co_{\mathfrak{h}_m}$-module $\co_{\mathfrak{h}_m(\tau_k)}^\vertical$ to be the ideal 
sheaf of locally $W$-torsion sections of the sheaf $\co_{\mathfrak{h}_m(\tau_k)}$, and define  
$\co_{\mathfrak{h}_m(\tau_k)}^\horizontal$ by the exactness of 
$$
0\map{} \co_{\mathfrak{h}_m(\tau_k)}^\vertical  \map{}  \co_{\mathfrak{h}_m(\tau_k)} \map{}  
\co_{\mathfrak{h}_m(\tau_k)}^\horizontal \map{}0.
$$
Under the  uniformization (\ref{sub ZCD}) we have isomorphisms of  coherent $\co_{\widehat{Z}(T)_{/W}}$-modules
 \begin{eqnarray*}
 \Tor_\ell^{\co_{\mathfrak{M}_0}}  (\mathfrak{B}_0^1 , \mathfrak{B}_0^2)  &\iso &
 \Tor_\ell^{ \co_{\mathfrak{h}_m}}  (  \co_{\mathfrak{h}_m(\tau_1)}^\vertical ,  \co_{\mathfrak{h}_m(\tau_2)}^\vertical ) \\
 \Tor_\ell^{\co_{\mathfrak{M}_0}}  (\mathfrak{B}_0^1 , \mathfrak{A}_0^2)  &\iso &
 \Tor_\ell^{ \co_{\mathfrak{h}_m}}  (  \co_{\mathfrak{h}_m(\tau_1)}^\vertical ,  \co_{\mathfrak{h}_m(\tau_2)}^\horizontal ) \\
 \Tor_\ell^{\co_{\mathfrak{M}_0}}  (\mathfrak{A}_0^1 , \mathfrak{B}_0^2)  &\iso &
 \Tor_\ell^{ \co_{\mathfrak{h}_m}}  (  \co_{\mathfrak{h}_m(\tau_1)}^\horizontal ,  \co_{\mathfrak{h}_m(\tau_2)}^\vertical ) . 
\end{eqnarray*}
Letting $\mu:\mathfrak{h}_m\map{}\Spf(W)$ denote the structure map and 
$$
\chi(\mathfrak{F}) = \sum_{k\ge 0} \length_W R^k\mu_*\mathfrak{F}
$$
the Euler characteristic of a coherent, properly supported,  locally $W$-torsion $\co_{\mathfrak{h}(m)}$-module 
$\mathfrak{F}$, Kudla-Rapoport-Yang have proved (see the proof of \cite[Proposition 7.6.4]{KRY})
 \begin{eqnarray*}
\sum_{\ell \ge 0} \chi \big(
 \Tor_\ell^{ \co_{\mathfrak{h}_m}}  (  \co_{\mathfrak{h}_m(\tau_1)}^\vertical ,  \co_{\mathfrak{h}_m(\tau_2)}^\vertical )  \big) 
  &=& -(p+1)\frac{p^{\ord_p(n)} -1}{p-1} \\
\sum_{\ell \ge 0} \chi \big(
 \Tor_\ell^{ \co_{\mathfrak{h}_m}}  (  \co_{\mathfrak{h}_m(\tau_1)}^\vertical ,  \co_{\mathfrak{h}_m(\tau_2)}^\horizontal )  \big)  
 &=& \ord_p( 4 \overline{Q}_0(  \tau_1))   \\
\sum_{\ell \ge 0} \chi \big(
 \Tor_\ell^{ \co_{\mathfrak{h}_m}}  (  \co_{\mathfrak{h}_m(\tau_1)}^\horizontal ,  \co_{\mathfrak{h}_m(\tau_2)}^\vertical )  \big)  
 &=& \ord_p( 4 \overline{Q}_0(  \tau_2)) .
\end{eqnarray*}
 Combining this with Lemma \ref{Lem:CD pullback}  and using 
 $$
 \ord_p( 4 \overline{Q}_0(  \tau_1))  + \ord_p( 4 \overline{Q}_0(  \tau_2)) = \ord_p(16\alpha\alpha^\sigma) 
 $$
 shows that
\begin{eqnarray}\lefteqn{ \nonumber
I_{\co_{Z(T)}}( [\mathfrak{O}_Y]^\vertical , \co_{M_0} )  }  \\
& & = 2\cdot  \left(-(p+1)\frac{p^{\ord_p(n) -1}}{p-1} +
 \ord_p(16\alpha\alpha^\sigma)  \right) \cdot 
 | K^\flat \backslash \Omega_0(\tau)\overline{U}^{\mathrm{max},p} / \overline{U}^p | . \label{first inert final}
\end{eqnarray}
It is easy to see that 
$$
| K^\flat \backslash \Omega_0(\tau)\overline{U}^{\mathrm{max},p} / \overline{U}^p | =
 \frac{ |H| }{ [\overline{U}_0^{\mathrm{max},p}  :  \overline{U}_0^p ]  } 
  \cdot | K^\flat \backslash \Omega_0(s_0)/ \overline{U}_0^p |, 
$$
and the right hand side of this equality is computed in \cite[Lemma 11.4]{kudla04a}. 
 Combining that calculation with  \cite[Proposition 9.1]{kudla04a} gives
\begin{eqnarray}
2\cdot | K^\flat \backslash \Omega_0(\tau)\overline{U}^{\mathrm{max},p} / \overline{U}^p | 
& =&  |H| \cdot \delta(d_K , \mathrm{disc}(B_0) ) \cdot H_0(t; \mathrm{disc}(B_0))  \nonumber \\
&=&
\frac{|H|}{2} \cdot \mathrm{deg}_\Q(\mathcal{Z}(t))  \label{second inert final}
\end{eqnarray}
where the functions $\delta$ and $H_0$ appearing  are those of \cite[\S 8]{kudla04a}.  
Finally, \cite[Lemma 7.9.1]{KRY}  tells us that
\begin{equation}\label{third inert final}
  \frac{ h_{\widehat{\omega}_0}(\mathcal{Z}^\vertical(t)_p  )  }{\log(p)}  
  = -\deg_\Q(\mathcal{Z}(t)) \cdot \left( \ord_p(n) - \frac{ (p+1)(p^{\ord_p(n)} -1) }{ 2(p-1) } \right) .
\end{equation}
Combining (\ref{first inert final}), (\ref{second inert final}), and (\ref{third inert final}) with $4t=n^2d_K$ 
completes the proof in the case of $p$ inert in $K$.

If $p$ is ramified or split in $K$ the claim similarly follows from calculations of Kudla-Rapoport-Yang.
If $p$ is ramified in $K$ then, as in \cite[(11.8)]{kudla04a}, rewrite (\ref{Z CD}) as 
$$
\widehat{Z}(T)_{/W} \iso \mathfrak{h}_0(\tau) \times
 \big( K^\flat \backslash \Omega_0(\tau)\overline{U}^{\mathrm{max},p} / \overline{U}^p \big)
 $$
and  the proof proceeds in exactly the same way as the inert case, by combining the proof of \cite[Proposition 7.6.4]{KRY} 
 with \cite[Lemma 7.9.1]{KRY}.    If $p$ is split in $K$,  let $K^{\flat\flat}$ denote the subgroup of elements of $K^\times$ whose 
 image in $(K\otimes_\Q\Q_p)^\times$ lies in $(\co_K\otimes_\Z\Z_p)^\times$ and fix an $\epsilon\in K^\times$ whose image in 
 $K\otimes_\Q\Q_p\iso \Q_p\times\Q_p$ has valuation $(1,-1)$.  As in \cite[(11.19)]{kudla04a}, rewrite (\ref{Z CD}) as
$$
\widehat{Z}(T)_{/W} \iso \big( \epsilon^\Z\backslash \mathfrak{h}_0(\tau) \big) \times  
\big( K^{\flat\flat} \backslash \Omega_0(\tau)\overline{U}^{\mathrm{max},p} / \overline{U}^p \big).
$$
Once again, comparing the proof of  \cite[Proposition 7.6.4]{KRY}  with \cite[Lemma 7.9.1]{KRY} we find that
$$
\frac{1}{|H|}   I_{\co_{Z(T)}} ( [\mathfrak{O}_Y]^\vertical , \co_{M_0} )  + 
     \frac{ h_{\widehat{\omega}_0}(\mathcal{Z}^\vertical(t)_p  )  }{\log(p)}   =0
 $$
while \cite[Proposition 3.4.5]{KRY} tells us that $\mathcal{Z}(t)_{/\Q}=\emptyset$.
\end{proof}

We now construct some  cycles on $M$.  As in (\ref{good cycle})  define a horizontal cycle
$$
C_p^\good = \sum_{D\in\Pi^\good(Y)} \mathrm{mult}_D(\co_Y)\cdot \phi(D)
$$
of codimension two on $M$, and define $C_p^\bad$ in the same way.  These cycles are 
$H$-invariant and so arise as the pullbacks of horizontal cycles on $\mathcal{M}_{/\Z_p}$, 
which we denote by $\mathcal{C}_p^\good$ and $\mathcal{C}_p^\bad$.  Now consider the class 
$[\mathfrak{O}_Y]^\vertical\in\mathbf{K}^\vertical_0(Y)$ defined after (\ref{ramified h-v}).  
By \cite[Proposition 4.2.3]{howardA} this class lies in the kernel of
$$
\mathbf{K}^\vertical_0(Y) \map{} \mathbf{K}_0(Y) \map{}\mathbf{K}_0(\Spec(\co_{Y,\eta}))
$$
for every $\eta\in Y$ with $\mathrm{dim}\overline{\{\eta\}} >1$.  Using  the notation of \cite[\S 2.2]{howardA}, 
\cite[Lemma 4.2.4]{howardA} shows that  $R\phi_*[\mathfrak{O}_Y]^\vertical\in F^2\mathbf{K}_0^Y(M)$, 
while the Gillet-Soul\'e isomorphism \cite[(9)]{howardA} and the homomorphism \cite[(7)]{howardA}  provide us with maps
$$
F^2\mathbf{K}_0^Y(M) \map{}\mathrm{CH}^2_Y(M) \map{} \mathrm{CH}^2_\vertical(M).
$$
The Chow groups here, as throughout \cite{howardA}, are Chow groups with rational coefficients.  
The image of $R\phi_*[\mathfrak{O}_Y]^\vertical$ under this composition, denoted $C_p^\vertical$, is 
$H$-invariant and so arises from some 
$$
\mathcal{C}_p^\vertical\in \mathrm{CH}^2_\vertical(\mathcal{M}_{/\Z_p}).
$$
Thus we have constructed a codimension two vertical cycle on $\mathcal{M}_{/\Z_p}$ with rational coefficients, 
which is determined  up the the addition of rational multiples of principal Weil divisors on $\mathcal{M}_{/\F_p}$.

\begin{Prop}\label{Prop:main ramified}
We have
\begin{eqnarray*}\lefteqn{
 I_p(\mathcal{C}_p^\bullet  , \mathcal{M}_0) + 
 I_p( \mathcal{C}_p^\vertical  , \mathcal{M}_0)   +    \sum_{  \substack{  T\in \Sigma(\alpha)  \\  \det(T)=0}  }
   \frac{  h_{\widehat{\omega}_0}(\mathcal{Z}^\vertical(t)_p  ) }{\log(p)}   }    \\
&  & =
  \sum_{ \substack{T\in\Sigma(\alpha) \\ \det(T)\not=0 } }  
  e_p(T)\cdot |\overline{\Gamma}_0\backslash \overline{L}_0(T) |  
      + \frac{1}{2}\cdot \sum_{ \substack{T\in\Sigma(\alpha) \\ \det(T) =0 } }   
 \deg_\Q(\mathcal{Z}(t))  \cdot  \ord_p\left(\frac{ 4 \alpha\alpha^\sigma}{t}\right) .
 \end{eqnarray*}
\end{Prop}

\begin{proof}
Using the decomposition (\ref{moduli decomp}) we find
\begin{eqnarray*}\lefteqn{
I_p(\mathcal{C}_p^\bullet  , \mathcal{M}_0) +  I_p( \mathcal{C}_p^\vertical  , \mathcal{M}_0)  } \\
&=&
\frac{1}{ |H| } \sum_{T\in\Sigma(\alpha)} I_{\co_{Z(T)} } ( [\co_Y]^\good , \co_{M_0} ) + 
\frac{1}{ |H|  } \sum_{T\in\Sigma(\alpha)}  I_{\co_{Z(T)} } ( [\mathfrak{O}_Y]^\vertical , \co_{M_0} ).
\end{eqnarray*}
The claim now follows from Propositions \ref{Prop ram I}, \ref{Prop ram II}, and \ref{Prop ram III}.
\end{proof}


\section{Pullbacks of arithmetic cycles}
\label{s:pullbacks}


We are now ready to put everything together to prove the main result of the paper, Theorem \ref{Thm:main result} below.

Fix an $\alpha\in\co_F$ and a $v\in F\otimes_\Q\R$, both totally positive.  Let $\mathcal{C}^\horizontal$, 
a codimension two cycle on $\mathcal{M}$,  be the Zariski closure of the cycle $\mathcal{C}_\Q$ defined by 
 (\ref{s:general cycle}), and similarly let $\mathcal{C}^\good$ and $\mathcal{C}^\bad$ be the Zariski closures of 
 $\mathcal{C}_\Q^\good$ and $\mathcal{C}_\Q^\bad$.  
Recall that Proposition \ref{Prop:green construction} provides us with Green currents 
$\Xi^\good(\alpha,v)$ and $\Xi^\bad(\alpha,v)$ for $\mathcal{C}^\good$ and $\mathcal{C}^\bad$, 
and hence $$\Xi(\alpha,v)=\Xi^\good(\alpha,v)+\Xi^\bad(\alpha,v)$$ is a Green current for $\mathcal{C}^\horizontal$.
Denote by
$$
\widehat{\mathcal{Y}}^\horizontal (\alpha,v)\in \widehat{\mathrm{CH}}^2(\mathcal{M})
$$
the arithmetic cycle class represented by $(\mathcal{C}^\horizontal,\Xi(\alpha,v))$.   We then have a decomposition 
$$
\widehat{\mathcal{Y}}^\horizontal (\alpha,v) = \widehat{\mathcal{Y}}^\good (\alpha,v)  +  \widehat{\mathcal{Y}}^\bad (\alpha,v)
$$  
in which $\widehat{\mathcal{Y}}^\good(\alpha,v)$ is the arithmetic cycle class represented by the pair 
$(\mathcal{C}^\good,\Xi^\good(\alpha,v))$,  and similarly  for $\widehat{\mathcal{Y}}^\bad(\alpha,v)$.
For every prime $p$ we have constructed a vertical cycle  $\mathcal{C}_p^\vertical$ of codimension two on $\mathcal{M}$.   
If $p\nmid\mathrm{disc}(B_0)$ then $\mathcal{C}_p^\vertical$ was defined at the end of \S \ref{s:good reduction}, and is 
nontrivial only if  $\mathcal{Y}_0$ has an irreducible component supported in characteristic $p$.  If $p\mid \mathrm{disc}(B_0)$ 
then $\mathcal{C}_p^\vertical$ was constructed in  \S \ref{s:bad reduction}.  In this latter case $\mathcal{C}_p^\vertical$ has 
rational coefficients and  is only defined up to the addition of rational multiples of principal Weil divisors on 
$\mathcal{M}_{/\F_p}$.  In either case,  we endow the cycle $\mathcal{C}_p^\vertical$ with the trivial Green 
current to obtain a class
$$
\widehat{\mathcal{Y}}_p^\vertical (\alpha)\in \widehat{\mathrm{CH}}^2(\mathcal{M}).
$$
The arithmetic cycle class
 \begin{equation}\label{the cycle class}
\widehat{\mathcal{Y}}(\alpha,v) = \widehat{\mathcal{Y}}^\horizontal (\alpha,v)
+  \sum_{p\mathrm{\ prime}} \widehat{\mathcal{Y}}^\vertical_p (\alpha)
\end{equation}
agrees with that constructed in \cite[\S 5.1]{howardA}.

\begin{Prop}
\label{Prop:adjunction application}
 If we abbreviate
 $$
b(\alpha,v) = \log\left(  \frac{\alpha v_1+ \alpha^\sigma v_2}{4  v_1 v_2 \alpha\alpha^\sigma  d_F \mathrm{disc}(B_0)  } \right) 
- J(4\pi  \alpha v_1 + 4\pi \alpha^\sigma v_2) 
$$
(the function $J$ was defined in \S \ref{s:adjunction})  then
\begin{eqnarray*}
\widehat{\deg}_{\mathcal{M}_0}   \widehat{\mathcal{Y}}(\alpha,v)  &=&
  \frac{1}{2} \cdot b(\alpha,v) \cdot \deg_\Q (\mathcal{C}^\bad)  -h_{\widehat{\omega}_0} (\mathcal{C}^\bad)  \\
 & &   +  \sum_{   \substack{  \tau \in \Gamma_0 \backslash L^\nonsing  \\  Q(\tau) = \alpha  } }
  \frac{1}{2\cdot |\mathrm{Stab}_{\Gamma_0}(\tau) |}   \int_{  X_0 }   \xi_0(v_1^{1/2} \tau_ 1)* \xi_0( v_2^{1/2} \tau_2)   \\
& &   +  \sum_{p \mathrm{\ prime}}  \log(p) \big( I_p( \mathcal{C}^\good,\mathcal{M}_0) 
+ I_p(\mathcal{C}_p^\vertical, \mathcal{M}_0)\big).
\end{eqnarray*}
Here $\deg_\Q$ is defined by (\ref{generic degree}) for irreducible cycles of codimension two on 
$\mathcal{M}$ and extended linearly to all cycles of codimension two, $\widehat{\omega}_0$ is the 
metrized Hodge bundle of \S \ref{s:hodge},  $h_{\widehat{\omega}_0}$ is the Arakelov height 
of \S \ref{s:adjunction},  $\mathrm{Stab}_{\Gamma_0}(\tau)$ is the stabilizer of $\tau$ in $\Gamma_0$, 
and $L^\nonsing$ and $\xi_0$  are as defined  in \S \ref{s:generic fiber}.
\end{Prop}

 \begin{proof}
 From the definition of $\widehat{\mathcal{Y}}(\alpha,v)$, we have
 \begin{equation}\label{obvious decomp}
 \widehat{\deg}_{\mathcal{M}_0}   \widehat{\mathcal{Y}}(\alpha,v)   = 
 \widehat{\deg}_{\mathcal{M}_0}  \big[ \widehat{\mathcal{Y}}^\good (\alpha,v)  
 + \sum_p \widehat{\mathcal{Y}}^\vertical_p (\alpha) \big] +
 \widehat{\deg}_{\mathcal{M}_0}   \widehat{\mathcal{Y}}^\bad(\alpha,v)  .
 \end{equation}
 From \S \ref{s:moduli} we know that
\begin{eqnarray}\nonumber
 \widehat{\deg}_{\mathcal{M}_0}    \big[ \widehat{\mathcal{Y}}^\good (\alpha,v)  
 + \sum_p \widehat{\mathcal{Y}}^\vertical_p (\alpha) \big]   &=& 
 \sum_p \log(p)  \big[ I_p(\mathcal{C}^\good,\mathcal{M}_0) + I_p(\mathcal{C}_p^\vertical,\mathcal{M}_0)  \big]  \\
& & +  I_\infty( \Xi^\good(\alpha,v)  ,\mathcal{M}_0 ),  \label{good decomp}
 \end{eqnarray}
 and, as in \cite[Proposition 3.2.1]{howardA},
 \begin{eqnarray*}
  I_\infty( \Xi^\good(\alpha,v)  ,\mathcal{M}_0 ) 
  &=& \sum_{ \substack{ \tau\in\Gamma_0\backslash L^\good  \\  Q(\tau)=\alpha  } }
      \frac{1}{| \mathrm{Stab}_{\Gamma_0}(\tau) |}  \int_{X_0} \xi_0(v_1^{1/2}\tau_1) * \xi_0(v_2^{1/2}\tau_2)
 \end{eqnarray*}
This gives a formula for the first term on the right hand side of (\ref{obvious decomp}), 
and we use the adjunction formula of \S \ref{s:adjunction} to compute the second term.  
 Indeed, If we extend the construction $\mathcal{D}\mapsto \widehat{\mathcal{D}}(v)$ of (\ref{augment}) 
 linearly to all horizontal cycles $\mathcal{D}$ of codimension two on $\mathcal{M}$ then 
 $$
 \widehat{\mathcal{Y}}^\bad(\alpha,v) = \widehat{\mathcal{C}}^\bad(\alpha v).
 $$
  Theorem \ref{Thm:adjunction}  extends linearly to all such $\mathcal{D}$ (providing one counts points 
  $P\in \mathcal{D}(\C)$ with appropriate multiplicities) and yields 
  \begin{eqnarray*}
  \widehat{\deg}_{\mathcal{M}_0}   \widehat{\mathcal{Y}}^\bad(\alpha,v) 
  &=&    - h_{\widehat{\omega}_0}(\mathcal{C}^\bad) + \frac{1}{2} b(\alpha,v) \deg_\Q(\mathcal{C}^\bad)   \\
  &  &   +   \frac{1}{2} \sum_{P \in \mathcal{C}^\bad(\C)} e_P^{-1}
  \sum_{ \substack{  \gamma\in\Gamma_0\backslash \Gamma \\  \gamma\not\in \Gamma_0  } }
    \int_{X_0} \mathbf{g}_0(\gamma_1 x_0 , \alpha_1 v_1) * \mathbf{g}_0(\gamma_2 x_0 , \alpha_2 v_2)
  \end{eqnarray*}
where $x_0\in X_0$ lies above $P$ under the orbifold  uniformization 
$[\Gamma_0\backslash X_0] \iso \mathcal{M}_0(\C)$.  As in the proof of Proposition \ref{Prop:green construction}, 
the cycle $\mathcal{C}^\bad(\C)$ on $\mathcal{M}_0(\C)$ is identified with the formal sum
$$
\mathcal{C}^\bad(\C) = \sum_{ \substack{   \tau\in \Gamma\backslash L^\bad   \\   Q(\tau)=\alpha } }
(x^+(\tau)+x^-(\tau)).
$$
Choosing each  coset representative $\tau\in \Gamma\backslash L^\bad$ to lie in $L^\sing$, 
so that $x^\pm(\tau)\in X_0$, we see from Lemma \ref{Lem:autos} and (\ref{in diag}) that 
$$
\gamma\tau \in L^{\sing} \iff \gamma x^\pm(\tau)\in X_0 \iff \gamma\in \Gamma_0.
$$ 
This observation gives the second equality of 
\begin{eqnarray*}\lefteqn{
\sum_{P \in \mathcal{C}^\bad(\C)} e_P^{-1} 
\sum_{ \substack{  \gamma\in\Gamma_0\backslash \Gamma \\  \gamma\not\in \Gamma_0  } } 
 \int_{X_0} \mathbf{g}_0(\gamma_1 x_0 , \alpha_1 v_1) * \mathbf{g}_0(\gamma_2 x_0 , \alpha_2 v_2)  } \\
&=& 
 \sum_{ \substack{   \tau\in \Gamma\backslash L^\bad   \\   Q(\tau)=\alpha } }  
  \frac{ 1 }{ |\mathrm{Stab}_{\Gamma_0}(\tau) | }
  \sum_{ \substack{  \gamma\in\Gamma_0\backslash \Gamma \\  \gamma\not\in \Gamma_0  } }  
  \int_{X_0}   \xi_0(v_1^{1/2} \tau_1) * \xi_0(v_2^{1/2} \tau_2)  \\
 &=& 
 \sum_{ \substack{   \tau\in \Gamma_0\backslash (L^\bad \smallsetminus L^\sing)   \\   Q(\tau)=\alpha } } 
   \frac{ 1 }{ |\mathrm{Stab}_{\Gamma_0}(\tau) | } \int_{X_0}   \xi_0(v_1^{1/2} \tau_1) * \xi_0(v_2^{1/2} \tau_2), 
\end{eqnarray*}
and using $L^\nonsing=L^\good\sqcup(L^\bad\smallsetminus L^\sing)$ we obtain
 \begin{eqnarray*}\lefteqn{
{   I_\infty( \Xi^\good(\alpha,v)  ,\mathcal{M}_0 )  +    \widehat{\deg}_{\mathcal{M}_0}   \widehat{\mathcal{Y}}^\bad(\alpha,v)    }
 = 
{  - h_{\widehat{\omega}_0}(\mathcal{C}^\bad) + \frac{1}{2} b(\alpha,v) \deg_\Q(\mathcal{C}^\bad)  }  }  \\
  &  &   +   \sum_{ \substack{   \tau\in \Gamma_0\backslash L^\nonsing   \\   Q(\tau)=\alpha } }  
   \frac{ 1 }{2 |\mathrm{Stab}_{\Gamma_0}(\tau) | } \int_{X_0}   \xi_0(v_1^{1/2} \tau_1) * \xi_0(v_2^{1/2} \tau_2).
    \end{eqnarray*}
Combining this last equality  with (\ref{obvious decomp}) and (\ref{good decomp}) completes the proof.
  \end{proof}

For every symmetric positive definite matrix $\mathbf{v}\in M_2(\R)$, and every $T\in\mathrm{Sym}_2(\Z)^\vee$, 
Kudla-Rapoport-Yang \cite{KRY} have defined an arithmetic cycle class
\begin{equation}\label{KRY zero cycle}
\widehat{\mathcal{Z}}(T,\mathbf{v}) \in \widehat{\mathrm{CH}}^2_\R(\mathcal{M}_0)
\end{equation}
in the $\R$-arithmetic Chow group defined in \cite[\S 2.4]{KRY}.   Our main result, an arithmetic form of the
 decomposition (\ref{moduli decomp}), relates the arithmetic degree of  (\ref{the cycle class}) along $\mathcal{M}_0$, 
 in the sense of (\ref{arithmetic degree II}),   with the arithmetic degree of (\ref{KRY zero cycle}), in the sense of 
 \cite[(2.4.10)]{KRY}.   Recall that $(v_1,v_2)$ denotes the image of $v$ under 
 $F\otimes_\Q\R\iso \R\times\R$ and that $\{\varpi_1,\varpi_2\}$ is our fixed $\Z$-basis of $\co_F$.  Define
\begin{equation}\label{twisty}
\mathbf{v} = R \left(\begin{matrix}  v_1 \\ & v_2  \end{matrix}\right) {}^tR
\qquad
R=  \left(\begin{matrix}   \varpi_1 & \varpi_1^\sigma  \\ \varpi_2  &  \varpi_2^\sigma  \end{matrix}\right).
\end{equation}

\begin{Thm}\label{Thm:main result}
If  either
\begin{enumerate}
\item
 $F(\sqrt{-\alpha})/\Q$ is not biquadratic, or 
 \item
 $2$  splits in $F$ and $\gcd(\alpha\co_F,\mathfrak{D}_F) = \co_F$
 \end{enumerate}
    then
$$
\widehat{\deg}_{\mathcal{M}_0} \widehat{\mathcal{Y}}(\alpha,v)  = 
\sum_{ T\in \Sigma( \alpha) }\widehat{\deg}\, \widehat{\mathcal{Z}}(T,\mathbf{v}).
$$
\end{Thm}

\begin{proof}
By \cite[Theorem 3.6.1]{KRY},  $\mathcal{Z}(T)_{/\Q}=\emptyset$ for $\det(T)\not=0$.
Passing to the generic fiber in (\ref{moduli decomp})   yields an isomorphism of stacks
$$
\bigsqcup_{ \substack{ T\in\Sigma(\alpha) \\ \det(T)=0 }}  \mathcal{Z}(T)_{/\Q} \iso \mathcal{Y}_0(\alpha)_{/\Q},
$$
which, after applying  (\ref{degenerate cycle}) and taking Zariski closures, gives the equality of cycles
$$
\sum_{  \substack{ T\in\Sigma(\alpha) \\ \det(T)=0 } } \mathcal{Z}^\horizontal (t) = \mathcal{C}^\bad 
$$
of codimension one in $\mathcal{M}_0$,  where $t$ is the positive integer defined by (\ref{little t}). 
As in the proof of  \cite[Lemma 7.9.1]{KRY} (\emph{i.e.}~combining \cite[(6.4.2)]{KRY}, \cite[(9.12)]{kudla04a}, 
\cite[Proposition 12.1]{kudla04a}, and \cite[Proposition 9.1]{kudla04a})  and using 
$\mathrm{Tr}(T\mathbf{v}) = \alpha v_1+\alpha^\sigma v_2$,  we have
\begin{eqnarray*}\lefteqn{
\widehat{\deg}\, \widehat{\mathcal{Z}}(T,\mathbf{v})   =
  -h_{\widehat{\omega}_0}(\mathcal{Z}^\horizontal(t)) 
  -h_{\widehat{\omega}_0}(\mathcal{Z}^\vertical(t)) }   \\ 
  & & +   \frac{1}{2}  \deg_\Q(\mathcal{Z}(t)) \cdot  
    \left[  \log \left(  \frac{\alpha v_1 +\alpha^\sigma v_2 } {t v_1 v_2 d_F \mathrm{disc}(B_0)}\right)  
     - J(4\pi\alpha v_1 + 4\pi \alpha^\sigma v_2) \right] 
\end{eqnarray*}
for every singular $T\in \Sigma(\alpha)$.     From this we deduce
\begin{eqnarray*}\lefteqn{
\sum_{  \substack{ T\in\Sigma(\alpha) \\ \det(T)=0 } }   \widehat{\deg}\, \widehat{\mathcal{Z}}(T,\mathbf{v})  
=    -h_{\widehat{\omega}_0}(\mathcal{C}^\bad)  -  
 \sum_{  \substack{ T\in\Sigma(\alpha) \\ \det(T)=0 } }     h_{\widehat{\omega}_0}(\mathcal{Z}^\vertical(t))      }  \\ 
& &    +   \frac{1}{2}  \deg_\Q(\mathcal{C}^\bad) \cdot   b(\alpha,v)   +   \frac{1}{2}  \sum_{  \substack{ T\in\Sigma(\alpha) \\ 
\det(T)=0 } }   \deg_\Q(\mathcal{Z}(t)) \cdot \log \left(  \frac{ 4 \alpha\alpha^\sigma } { t  }\right)  .
\end{eqnarray*}

 Recall from \S \ref{s:generic fiber} that $V$ is the $F$-vector space of trace zero elements of $B$, and $V_0$ is the 
 $\Q$-vector space of trace zero elements of $B_0$.  Let $L_0\subset V_0$ be the $\Z$-submodule of trace zero elements of 
 $\co_{B_0}$, with $\Gamma_0$ acting on $L_0$ by conjugation.  For each $T\in\mathrm{Sym}_2(\Z)^\vee$ let 
 $L_0(T)\subset L_0\times L_0$ be the subset of pairs $(s_1,s_2)$ satisfying (\ref{quadratic matrix}), where 
 $[s_1,s_2]=-\mathrm{Tr}(s_1s_2)$.  For each nonsingular $T\in\Sigma(\alpha)$ with $\mathrm{Diff}(T,B_0)=\{\infty\}$ 
 (in the sense of \cite[\S 3.6]{KRY})  and $(s_1,s_2)\in L_0(T)$,   define
 $$
 \tau= s_1\varpi_2+s_2\varpi_2 \in V\iso V_0\otimes_{\Q}F,
 $$
 and let $(\tau_1,\tau_2)$ be the image of $\tau$ under
 $$
 V\otimes_\Q\R\iso V_0\otimes_\Q\R \times V_0\otimes_\Q\R.
 $$
 Then, by \cite[\S 6.3]{KRY},
$$
\widehat{\deg}\,   \widehat{\mathcal{Z}}(T,\mathbf{v}) 
 =  \sum_{ (s_1,s_2)\in\Gamma_0\backslash L_0(T)  } \frac{1}{2\cdot e_{s_1,s_2}} 
  \int_{X_0} \xi_0(v_1^{1/2}\tau_1) * \xi_0(v_2^{1/2}\tau_2),
$$
where we have abbreviated 
$
e_{s_1,s_2} = | \mathrm{Stab}_{\Gamma_0}(s_1) \cap \mathrm{Stab}_{\Gamma_0}(s_2) | .
$
Using the bijection 
$$
\bigsqcup_{  \substack{T\in\Sigma(\alpha) \\ \det(T)\not=0 } }  L_0(T) \map{} \{  \tau\in L^\nonsing : Q(\tau)=\alpha  \}
$$
defined by $(s_1,s_2)\mapsto s_1\varpi_1+s_2\varpi_2$ and the fact that $L(T)=\emptyset$ unless 
$\mathrm{Diff}(T,B_0)=\{\infty\}$, we obtain
$$
\sum_{ \substack{ T\in \Sigma(\alpha) \\ \det(T)\not=0  \\  \mathrm{Diff}(T,B_0)=\infty }}
\widehat{\deg}\,   \widehat{\mathcal{Z}}(T,\mathbf{v})  
=  \sum_{  \substack{   \tau\in \Gamma_0 \backslash  L^\nonsing \\  Q(\tau)=\alpha }  }   
 \frac{1}{2\cdot |\mathrm{Stab}_{\Gamma_0}(\tau)  |  }     \int_{ X_0} \xi_0(v_1^{1/2}\tau_1) * \xi_0(v_2^{1/2}\tau_2).
 $$

Using Proposition \ref{Prop:adjunction application} and the above formulas, we are reduced to verifying
  \begin{eqnarray}\lefteqn{ \nonumber
    \sum_{  \substack{  T\in \Sigma(\alpha)  \\  \det(T)=0}  }
     h_{\widehat{\omega}_0}(\mathcal{Z}^\vertical(t)  )    + 
       \sum_{p \mathrm{\ prime}}   \big[ I_p( \mathcal{C}^\good,\mathcal{M}_0) + 
       I_p(\mathcal{C}_p^\vertical , \mathcal{M}_0) \big]\cdot \log(p) } \\ 
 & & =  
   \sum_{ p < \infty }   \sum_{  \substack{ T\in \Sigma(\alpha)   \\  \det(T)\not=0 \\ \mathrm{Diff}(B_0,T) = \{p\}  }  }  
   \widehat{\deg} \  \widehat{\mathcal{Z}}(T,\mathbf{v}) +
          \frac{1}{2}   \sum_{ \substack{  T\in \Sigma(\alpha)  \\ \det(T)=0 }}  
           \deg_\Q (\mathcal{Z}(t))  \cdot \log\left(  \frac{ 4 \alpha\alpha^\sigma }{t }  \right) . \label{last step}
\end{eqnarray}
For a prime $p$ that does not divide $\mathrm{disc}(B_0)$ and a nonsingular 
$T\in\Sigma(\alpha)$ with $\mathrm{Diff}(B_0,T) =\{ p\}$,  the arithmetic cycle class 
$\widehat{\mathcal{Z}}(T,\mathbf{v})$ is studied in \cite[\S 6.1]{KRY}. Comparing with 
Proposition \ref{Prop:main unramified} gives
\begin{eqnarray*} \lefteqn{
\big[  I_p( \mathcal{C}_p^\good, \mathcal{M}_0 )  + I_p(  \mathcal{C}_p^\vertical, \mathcal{M}_0 )  \big]  \cdot\log(p)  }   \\
 & = &
 \sum_{  \substack{T \in \Sigma(\alpha)   \\ \det(T)\not=0 \\ \mathrm{Diff}(T,B_0)=\{p\}    }  }
\widehat{\deg}\,  \widehat{\mathcal{Z}}(T,\mathbf{v})  +
 \frac{ 1 }{2} \sum_{  \substack{ T\in\Sigma(\alpha)  \\   \det(T)=0  }  }
\deg_\Q(\mathcal{Z}(t))  \cdot  \ord_p\left(\frac{ 4 \alpha\alpha^\sigma}{t}\right) \log(p)  .
  \end{eqnarray*}
For a prime $p$ dividing $\mathrm{disc}(B_0)$ and a nonsingular $T\in\Sigma(\alpha)$ with 
$\mathrm{Diff}(B_0,T)=\{p\}$,  the arithmetic cycle class $\widehat{\mathcal{Z}}(T,\mathbf{v})$ is studied  in 
\cite[\S 6.2]{KRY}.  Comparing with Proposition \ref{Prop:main ramified} gives
\begin{eqnarray*} \lefteqn{
\big[  I_p( \mathcal{C}_p^\good, \mathcal{M}_0 )  + I_p(  \mathcal{C}_p^\vertical, \mathcal{M}_0 )  \big]  
\cdot\log(p)   +   \sum_{  \substack{  T\in \Sigma(\alpha)  \\  \det(T)=0}  }
     h_{\widehat{\omega}_0}(\mathcal{Z}^\vertical(t)_p  )   } \\
 &  &
 = \sum_{  \substack{T \in \Sigma(\alpha)   \\ \det(T)\not=0 \\ \mathrm{Diff}(T,B_0)=\{p\}    }  }
\widehat{\deg}\,  \widehat{\mathcal{Z}}(T,\mathbf{v})  +   \frac{ 1 }{2} \sum_{  \substack{ T\in\Sigma(\alpha)  \\   \det(T)=0  }  }
\deg_\Q(\mathcal{Z}(t))  \cdot  \ord_p\left(\frac{ 4 \alpha\alpha^\sigma}{t}\right) \log(p)  .
  \end{eqnarray*}
The above two formulas prove (\ref{last step}), and complete the proof of the theorem.
\end{proof}

\begin{Cor}
Suppose that $\alpha\in\co_F$ and $v\in F\otimes_\Q\R$ are both totally positive. 
If either $F(\sqrt{-\alpha})$ is not biquadratic, or if $2$ splits in $F$ and  $\alpha\co_F$ is
 relatively prime to the different of $F/\Q$,  then 
$$
\widehat{\mathrm{deg}}_{\mathcal{M}_0} \widehat{\mathcal{Y}}(\alpha,v) = c(\alpha,v)
$$
where $c(\alpha,v)$ is the Fourier coefficient appearing in (\ref{modular coefficients}).
\end{Cor}

\begin{proof}
By \cite[Lemma 5.2.1]{howardA} 
$$
c(\alpha,v) = \sum_{ T\in \Sigma( \alpha) }\widehat{\deg}\, \widehat{\mathcal{Z}}(T,\mathbf{v})
$$
and so the proof is immediate from Theorem \ref{Thm:main result}.
\end{proof}

\bibliographystyle{plain}
\def\cprime{$'$}

\end{document}